\documentclass[11pt]{amsart}

\usepackage{mathtools}
\usepackage{amsmath, amsthm, amssymb, amsfonts, enumerate}
\usepackage{enumitem}
\usepackage{dsfont}
\usepackage{color}
\usepackage{geometry}
\usepackage{todonotes}
\usepackage{graphicx}
\usepackage{caption}
\usepackage{float}
\usepackage{soul}

\usepackage[unicode,colorlinks=true,linkcolor=blue,urlcolor=blue, citecolor=blue]{hyperref}
\usepackage[euler]{textgreek}

\mathtoolsset{showonlyrefs}

\def \cA{\mathcal{A}}
\def \cB{\mathcal{B}}
\def \cC{\mathcal{C}}
\def \cF{\mathcal{F}}
\def \cH{\mathcal{H}}
\def \cI{\mathcal{I}}
\def \cJ{\mathcal{J}}
\def \cK{\mathcal{K}}
\def \cL{\mathcal{L}}
\def \cO{\mathcal{O}}
\def \cT{\mathcal{T}}
\def \P{\mathsf P}
\def \Q{\mathsf Q}
\def \E{\mathsf E}
\def \N{\mathbb{N}}
\def \R{\mathbb{R}}
\def \F{\mathbb{F}}
\def \ud{\mathrm{d}}
\def \e{\mathrm{e}}

\newcommand{\eps}{\varepsilon}

\newtheorem{theorem}{Theorem}[section]
\newtheorem{lemma}[theorem]{Lemma}
\newtheorem{corollary}[theorem]{Corollary}
\newtheorem{proposition}[theorem]{Proposition}

\newtheorem{remark}[theorem]{Remark}
\newtheorem{assumption}[theorem]{Assumption}

\theoremstyle{definition}
\newtheorem{problem}{Problem}

\DeclareMathOperator*{\sign}{sign}

\makeatletter
\@namedef{subjclassname@2020}{\textup{2020}Mathematics Subject Classification}
\makeatother

\title[Stopper vs.\ singular-controller games on the half-line]{Finite-time horizon,\\ stopper vs.\ singular-controller games on the half-line}
\author[Bovo]{Andrea Bovo}
\author[De Angelis]{Tiziano De Angelis}
\address{A.\ Bovo: School of Management and Economics, Dept.\ ESOMAS, University of Torino, Corso Unione Sovietica, 218 Bis, 10134, Torino, Italy.}
\email{\href{mailto:andrea.bovo@unito.it}{andrea.bovo@unito.it}}
\address{T.\ De Angelis: School of Management and Economics, Dept.\ ESOMAS, University of Torino, Corso Unione Sovietica, 218 Bis, 10134, Torino, Italy; Collegio Carlo Alberto, Piazza Arbarello 8, 10122, Torino, Italy.}
\email{\href{mailto:tiziano.deangelis@unito.it}{tiziano.deangelis@unito.it}}
\date{\today}

\numberwithin{equation}{section}

\begin{document}

\begin{abstract}
We prove existence of a value for two-player zero-sum stopper vs.\ singular-controller games on finite-time horizon, when the underlying dynamics is one-dimensional, diffusive and bound to evolve in $[0,\infty)$. We show that the value is the maximal solution of a variational inequality with both obstacle and gradient constraint and satisfying a Dirichlet boundary condition at $[0,T)\times\{0\}$. Moreover, we obtain an optimal strategy for the stopper.
In order to achieve our goals, we rely on new probabilistic methods, yielding gradient bounds and equi-continuity for the solutions of penalised partial differential equations that approximate the variational inequality.
\end{abstract}

\maketitle

\section{Introduction}

Two-player zero-sum stochastic games of singular-controller vs.\ stopper type have been studied in various formulations when the underlying diffusive dynamics evolves in $\R^d$ (cf.\ \cite{bovo2022variational,bovo2023,bovo2023b}). Motivated by applications that we are going to discuss in more detail below, in this paper we consider games on finite-time horizon in which the underlying dynamics is one-dimensional and it is bound to stay positive (in the sense that the process is absorbed if it falls below zero and the game ends). The state space of our game is $[0,T]\times[0,\infty)$. We prove that the game admits a value and we produce an optimal stopping rule for the stopper, which is optimal against {\em any} choice of the control. The value is the maximal solution in a Sobolev class of a variational inequality with obstacle and gradient constraint, {\em and} with an additional Dirichlet boundary condition at $[0,T]\times\{0\}$. These results pave the way to the analysis of a saddle point in the game which is obtained in the companion paper \cite{bovo2023saddlearxiv} via free boundary methods.

Notice that singular controls may induce jumps in the controlled dynamics. However, we specify the class of admissible controls so as to guarantee that the controlled dynamics does not jump strictly below zero. This is consistent with economic/financial applications discussed later on. We also emphasise that the choice of domain $[0,T]\times[0,\infty)$ is arbitrary and we may equivalently consider domains of the form $[0,T]\times[a,\infty)$ for some $a\in\R$, depending on the specific application of the model. All of our results and methods continue to apply as shown in Remark \ref{rem:a-abs}.

Although the statement of the main result is similar to the one in \cite{bovo2022variational}, the path to its proof is based on entirely new estimates which illustrate a methodology of independent interest. There are two key technical steps in our analysis: 

(i) We construct sufficiently regular solutions to the so-called penalised problems, that approximate the general variational problem. Such construction requires a priori bounds on the gradient of the solution, which are particularly challenging near the boundary $[0,T)\times\{0\}$ of the domain. As we will explain below, existing techniques do not apply to our game setup and therefore we develop a new, fully probabilistic approach that shows interesting connections to Fleming logarithmic transform\footnote{In the PDE literature this is better known as Cole-Hopf transform, after the papers \cite{cole1951quasi,hopf1950partial} (cf.\ also historical remarks in \cite[Ch.\ VI]{fleming2006controlled}).} for linear-quadratic stochastic control problems (cf.\ \cite[Chapters\ III and VI]{fleming2006controlled}).

(ii) We prove compactness of the family of solutions for the penalised problems, in order to pass to the limit along a sequence thereof and to obtain a solution of the general variational problem. Compactness in a suitable Sobolev class boils down to a uniform bound for the norm of the solution of the penalised problem. This can be done partly relying on ideas borrowed from \cite{bovo2022variational}. However, those ideas do not provide equi-continuity of the family near the boundary $[0,T]\times\{0\}$ and they are therefore insufficient to close our argument. Our second technical contribution is indeed an equi-continuity estimate uniformly with respect to the penalisation parameters. This estimate is entirely probabilistic and it leverages on a change of probability measure and a time-change.
\medskip

It is customary in the literature to approximate variational problems with hard constraints via families of semi-linear partial differential equations (PDEs) that in some sense ``soften'' the constraints. This is known as {\em penalisation method}. The resulting family of PDEs is parametrised by one or more indexes that will eventually tend to zero. The difficulties with this approach are to obtain existence/uniqueness of solutions for the penalised problem and bounds thereof in suitable norms, uniformly with respect to the penalisation parameters. When considered separately, optimal stopping problems and singular control problems have been extensively studied with variational inequalities following the penalisation method. Optimal stopping is linked to so-called obstacle problems, whereas singular control is linked to variational inequalities with gradient constraints.
On the one hand, obstacle problems on domains have been studied in depth, with seminal contributions collected in the monograph \cite{bensoussan2011applications}. On the other hand, variational problems with gradient constraint on a domain have also received significant attention (see, e.g., \cite{evans1979second}, \cite{hynd2013analysis}, \cite{ishii1983regularity}, \cite{kelbert2019hjb}, \cite{wiegner1981character}). In some instances the methods from those papers cannot be easily transferred to concrete stochastic control problems with absorbing boundary. In particular, a classical problem in mathematical finance, involving a variational inequality with gradient constraint on a domain, is the celebrated {\em dividend problem}. The finite-time horizon version of the problem, which shares the same domain as in our paper, was solved with great technical effort in \cite{grandits2013optimal}, \cite{grandits2014existence}, \cite{grandits2015optimal} with PDE methods. Later on, \cite{de2017dividend} found a shorter probabilistic approach, which however does not apply to our game setting.

Stochastic games on domains and their associated variational problems have been studied to a lesser extent. For stopping games (so-called Dynkin games), we mention the pioneering work \cite{friedman1973stochastic}, where PDE methods combined with some probabilistic considerations are used. Instead, a controller-stopper game (with classical controls) on a domain was considered more recently in \cite{bayraktar2011regularity}, with methods relying upon dynamic programming and viscosity theory. 

New technical difficulties arise when one tries to combine optimal stopping with singular control, because certain estimates from the PDE literature for the penalised problems fail (e.g., additional terms appear that are not easily controlled or, more simply, some inequalities are reversed and key arguments of proof break down). These hurdles are {\em structural} in zero-sum games of singular control and stopping, because the penalising terms in the approximating PDEs have opposite signs (cf.\ Section \ref{sec:pen-p}). In the interior of the domain $[0,T)\times(0,\infty)$ we can adopt the new estimates developed in \cite{bovo2022variational}, but those are insufficient to control the regularity of the solutions to the penalised problems near the boundary $[0,T]\times\{0\}$ of the domain. Moreover, all bounds need to be uniform with respect to the penalisation parameters, in order to be able to pass to the limit as those parameters vanish. In this paper we devise ways to deal with all of the new technical difficulties.

In order to clarify our technical contribution we summarise the main ideas of (i) and (ii) listed above. Concerning (i): in order to obtain a lower bound on the gradient of the solution to the penalised problem near $[0,T)\times\{0\}$, we first use a lower bound on the Hamiltonian of the associated control problem (cf.\ \eqref{eq:lowbnHepsm}); that introduces an auxiliary linear-quadratic problem (cf.\ \eqref{eq:defw}) whose value is a proxy for our value function; then we use Fleming transform to reduce the Hamilton-Jacobi-Bellman equation for the linear-quadratic problem to a linear PDE; finally, we provide uniform bounds on the solution of such PDE. It is important to notice that we cannot obtain the latter bounds directly for the original domain $[0,T]\times[0,\infty)$, because of a lack of compatibility conditions for the parabolic operator (see proof of Lemma \ref{lem:lbw} for details). Instead we also need to approximate the boundary $[0,T)\times\{0\}$ with a suitable time-dependent (smooth) boundary. We must also ensure that we control the dependence of all bounds with respect to this additional layer of approximation. Concerning (ii): in order to prove equi-continuity of the family of solutions to the penalised problem, we must develop several estimates which we summarise in the statements of Theorem \ref{thm:equic_m} and Proposition \ref{prop:equiuepsmuN}. Those are needed in order to proceed with the limits in Section \ref{sec:final} that yield the value of the game. The equi-continuity estimates are delicate because we must keep track of how the constants depend on the penalisation parameters. In the stochastic control literature we could not find any similar estimates.

In summary, we believe that the combination of techniques used in the paper is novel and it adds to our understanding of the interplay between probabilistic and PDE methods for the study of stochastic games. Indeed, all the key estimates are performed with probabilistic tools rather than analytical ones.

\subsection{Possible applications} 
The paper is motivated by some applications of stochastic control and stochastic games in finance and economics that assign to the underlying dynamics the meaning of asset prices, portfolio values, interest rates, etc. In all those cases, it is natural to assume that the underlying dynamics should evolve in a subspace of $\R^d$. It turns out that stopper vs.\ singular-controller stochastic games on a domain are much more challenging than their analogue on $\R^d$ and they have not been studied yet.
We consider the simplest possible case of a real-valued controlled process which is absorbed upon hitting a fixed threshold. With no loss of generality we set the threshold equal to zero (cf.\ Remark \ref{rem:a-abs}).

We review some applications that are suggested by the existing literature. In particular, we briefly illustrate models for the control of an exchange rate, for optimal investment decisions with random time horizon and for American options with cancellation feature. 

({\em Optimal control of an exchange rate}.) In the Introduction of \cite{hernandez2015zero}, zero-sum stopper vs.\ singular-controller games are proposed for a model in which a central bank controls the exchange rate until the time of a possible political veto. While in \cite{hernandez2015zero} the central bank is able to directly control the exchange rate, in other papers it controls the logarithm of the exchange rate (cf.\ \cite{ferrari2020singular}) or the logarithm of the so-called velocity-adjusted money stock (cf.\ \cite{miller1996optimal}). In \cite{ferrari2020singular} it is assumed that the controlled dynamics evolves on a possibly bounded domain and the optimisation terminates when such dynamics leaves the domain. We share with the setup in \cite{ferrari2020singular} the presence of an absorption time but we also add a finite-time horizon and extend the model to include the game feature. In \cite{miller1996optimal} the log of the velocity-adjusted money stock follows a controlled Brownian motion but there is no absorption of the controlled dynamics. Then, a game with the dynamics from \cite{miller1996optimal} would be more closely related to our earlier works on unbounded domains (cf.\ \cite{bovo2022variational,bovo2023,bovo2023b}).

({\em Optimal investment decision with random time horizon}.) It is natural to assume that there may be uncertainty on the investment horizon of an agent, due to several factors as, e.g., mortality risk or health contingencies that may result in sudden need for liquidity. Investment decision problems of this kind were studied, for example in \cite{blanchet2008investment} and \cite{young2004investment}, where the random time $\tau$ at which the optimisation ends is not a stopping time for the filtration of the underlying asset but it is exogenously specified. Instead, when the agent is able to choose {\em optimally} the time at which the optimisation ends, we enter the realm of so-called optimal investment with discretionary stopping (e.g.,\ \cite{davis1994discretionary}). In that case the time at which the optimisation ends is a stopping time for the agent's filtration. Our framework complements those two strands of the optimal investment literature by considering a worst-case scenario model in which the agent is forced to terminate the optimisation at the {\em least favourable} stopping time. This point of view is considered for example in the context of mortgage valuations in \cite{kitapbayev2024mortgage} but without our game feature (i.e., there is no investment decision). The game that we analyse can be interpreted as a game between an investor and an adversarial nature. Such an interpretation is common in the broader context of Knightian uncertainty and Stackelberg games (see, e.g., \cite{nutz2015nonlinear,riedel2009stopping}). To some extent, our specific problem may be considered as a special case of Knightian uncertainty, where the uncertainty is over the law of the stopped process.

({\em American options with cancellation feature}.) In mathematical finance, we have in mind an extension of the classical framework for cancellable American options (so-called game options; see, e.g., \cite{kifer2000game,kifer2013dynkin,kyprianou2004some,yam2014game}), which also encompasses so-called callable convertible bonds (introduced by \cite{si2004perpetual,sirbu2006two} and widely studied in the literature).
Game options give the issuer of an American option the right to cancel the option at any time, upon the payment of a cancellation fee (similarly, in callable convertible bonds, the firm recalls the bond). Their rational price coincides with the value of a zero-sum Dynkin (stopping) game, as shown in \cite{kifer2000game}. Building on this understanding, our class of games is well-suited to describe contracts which allow gradual cancellation of the option (or gradual recall of the bond) so that, in principle, such contracts would be cheaper for the buyer and less risky for the issuer. 
\medskip

{\bf An example.} We briefly present a model for callable convertible bonds that fits our framework. Following \cite{hobson2024convertible}, let us consider a firm with value described by a geometric Brownian motion 
\[
X^0_t\coloneqq x\exp\Big(\sigma W_t+(\mu-\sigma^2/2)t\Big),
\]
for some parameters $\mu\in\R$, $\sigma>0$ and with initial condition $x>0$, where $(W_t)_{t\ge 0}$ is a Brownian motion. A convertible bond pays coupons to the holder at a constant rate $c>0$ until the time $\tau$ at which the holder decides to convert. In the model by \cite{hobson2024convertible} (cf.\ also \cite{sirbu2006two}), at the conversion time, the holder receives from the issuer a proportion $\gamma\in(0,1]$ of the firm's value. If the convertible bond is ``callable'', the firm may decide to recall the bond at any time $\rho\in[0,T]$ earlier than $\tau$ by paying a fixed penalty $K>0$ to the holder. If the issuer calls the bond exactly at time $\tau\in[0,T]$, various forms of the payoff may be specified, but for simplicity we assume that the holder takes precedence (as in \cite{sirbu2006two}). Then, the rational price of the callable convertible bond is described by the value of a Dynkin game with payoff
\[
\E\Big[\int_0^{\tau\wedge\rho}\e^{-r t}c\,\ud t+ \mathds{1}_{\{\tau\le \rho,\tau<T\}}\e^{-r\tau}\gamma X^0_\tau+\mathds{1}_{\{\rho<\tau\}}\e^{-r\rho} K+\mathds{1}_{\{\rho=\tau=T\}}\e^{-rT} g\big(X^0_T\big)\Big],
\]
where the payoff at the terminal time $T$ may be borrowed for example from \cite{sirbu2006two}. A variant of the classical model, which would fit our framework, can be introduced by assuming that the issuer or the convertible bond is allowed to reduce the proportion of the firm's value that can be converted by the holder. That is, although the firm's value $X^0$ is {\em not} affected by the issuer, the contract stipulates that the holder can only convert an amount $X_t=X^0_t(1-R_t)$ where $R_t\in[0,1]$ is the fraction of firm's value that was gradually recalled by the issuer up to time $t$. Assuming that $(R_t)_{t\ge 0}$ is c\`adl\`ag and non-decreasing, the dynamics of the remaining {\em convertible} amount reads
\[
\ud X_t=\mu X_t\ud t+\sigma X_t\ud W_t-X^0_t\ud R_t,\quad t\in[0,T].
\]
Redefining $\ud \nu_t\coloneqq X^0_t\ud R_t$ we obtain another c\`adl\`ag, non-decreasing control $(\nu_t)_{t\ge 0}$ and the resulting process $X_t=X^\nu_t$ solves an SDE that fits our framework (cf.\ \eqref{eq:prcXcntrll}). Now, on the one hand, if the holder converts at time $\tau$ they receive $\gamma X^\nu_\tau$; on the other hand, if the issuer wants to recall a fraction $R_t=\int_{[0,t]}\frac{\ud \nu_s}{X^0_s}$ of the firm's value they have to pay a proportional cost. Notice that the issuer may recall the bond in full by letting the process $(R_t)_{t\in[0,T]}$ jump to $1$. Thus, we recover the stopping time $\rho$ from the classical formulation as $\rho=\inf\{t\ge 0\,|\,R_t=1\}$. In this example, $\rho$ also coincides with the absorption time for the controlled dynamics, i.e., $\rho=\tau_0\coloneqq\inf\{t\ge 0\,|\,X^\nu_t\le 0\}$. Then, we may formulate the rational price of this contract as the value of a game with payoff
\[
\E\Big[\int_0^{\tau\wedge\tau_0}\e^{-r t}c\,\ud t+\int_{[0,\tau\wedge\tau_0]}\e^{-rt} k\ud \nu_t + \mathds{1}_{\{\tau\le \tau_0,\tau<T\}}\e^{-r\tau}\gamma X^\nu_\tau + \mathds{1}_{\{\tau=T\le \tau_0\}}\e^{-rT} g\big(X^\nu_T\big)\Big],
\]
where $k>0$ is the proportional cost paid by the issuer to recall a fraction $X^0_t\ud R_t$ of the bond. This game fits our framework.

\subsection{Summary of the paper}
The paper is organised as follows: in Section \ref{sec:setting} we set up the problem, formulate our assumptions and state the main result (Theorem \ref{thm:usolvar}) concerning existence of a value, solvability of the variational problem and structure of an optimal stopping strategy. Section \ref{sec:summary} provides an overview of our methods for the proof of Theorem \ref{thm:usolvar}. In Section \ref{sec:penalised-a} we formulate penalised problems on bounded domains, with bounded approximations of the original game's payoff and diffusion parameters. Then we prove the existence of a unique smooth solution to the penalised problem using our new gradient estimates in Lemmas \ref{lem:fntbndr} and \ref{lem:lbw}. In Section \ref{sec:penalised-b} we obtain the solution of the penalised problem on unbounded domain. In that section we obtain one of our main technical estimates for the equi-continuity of the solutions of the penalised problems on bounded domains. Finally, in Section \ref{sec:final} we pass to the limit in the penalisation parameters and smoothing parameters to obtain the value function of our original game and an optimal strategy for the stopper. A short technical appendix concludes the paper.

\section{Setting and main results}\label{sec:setting}
Let $(\Omega,\cF,\P)$ be a probability space equipped with a right-continuous filtration $\F=(\mathcal{F}_s)_{s\in[0,\infty)}$ and with an $\F$-adapted, one-dimensional Brownian motion $(W_s)_{s\in[0,\infty)}$. For $T\in(0,\infty)$, we set $\cO=[0,T)\times(0,\infty)$ and $\overline\cO=[0,T]\times[0,\infty)$. We take continuous functions $h:\overline\cO\to [0,\infty)$ and $g:\overline\cO\to[0,\infty)$ and let $r\ge 0$ and $\bar{\alpha}>0$ be fixed constants. 

We consider a zero-sum controller/stopper game in which the stopper picks a random time $\tau$ and the controller picks a control process $\nu=(\nu_t)_{t\ge 0}$ with paths of bounded variation. First we introduce the game's payoff and then we give precise statements concerning the class of admissible pairs $(\tau,\nu)$. For $(t,x)\in\overline\cO$, the game's {\em expected} payoff is defined as
\begin{align*}
\begin{aligned}
\cJ_{t,x}(\nu,\tau)&\coloneqq \E_{x}\Big[\e^{-r(\tau\wedge\tau_0)}g\big(t\!+\!\tau\wedge\tau_0,X^{\nu}_{\tau\wedge\tau_0}\big)+\!\int_0^{\tau\wedge\tau_0}\!\! \e^{-rs}h(t\!+\!s,X_s^{\nu})\,\ud s+\!\int_{[0,\tau\wedge\tau_0]}\!\! \e^{-rs}\bar{\alpha}\,\ud |\nu|_s \Big],
\end{aligned}
\end{align*}
where $X^\nu=(X^\nu_s)_{s\ge 0}$ denotes a stochastic dynamics controlled with $\nu$, $|\nu|_s$ is the total variation of the process $\nu$ on the interval $[0,s]$ and 
\begin{align}\label{eq:tau0}
\tau_0\coloneqq\tau_0(t,x;\nu)=\inf\big\{s\ge 0\,\big|\, X^\nu_s\le 0\big\}\wedge (T-t). 
\end{align}
The class of admissible times $\tau$ reads 
$\cT_t:=\big\{\tau\,|\,\text{$\tau$ is $\F$-stopping time with $\tau\in[0,T-t]$, $\P$-a.s.}\big\}$.

Given a c\`adl\`ag process of bounded variation $\nu=(\nu_s)_{s\ge 0}$, its Jordan decomposition reads $\nu=\nu^+-\nu^-$ for c\`adl\`ag, non-decreasing processes $\nu^\pm=(\nu^\pm_s)_{s\ge 0}$. Jumps of the process $\nu$ are denoted $\Delta \nu_t=\nu_t-\nu_{t-}$ for $t\ge 0$.
The class of admissible controls is a subset of the class of bounded variation processes that reads 
\begin{align}\label{eq:cA}
\qquad\cA_{t,x}\coloneqq \left\{\nu \left|
\begin{array}{l}
\text{$\nu=\nu^+-\nu^-$ where $(\nu_s^+)_{s\ge 0}$ and $(\nu_s^-)_{s\ge 0}$ are $\F$-adapted, real-valued,}\\ [+3pt]
\text{non-decreasing, c\`adl\`ag, with $\P(\nu_{0-}^\pm=0)=1$, $\E[|\nu_{T-t}^\pm|^2]<\infty$,}\\ [+3pt]
\text{such that $X_{\tau_0}^\nu=0$, $\P$-a.s.\ on $\{\tau_0<T-t\}$ and $\nu_s^\pm-\nu_{\tau_0}^\pm=0$ for $s\ge \tau_0$}
\end{array}
\right. \right\}.
\end{align}

For $\nu \in\cA_{t,x}$ the controlled dynamics is given by
\begin{align}\label{eq:prcXcntrll}
X_s^{\nu}=x+\int_0^s \mu(X_u^{\nu})\ud u + \int_0^s \sigma(X_u^{\nu})\ud W_u +\nu_s, \quad 0\leq s\leq T-t,
\end{align}
where $\mu:\R\to\R$ and $\sigma:\R\to [0,\infty)$ are continuous functions. We will make precise assumptions on $\mu$ and $\sigma$ in Assumption \ref{ass:gen1}. Those assumptions guarantee existence of a unique $\F$-adapted solution of \eqref{eq:prcXcntrll} by, e.g., \cite[Thm.\ 2.5.7]{krylov1980controlled}. When necessary, the dynamics $X^\nu$ will be extended from $[0,T-t]$ to $[0,\infty)$ simply by taking $\nu_s\equiv \nu_{T-t}$ for $s\ge T-t$. Finally, the notation $\E_x$ in the definition of $\cJ_{t,x}(\tau,\nu)$ refers to the expectation under the measure $\P_x(\,\cdot\,)=\P(\,\cdot\,|X^{\nu}_{0-}=x)$, where $X^{\nu}_{0-}$ is the state of the dynamics before a possible jump of the control at time zero. 

The {\em lower} and {\em upper} value of the game are defined respectively as 
\begin{align}\label{eq:lowuppvfnc}
\underline{v}(t,x)\coloneqq\adjustlimits\sup_{\tau\in \mathcal{T}_t}\inf_{\nu\in \cA_{t,x}} \cJ_{t,x}(\nu,\tau)\quad\text{and}\quad
\overline{v}(t,x)\coloneqq\adjustlimits\inf_{\nu \in \cA_{t,x}}\sup_{\tau\in \mathcal{T}_t} \cJ_{t,x}(\nu,\tau),
\end{align}
so that $\underline{v}(t,x) \leq \overline{v}(t,x)$. If equality holds then we say that the game admits a value 
\begin{align}\label{eq:valfunc}
v(t,x)\coloneqq \underline{v}(t,x) = \overline{v}(t,x).
\end{align}

\begin{remark}\label{rem:a-abs}
We can replace the domain $[0,T)\times(0,\infty)$ by $[0,T)\times(a,\infty)$ for some $a\in\R$ without altering our analysis. Indeed, in the latter case we replace $\tau_0$ by $\tau_a=\inf\{s\ge 0\,|\, X^\nu_s\le a\}$ and functions $g$ and $h$ must be defined as mappings $[0,T]\times[a,\infty)\to [0,\infty)$. Then, a simple translation of the dynamics yields the exact same setup as described above: taking $Z^\nu_t\coloneqq X^\nu_t-a$ it is immediate to verify that $\tau_a=\inf\{s\ge 0\,|\, X^\nu_s\le a\}=\inf\{s\ge 0\,|\, Z^\nu_s\le 0\}\eqqcolon\bar\tau_0$ and the shifted dynamics reads
$\ud Z_t=\bar\mu(Z_t)\ud t+\bar\sigma(Z_t)\ud W_t+\ud \nu_t$,
with $\bar \mu(z)\coloneqq\mu(z+a)$ and $\bar \sigma(z)\coloneqq\sigma(z+a)$. Thus, redefining also the cost functions $\bar g(t,z)\coloneqq g(t,z+a)$ and $\bar h(t,z)\coloneqq h(t,z+a)$ we immediately fall back into the setting described above with absorption at $[0,T)\times\{0\}$ for the process $(t,Z^\nu_t)_{t\in[0,T]}$.
\end{remark}

The domain of the game's state-dynamics is $\cO$. As in \cite{bovo2022variational} such domain is unbounded but, differently from \cite{bovo2022variational}, the controlled dynamics may leave the domain by crossing the boundary $[0,T]\times\{0\}$. From the point of view of variational inequalities we must now include boundary conditions at $[0,T]\times\{0\}$ which were not required in the approach of \cite{bovo2022variational}. This leads to a much more complicated analysis of the problem. Next we formulate a variational inequality for the current set-up.

The infinitesimal generator of the uncontrolled time-space process $(t,X^{0})$ is defined on sufficiently smooth functions as
\begin{align}\label{eq:infinitesimal}
\partial_t \varphi(t,x)+(\mathcal{L}\varphi)(t,x)\coloneqq\partial_t \varphi(t,x)+\tfrac{\sigma^2(x)}{2}\partial^2_{xx}\varphi(t,x)+ \mu(x)\partial_x\varphi(t,x), 
\end{align}
where $\partial_x\varphi$ and $\partial^2_{xx}\varphi$ denote the first and second order spatial derivatives of the function $\varphi$, respectively, and $\partial_t \varphi$ its time derivative. In particular we will apply the operator to functions from the Sobolev space $W^{1,2;p}_{\ell oc}(\cO)$, i.e., functions with one time derivative and two space derivatives (in the weak sense) in $L^p(\cK)$ for any compact $\cK\subset\cO$. The variational problem that we associate with the value of the game (if it exists) is the following:
\begin{problem}\label{prb:varineq}
Fix $p>3$. Find a function $u\in W^{1,2;p}_{\ell oc}(\cO)\cap C(\overline\cO)$ such that, letting
\begin{align*}
\mathcal{I}^u\coloneqq \big\{(t,x)\in\cO\big|\:|\partial_x u(t,x)|<\bar{\alpha}\big\}\quad\text{and} \quad\mathcal{C}^u\coloneqq \big\{(t,x)\in\cO\big|\:u(t,x)>g(t,x)\big\},
\end{align*}
$u$ satisfies: 
\begin{align}\label{eq:inipde}
\begin{cases}
(\partial_tu +\mathcal{L}u-ru)(t,x)=-h(t,x), &\qquad \text{{for all}}\ (t,x)\in\mathcal{C}^u\cap\mathcal{I}^u; \\
(\partial_tu +\mathcal{L}u-ru)(t,x)\geq -h(t,x), &\qquad \text{for a.e.}\ (t,x)\in\mathcal{C}^u; \\
(\partial_tu +\mathcal{L}u-ru)(t,x)\leq -h(t,x), &\qquad \text{for a.e.}\ (t,x)\in \mathcal{I}^u; \\
u(t,x)\geq g(t,x),&\qquad \text{for all}\ (t,x)\in \overline\cO;\\
|\partial_x u(t,x)|\leq \bar{\alpha}, &\qquad \text{for all}\ (t,x)\in\cO; \\
u(t,0)=g(t,0), &\qquad \text{for all}\ t\in[0,T];\\
u(T,x)=g(T,x), &\qquad \text{for all}\ x\in(0,\infty),
\end{cases}
\end{align}
with $|u(t,x)|\leq c(1+|x|^2)$ for all $(t,x)\in\overline\cO$ and a suitable $c>0$.\hfill$\blacksquare$
\end{problem}
Differently from \cite[Problem 1]{bovo2022variational}, where the domain of the variational problem is $[0,T]\times\R^d$, the sixth line of \eqref{eq:inipde} provides a boundary condition at $[0,T]\times\{0\}$. The condition $|\partial_x u|\leq \bar{\alpha}$ in the fifth line holds for all $(t,x)\in\cO$ because of the embedding $W^{1,2;p}(\cK)\hookrightarrow C^{0,1;\gamma}(\cK)$ for any compact $\cK\subset \R^2$. The embedding holds for $p>3$ and $\gamma\in (0,1-3/p)$ (see \cite[App.\ E]{fleming2012deterministic}), with $C^{0,1;\gamma}$ the space of $\gamma$-H\"older continuous functions with $\gamma$-H\"older continuous spatial derivative (relative to the parabolic distance). Notice that the two sets $\cI^u$ and $\cC^u$ are open in $\cO$.

Next we give assumptions under which we obtain our main result (Theorem \ref{thm:usolvar}). For any open set $\cI\subset\R$ we use the notation $C^{\gamma}_{\ell oc}(\cI)$ for functions $f:\cI\to \R$ which are $\gamma$-H\"older continuous on any compact contained in $\cI$. We use $C^{1;\gamma}_{\ell oc}(\cI)$ for functions from $C^\gamma_{\ell oc}(\cI)$ with $\gamma$-H\"older continuous first derivative on any compact contained in $\cI$. Similarly, we use $C^{j,k;\gamma}_{\ell oc}(\overline\cO)$ for functions $f:\overline\cO\to \R$ which are $\gamma$-H\"older continuous on any compact $\cK\subset\overline\cO$ with $\gamma$-H\"older continuous $j$-th order time derivative and $k$-th order spatial derivative on $\cK$, for $j,k=\N\cup\{0\}$. We use $C^{j,k;\gamma}(\overline\cO)$ when those properties hold on the whole space $\overline\cO$, i.e., we formally set $\cK=\overline\cO$. Finally, we use $C^{j,k;\gamma}_b(\overline\cO)$ for functions in $C^{j,k;\gamma}(\overline\cO)$ which are bounded along with all their derivatives. For an arbitrary set $B\subset\R$, when a family of functions $(f_\beta)_{\beta\in B}\in W^{1,2;p}_{\ell oc}(\cO)$ is such that for any compact $\cK\subset\cO$ there is a constant $c=c(\cK)>0$ for which $\sup_{\beta\in B}\|f_\beta\|_{W^{1,2;p}(\cK)}\le c$, we simply say that $(f_\beta)_{\beta\in B}$ is bounded in $W^{1,2;p}_{\ell oc}(\cO)$ (and analogously for $C^{j,k:\gamma}_{\ell oc}(\cO)$, $C^{1;\gamma}_{\ell oc}(\cI)$, etc.).

\begin{assumption}[Controlled SDE]\label{ass:gen1}
The coefficients $\mu:\R\to\R$, $\sigma:\R\to[0,\infty)$ are continuous on $\R$, with $\sigma(x)=\sigma(0)$ and $\mu(x)=\mu(0)$ for $x<0$. Moreover
\begin{itemize}
\item[(i)] $\mu \in C^1((0,\infty))$ and Lipschitz on $[0,\infty)$ with constant $D_1$; 
\item[(ii)] $\sigma\in C^1((0,\infty))$ and $|\sigma(x)-\sigma(y)|\le D_\gamma|x-y|^\gamma$ for $x,y\in[0,\infty)$ with $\gamma> 1/2$;
\item[(iii)] $\sigma(x)>0$ for $x\in(0,\infty)$ and
\begin{align}\label{eq:lgc}
|\mu(x)|+\sigma(x)\leq D_1(1+|x|), \text{ for all }x\in[0,\infty).
\end{align}
\end{itemize}
\end{assumption}
\begin{assumption}[Functions $g$ and $h$]\label{ass:gen2}
For the functions $g,h:\overline\cO\to[0,\infty)$, the following hold:
\begin{enumerate}
\item[(i)] $g \in C^{1,2;\gamma}_{\ell oc}(\overline\cO)$ and $h\in C^{0,1;\gamma}_{\ell oc}(\overline\cO)$ for some $\gamma\in(0,1)$, and there is $K_0>0$ such that
\[
g(t,x)-g(s,x)\le K_0(t-s)\, \text{ and }\, h(t,x)-h(s,x)\le K_0 (t-s),\,\text{ for }0\le s\le t\le T,\, x\in[0,\infty);
\]
\item[(ii)] $g$ is such that
$|\partial_x g(t,x)|\leq \bar{\alpha}$, for all $(t,x)\in\overline\cO$;
\item[(iii)] For $(t,x)\in\overline\cO$,
$0\le h(t,x) \leq K_1(1\!+\!|x|^2)$ and $0\le g(t,x) \leq K_1(1\!+\!|x|)$,
for some $K_1\!\in\!(0,\infty)$.
\end{enumerate}
\medskip

Letting 
$\Theta(t,x)\coloneqq h(t,x)+\partial_t g(t,x)+(\cL g)(t,x)-rg(t,x)$, for $(t,x)\in\overline\cO$,
we further assume:
\begin{enumerate}
\item[(iv)] The function $\Theta$ is locally Lipschitz on $\cO$ and there is $K_2\in(0,\infty)$ such that
\begin{align}\label{defn:infTheta}
\Theta(t,x)\ge -K_2,\quad\text{for all $(t,x)\in\overline\cO$};
\end{align}
\item[(v)] We have $\liminf_{x\to\infty}\Theta(T,x)>0$ with $\Theta(T,0)<0$.
\end{enumerate}
\end{assumption}
For future reference we also introduce 
\begin{align}\label{eq:OT}
\overline \Theta\coloneqq\inf\big\{x\in(0,\infty)\,\big|\,\Theta(T,x)=0\big\} 
\end{align}
and we notice that $\overline\Theta\in(0,\infty)$ thanks to (v) in the assumption above.

Next we state our main result.
\begin{theorem}\label{thm:usolvar}
Let Assumptions \ref{ass:gen1} and \ref{ass:gen2} hold. The game described above admits a value $v$ (i.e., \eqref{eq:valfunc} holds) and $v$ is the maximal solution to Problem \ref{prb:varineq}, if either {\bf A1} or {\bf A2} below hold:
\begin{itemize}
\item[{\bf A.1}]

\begin{itemize} 
\item[ (i)] There is $\kappa>0$ such that $|\mu(x)|\le \kappa^{-1}$ and $\kappa^{-1}\ge \sigma(x)\ge \kappa$; 
\item[(ii)] $h\in C^{0,1;\gamma}_b(\overline \cO)$ and Lipschitz, and $g\in C^{1,2;\gamma}_b(\overline \cO)$.
\end{itemize}
\item[{\bf A.2}] The maps $x\mapsto g(t,x)$ and $x\mapsto h(t,x)$ are non-decreasing for all $t\in[0,T]$.
\end{itemize}
Moreover, if either {\bf A.1} holds or $x\mapsto \sigma(x)$ is Lipschitz continuous and {\bf A.2} holds, we have
\begin{align*}
v(t,x)=\inf_{\nu \in\cA_{t,x}}\cJ_{t,x}\big(\nu,(\tau_*\wedge\sigma_*)(\nu)\big), \quad(t,x)\in\overline\cO,
\end{align*}
where
\begin{align}\label{eq:taustar}
\begin{aligned}
\sigma_*(\nu)&\coloneqq \inf\big\{s\geq 0\,\big|\, v(t+s,X_{s-}^{\nu})=g(t+s,X_{s-}^{\nu})\big\}\wedge(T-t),\\
\tau_*(\nu)&\coloneqq \inf\big\{s\geq 0\,\big|\, v(t+s,X_s^{\nu})=g(t+s,X_s^{\nu})\big\}\wedge(T-t).
\end{aligned}
\end{align}
\end{theorem}

Without loss of generality it is possible to restrict the class of admissible controls to those bounded in expectation. The proof is the same as in \cite[Lem.\ 3.1]{bovo2023} and it is omitted.
\begin{lemma}[{\cite[Lem.\ 3.1]{bovo2023}}]\label{lem:Aopt}
Letting $\cA_{t,x}^{opt}\coloneqq \big\{\nu\in\cA_{t,x}|\E_x\big[|\nu|_{T-t}\big]\le K_5(1+x^2)\big\}$, for suitable $K_5>0$, we have
$\underline{v}(t,x)=\sup_{\tau\in\cT_t}\inf_{\nu\cA_{t,x}^{opt}}\cJ_{t,x}(\nu,\tau)$
for any $(t,x)\in\overline\cO$.
\end{lemma}
Notice that the bound is quadratic in $x$ because of quadratic growth of $h$.

\subsection{A summary of the proof of Theorem \ref{thm:usolvar}}\label{sec:summary}

In order to convey the main ideas in the proof of Theorem \ref{thm:usolvar} we provide an outline of the main steps and arguments. Several layers of approximation for the game are necessary and we review their need and how they link to one another.

In order to use tools developed in \cite{bovo2022variational} we must work with a uniformly non-degenerate controlled diffusion but Assumption \ref{ass:gen1} allows $\sigma(0)=0$. Then, the first layer of approximation introduces new diffusion coefficients $\sigma_\kappa$ that are bounded and separated from zero. At no additional cost, we also replace $\mu$ with a bounded drift $\mu_\kappa$ which will then simplify some equi-continuity estimates for the solution of the penalised problem. In particular, the ratio between $\mu_\kappa$ and $\sigma_\kappa$ is a key ingredient in the estimates using a measure-change and time-change in the proof of Theorem \ref{thm:equic_m}. The same equi-continuity estimates also necessitate of bounded cost functions $g^N$ and $h^N$, which approximate the original functions $g$ and $h$. However, when we approximate $g$ with $g^N$ we need that the latter continues to satisfy condition (ii) in Assumption \ref{ass:gen2}. Therefore, we need to adjust also the cost $\bar \alpha$ and replace it with a suitable $\bar \alpha^N$.

The approach proposed in \cite{bovo2022variational} relies initially on the solution of penalised problems on a sequence of cylinders $[0,T]\times B_m(0)$, where $B_m(0)$ is the ball in $\R^d$ of radius $m$ and centred in zero. Since in the limit as $m\to\infty$ the domains diverge to $[0,T]\times\R^d$, the conditions on the boundary of $B_m(0)$ can be chosen arbitrarily in the penalised problems. In particular, choosing {\em zero} boundary conditions is crucial to obtain a gradient bound for the solution of the penalised problem (cf.\ \cite[Lem.\ 3]{bovo2022variational}). In our current setting, again we want to solve penalised problems on bounded domains $[0,T]\times[0,m]$, but we cannot choose arbitrarily the boundary condition at $[0,T)\times\{0\}$. Indeed, the latter must be set to $g(t,x)$. Although $g$ is smooth, it is not equal to zero in general. This introduces two difficulties: first, the compatibility conditions for parabolic Cauchy-Dirichlet problems fail (cf.\ \eqref{eq:comp} and \cite[Thm.\ 3.7]{friedman2008partial}) and, second, the gradient bounds from \cite{bovo2022variational} fail. To overcome the first issue we deform the boundary $[0,T)\times\{0\}$ by introducing a smooth, time-dependent, lower boundary $\zeta^\kappa_m(t)$, which depends on the approximation of the coefficients $(\mu_\kappa,\sigma_\kappa)$ and on the size of the domain $m$.

Once we have performed the above approximations the key steps are summarised below. We denote $u^{N,\eps,\delta}_{\kappa,m}$ the solution of the penalised problem, where $\delta$ and $\eps$ are the penalisation parameters.
\begin{itemize}
\item[(a)] In Lemma \ref{lem:fntbndr} we obtain the new gradient bound for $u^{N,\eps,\delta}_{\kappa,m}$. This requires a separate technical lemma (Lemma \ref{lem:lbw}). In those two lemmas we combine probabilistic estimates with the Fleming logarithmic transform, as detailed in the Introduction.
\item[(b)] For any compact $\cK\subset[0,T)\times(0,\infty)$ we obtain a bound on the $W^{1,2;p}(\cK)$-norm of $u^{N,\eps,\delta}_{\kappa,m}$, uniformly with respect to $N,\kappa, \eps, \delta,m$ but depending on $\cK$. For that we borrow results from \cite{bovo2022variational}. This bound is not enough to pass to the limit in the penalised problems because we do not control the regularity of $u^{N,\eps,\delta}_{\kappa,m}$ near $[0,T)\times\{0\}$.
\item[(c)] Since $\sigma(0)=0$ is allowed by our assumptions, ultimately we cannot expect more than continuity of the value of the game at $[0,T]\times\{0\}$. Theorem \ref{thm:equic_m} yields equi-continuity of $u^{N,\eps,\delta}_{\kappa,m}$, uniformly in $\delta$ and $m$. The proof is lengthy, because there are several terms that need to be estimated and we need to keep track of the dependence of our estimates on the parameters $N,\kappa, \eps, \delta,m$. The paper contains full details as it may otherwise prove challenging to reproduce the correct bounds.
\item[(d)] Thanks to (b) and (c), we let $m\to\infty$ and $\delta\to 0$ and obtain that the resulting function $u^{N,\eps}_\kappa$ is the value of a classical-controller vs.\ stopper game (Lemma \ref{lem:uepsmuN}). For $u^{N,\eps}_\kappa$ we obtain new equi-continuity estimates, uniformly in $\eps$ (Proposition \ref{prop:equiuepsmuN}). The new equi-continuity and the bound from (b) allow us to let also $\eps\to 0$. The resulting function $u^{N}_\kappa$ is the value of our original game but with bounded payoffs $g^N$, $h^N$ and regularised coefficients $(\mu_\kappa,\sigma_\kappa)$ (Lemma \ref{lem:uNkappagame}). 
\item[(e)] In order to let also $\kappa\to 0$ we need slightly non-standard stability estimates in $L^1$ for the controlled dynamics (Lemma \ref{lem:convkappa}). Then we obtain the value function $u^N$ of our original game but with bounded payoffs (Proposition \ref{prop:uNkappato0}). Finally, we let $N\to \infty$ and conclude the proof of our Theorem \ref{thm:usolvar}. It is worth noticing that in this last group of estimates, we use that in our game it is sub-optimal to select controls such that $\nu^+_{\tau_0} > 0$ with positive probability (cf.\ Lemma \ref{lem:nu^-} and Corollary \ref{cor:nu^-}). 
\end{itemize}

\section{A penalised problem on bounded domain}\label{sec:penalised-a}

In this section we introduce a family of stochastic games on bounded domains and with classes of admissible controls different from the original ones. Such games are an approximation of the game with payoff $\cJ_{t,x}$ from the previous section and we study them via penalised PDEs. 

\subsection{An approximation of the game's payoff}\label{sec:approx}
We start with an approximation of the payoff functions ruled by three parameters: $N$, $\kappa$ and $m$. The first one is used to make the payoff functions and their derivatives bounded; the second one is used to obtain a SDE with diffusion coefficient strictly separated from zero and bounded coefficients; 
the last one refers to a localisation of the problem on bounded domains. Given a function $f:\overline\cO\to [0,\infty)$ and $k\in\N$ we introduce the set $A^f_{k}\coloneqq\{f\le k\}\cap\{0\le x\le k\}$, where $\{f\le c\}\coloneqq\{(t,x)\in\overline\cO:f(t,x)\le c\}$ for any $c\in [0,\infty)$. 

Let $(g^N)_{N\in\N}\!\subset\! C^{1,2;\gamma}_b(\overline\cO)$ and $(h^N)_{N\in\N}\!\subset\! C^{0,1;\gamma}_b(\overline\cO)$ be such that, for $N\!\in\!\N$, $0\!\le\! g^N\!\le\! g$, $0\!\le\! h^N\!\le\! h$,
\begin{align}\label{eq:thetaN}
\begin{aligned}
&0\le g^N(t,x)+h^N(t,x)\leq N, &(t,x)\in \overline\cO;\\
&g^N(t,x)\mathds{1}_{A^g_{N-1}}(t,x)=g(t,x)\mathds{1}_{A^g_{N-1}}(t,x),& (t,x)\in \overline\cO;\\ 
&h^N(t,x)\mathds{1}_{A^h_{N-1}}(t,x)=h(t,x)\mathds{1}_{A^h_{N-1}}(t,x),& (t,x)\in \overline\cO;\\
&|h^N(t,x)-h^N(s,y)|\le L_N(|x-y|+|t-s|),\qquad & t,s\in [0,T]\text{ and } x,y\in [0,\infty);\\
&|g^N(t,x)-g^N(s,y)|\le \bar{\alpha}^{N}(|x-y|+|t-s|),& t,s\in [0,T]\text{ and } x,y\in [0,\infty),
\end{aligned}
\end{align}
where $L_N$ is a suitable constant and $(\bar{\alpha}^{N})_{N\in\N}\subset \R_+$ is such that $\bar{\alpha}^{N}\downarrow \bar{\alpha}$ as $N\to\infty$.

For $\kappa\in(0,1)$, we introduce $\sigma_\kappa,\mu_\kappa\in C([0,\infty))\cap C^1_{\ell oc}((0,\infty))$ defined so that $|\mu_\kappa(x)|\le \kappa^{-1}$ and $\kappa^{-1}\ge \sigma_\kappa(x)\ge \kappa$ for $x\in[0,\infty)$. We also assume $\sigma_\kappa\to \sigma$ and $\mu_\kappa\to \mu$ uniformly on compact subsets of $[0,\infty)$ as $\kappa\to 0$. Then, with no loss of generality $\mu_\kappa$ and $\sigma_\kappa$ continue to enjoy linear growth with a constant that can be chosen equal to $D_1$ (as in \eqref{eq:lgc}) independently of $\kappa$.

Under {\bf A.1} in Theorem \ref{thm:usolvar} these approximations are superfluous. Under {\bf A.2} instead we take non-decreasing $x\mapsto g(t,x)$ and $x\mapsto h(t,x)$ and so we assume $x\mapsto g^N(t,x)$ and $x\mapsto h^N(t,x)$ non-decreasing as well (cf.\ Section \ref{sec:valueunb}).

Let $\cL_\kappa$ be the analogue of $\cL$ in \eqref{eq:infinitesimal} but with $(\mu,\sigma)$ therein replaced by $(\mu_\kappa,\sigma_\kappa)$. We define 
\begin{align*}
\Theta^{N}_{\kappa}(t,x)\coloneqq h^N(t,x)+\partial_t g^N(t,x)+(\cL_\kappa g^N)(t,x)-rg^N(t,x).
\end{align*}
Since $g^N\in C^\infty_b(\overline\cO)$ and $\mu_\kappa,\sigma_\kappa,h^N$ are bounded, there is a constant $C_\Theta=C_\Theta(N,\kappa)>0$ such that 
\begin{align}\label{eq:ThetaNkmbnd}
|\Theta^{N}_{\kappa}(t,x)|\le C_\Theta(N,\kappa),\quad\text{for all $(t,x)\in\overline\cO$.}
\end{align}
Recalling $\overline\Theta$ from \eqref{eq:OT} we set 
\begin{align}\label{eq:overtheta}
\overline{\Theta}^{N}_{\kappa}\coloneqq \inf\big\{x\in[0,\infty)\,\big|\, \Theta^{N}_{\kappa}(T,x)=0\big\}.
\end{align}
For large $N$ we can assume with no loss of generality that $\{T\}\times[0,\overline\Theta+1] \in A^g_{N-1}\cap A^h_{N-1}$. Therefore $|\Theta^{N}_{\kappa}(T,x)-\Theta(T,x)|\le |\sigma_\kappa^2(x)-\sigma^2(x)||\partial_{xx}g(T,x)|+|\mu_\kappa(x)-\mu(x)||\partial_{x}g(T,x)|$ for $x\in[0,\overline\Theta+1]$. Since $\Theta(T,0)<0$, then there must be $\kappa_0>0$ such that $\overline\Theta^{N}_{\kappa}\in(0,\overline \Theta+1)$ for all $\kappa\in(0,\kappa_0)$.

Next we localise the problem. We introduce functions $(\zeta^\kappa_m)_{m\in\N}\subset C^\infty([0,T])$ and $(\xi_{m})_{m\in\N}\subset C^\infty([0,\infty))$ such that for each $m\in\N$:
\begin{itemize}
\item[(i)] $\zeta^\kappa_m$ is non-decreasing with $\zeta^\kappa_m(s)=0$ for $s\in[0,T-\tfrac{1}{m}]$ and $\zeta^\kappa_m(T)=\overline{\Theta}^{N}_{\kappa}$;
\item[(ii)] $0\leq \xi_{m}\leq 1$ on $[0,\infty)$, with $\xi_m=1$ on $[0,m]$ and $\xi_m=0$ on $[m+1,\infty)$;
\item[(iii)] There is $C_0>0$ independent of $m\in\N$ such that
$|\dot{\xi}_m|^2\leq C_0\xi_m$ and $|\ddot \xi_m|\le C_0$ on $\R$.
\end{itemize}
Notice that $\xi_m$ can be constructed as in \cite[App.\ A.1]{bovo2022variational}.
For $m,N\in\N$ and $(t,x)\in\overline\cO$ we consider payoffs
$g_m^N(t,x)\coloneqq \xi_{m-1}(x)g^N(t,x)$ and $h_m^N(t,x)\coloneqq \xi_{m-1}(x)h^N(t,x)$,
and we introduce 
\begin{align}\label{eq:ThetaNkm}
\Theta^{N}_{\kappa,m}(t,x)= h^N_m(t,x)+\partial_t g^N_m(t,x)+(\cL_\kappa g^N_m)(t,x)-rg^N_m(t,x).
\end{align}
An easy calculation and properties of $\xi_m$ allow to show $|\Theta^{N}_{\kappa,m}(t,x)|\le C_\Theta(N,\kappa)$, where the constant may be taken as in \eqref{eq:ThetaNkmbnd} with no loss of generality.

We replace the constant cost of control $\bar{\alpha}^N$ with a state-dependent one, defined as
\begin{align*}
\bar{\alpha}^N_m(t,x)=\Big((\bar{\alpha}^{N})^2+N^2|\dot{\xi}_m(x)|^2+2g^N(t,x)\xi_{m}(x)\dot{\xi}_m(x)\partial_x g^N(t,x)\Big)^{1/2}.
\end{align*}
By regularity of $g^N$ and $\xi_m$, we have that $(\bar{\alpha}^N_m(t,x))^2$ is Lipschitz in $x$ and $\gamma/2$-H\"older in $t$ (uniformly in $m$) and we can assume with no loss of generality that the Lipschitz/H\"older constant is $L_N$ as in \eqref{eq:thetaN}. Clearly $\bar\alpha^N_m(t,x)=\bar \alpha^N$ for $(t,x)\in[0,T]\times([0,m]\cup[m+1,\infty))$. The expression under square root is positive, as shown in \cite[Eq.\ below (17)]{bovo2022variational}, and 
\begin{align}\label{eq:alphaNMbnd}
|\partial_x g_m^N(t,x)|\le \bar{\alpha}^N_m(t,x)\le \Big((\bar{\alpha}^{N})^2+N^2C_0+2N\sqrt{C_0}\alpha^N_0\Big)^{1/2}\eqqcolon \Lambda_N,
\end{align} 
(cf.\ \cite[Eq.\ below (17)]{bovo2022variational}), where the second inequality uses $|\dot \xi|^2\le C_0\xi$.

We will consider domains 
$\cO^\kappa_m\coloneqq\{(t,x)\in\cO:\zeta^\kappa_m(t)<x<m\}$
with parabolic boundary 
\begin{align*}
\partial_P\cO^\kappa_m= \big\{(t,\zeta^\kappa_m(t)):t\in[0,T)\big\}\cup\big([0,T)\times\{m\}\big)\cup\big( \{T\}\times [\overline{\Theta}^{N}_{\kappa},m]\big).
\end{align*}
With no loss of generality we assume $N$, $1/\kappa$ and $m$ sufficiently large so that $\overline\Theta^{N}_{\kappa}=\overline\Theta^{N}_{\kappa,m}<m$, where $\overline\Theta^{N}_{\kappa,m}$ is the analogue of \eqref{eq:overtheta} but with $\Theta^{N}_{\kappa}(T,x)$ replaced by $\Theta^{N}_{\kappa,m}(T,x)$.

We are now ready to introduce a game which approximates in a suitable sense the original one. Let us fix $(\eps,\delta)\in(0,1)^2$. For $(t,x,y)\in\overline\cO\times \R$ we define the Hamiltonian 
\begin{align}\label{eq:hmltn}
H_{m}^{N,\eps}(t,x,y)\coloneqq \sup_{p\in\mathbb{R}}\big\{y p-\psi_\eps\big(|p|^2- (\bar{\alpha}^N_m(t,x))^2\big)\big\},
\end{align} 
where $\psi_\eps\in C^2(\R)$ is a non-negative, convex function such that $\psi_\eps(y)=0$ for $y\leq0$, $\psi_\eps(y)>0$ for $y>0$, $\psi_\eps'\geq0$ and $\psi_\eps(y)=\frac{y-\eps}{\eps}$ for $y\geq 2\eps$.

We consider control classes $\cA^\circ_{t,x}\coloneqq\{\nu\in\cA_{t,x}\,|\,t\mapsto \nu_t\text{ absolutely continuous}\}$ and
\[
\mathcal{T}^\delta_t\coloneqq \{w\,|\,\text{$(w_s)_{s\in[0,\infty)}$ progressively measurable with $0\leq w_s\leq\tfrac{1}{\delta}$, $\forall s\in[0,T-t]$, $\P$-a.s.}\}.
\]
For any $\nu\in\cA^\circ_{t,x}$, let
\begin{align}\label{eq:SDEcntrll}
 X_s^{\nu,\kappa}=x+\int_0^s \big[\mu_\kappa(X_u^{\nu,\kappa})+ \dot{\nu}_u\big]\ud u + \int_0^s \sigma_\kappa(X_u^{\nu,\kappa})\ud W_u, \qquad \text{for $0\leq s\leq T-t$}.
\end{align}
Sometimes we use $X^{\nu,\kappa;x}$ to keep track of the starting point of $X^{\nu,\kappa}$. Moreover, we introduce 
\begin{align}\label{eq:rhom}
\rho_{m}=\rho_{m} (t,x;\nu,\kappa) \coloneqq \inf\big\{s\geq0\,\big|\,X_{s}^{\nu,\kappa;x}\notin (\zeta_m(t+s),m)\big\}\wedge (T-t).
\end{align}

Given $(t,x)\in\overline{\cO^\kappa_m}$ and $(\nu,w)\in\cA^\circ_{t,x}\times\cT^\delta_t$ the expected payoff of the approximating game reads:
\begin{align}\label{eq:Jpen}
\begin{aligned}
\cJ^{N,\kappa,\eps,\delta,m}_{t,x}(\nu,w)= \E_{x}\Big[&R^w_{\rho_m}g_m^N(t\!+\!\rho_m,X_{\rho_m}^{\nu,\kappa})\\
&+\int_0^{\rho_{m}}\!\!R^w_s\big[(h_m^N+w_sg_m^N\!+\! H_{m}^{N,\eps}(\cdot,\dot{\nu}_s))(t\!+\!s,X_s^{\nu,\kappa})\big]\,\ud s\Big],
\end{aligned}
\end{align}
where $R^w_s\coloneqq \exp\!\big( \!-\!\int_0^s \left[r+w_\lambda\right] \ud\lambda\big)$.
The associated upper and lower value read, respectively, 
\begin{align*}
\overline{v}^{N,\eps,\delta}_{\kappa,m} (t,x)=\adjustlimits\inf_{\nu\in \cA^{\circ}_{t,x}}\sup_{w\in \mathcal{T}^\delta_t} \cJ^{N,\kappa,\eps,\delta,m}_{t,x}(\nu,w)\quad \text{and}\quad
\underline{v}^{N,\eps,\delta}_{\kappa,m} (t,x)=\adjustlimits \sup_{w\in \mathcal{T}^\delta_t}\inf_{\nu\in \cA^{\circ}_{t,x}} \cJ^{N,\kappa,\eps,\delta,m}_{t,x}(\nu,w), 
\end{align*}
so that $\underline{v}^{N,\eps,\delta}_{\kappa,m} \leq \overline{v}^{N,\eps,\delta}_{\kappa,m} $. We refer to this game as the penalised game on bounded domain.
\begin{remark}
Since we will let $m\to \infty$ while keeping $N$ and $\kappa$ fixed, it is convenient to simplify our notation by setting $\zeta_m=\zeta^\kappa_m$ and $\cO_m=\cO_m^\kappa$. It is clear that $\cO_m^\kappa\uparrow\cO$ as $m\to\infty$, for any $\kappa\in(0,1)$.
\end{remark}

\subsection{A semi-linear PDE}\label{sec:pen-p}
The value of the penalised problem is obtained as solution to a semi-linear PDE formulated as follows: letting $(y)^+\coloneqq \max\{0,y\}$ for $y\in\R$, the PDE of interest reads
\begin{problem}\label{prb:penprob}
Find $u=u^{N,\eps,\delta}_{\kappa,m}$ in $C^{1,2;\gamma}(\overline{\cO}_m)$, for $\gamma\in(0,1)$ as in Assumption \ref{ass:gen2}, that solves:
\begin{align}\label{eq:penprob}
\begin{cases}\partial_tu+\cL_\kappa u-ru=-h_m^N-\frac{1}{\delta}\big(g_m^N-u\big)^++\psi_\eps\big(|\partial_x u|^2- (\bar{\alpha}^N_m)^2\big), &\text{on } \cO_m, \\
u(t,x)=g_m^N(t,x), & (t,x)\in\partial_P \cO_m.
\end{cases}
\end{align}
\hfill$\blacksquare$
\end{problem}

A standard verification theorem shows that any solution of Problem \ref{prb:penprob} coincides with the value function of the penalised game. We omit the proof and refer to the analogous results from \cite{bovo2022variational}:
\begin{proposition}[{\cite[Prop.\ 1 and Rem.\ 3]{bovo2022variational}}]\label{prp:probrap1}
Let $u^{N,\eps,\delta}_{\kappa,m} $ be a solution of Problem \ref{prb:penprob}. Then 
\begin{align}\label{eq:probrap}
u^{N,\eps,\delta}_{\kappa,m} (t,x)=\overline{v}^{N,\eps,\delta}_{\kappa,m} (t,x)=\underline{v}^{N,\eps,\delta}_{\kappa,m} (t,x),\quad\text{for all $(t,x)\in\overline{\cO}_m$}.
\end{align}

For any $\nu\in\cA^\circ_{t,x}$ the process 
\begin{align*}
w^*_s=w^*_s(\nu)\coloneqq \tfrac1\delta \mathds{1}_{\big\{u^{N,\eps,\delta}_{\kappa,m}(t+s,X_s^{\nu,\kappa})\leq g^N_m(t+s,X_s^{\nu,\kappa})\big\}}
\end{align*}
is optimal for the maximiser in the sense that $u^{N,\eps,\delta}_{\kappa,m}(t,x)=\inf_{\nu\in\cA^\circ_{t,x}}\cJ^{N,\kappa,\eps,\delta,m}_{t,x}(\nu,w^*(\nu))$.
The controlled SDE \eqref{eq:SDEcntrll} admits a unique $\F$-adapted solution $(X^*_s)_{s\in[0,T-t]}$ when $\nu=\nu^*$ with
\begin{align}\label{eq:optcntr}
\dot{\nu}_s^*\coloneqq -2\psi_{\eps}'\Big(\big|\partial_x u^{N,\eps,\delta}_{\kappa,m}(t+s,X_s^*)\big|^2- \big(\bar{\alpha}^N_m(t+s,X_s^*)\big)^2\Big)\partial_x u^{N,\eps,\delta}_{\kappa,m}(t+s,X_s^*),
\end{align}
and $\nu^*\in\cA^\circ_{t,x}$ is optimal for the minimiser in the sense that
$u^{N,\eps,\delta}_{\kappa,m}(t,x)=\sup_{w\in\cT^\delta_t}\cJ^{N,\kappa,\eps,\delta,m}_{t,x}(\nu^*,w)$.
\end{proposition}

Uniqueness of the solution of Problem \ref{prb:penprob} follows from the probabilistic representation in \eqref{eq:probrap}.
\begin{corollary}\label{cor:unique-pen}
There is at most one solution to Problem \ref{prb:penprob}.
\end{corollary}

The next lemma can be proven as in \cite{bovo2022variational} thanks to \eqref{eq:lgc} and Assumption \ref{ass:gen2}(iii). We omit details. 
\begin{lemma}[{\cite[Lem.\ 1]{bovo2022variational}}]\label{lem:polygrow}
Let $u^{N,\eps,\delta}_{\kappa,m}$ be a solution of Problem \ref{prb:penprob}. Then, there is a constant $K_3>0$ independent of $N,\kappa,\delta,\eps,m$ such that
$0\leq u^{N,\eps,\delta}_{\kappa,m}(t,x)\leq K_3(1+|x|^2)$, for all $(t,x)\in\overline{\cO}_m$.
In particular, for any $m\geq m_0\in \N$ and $(\kappa,\eps,\delta)\in(0,1)^3$ and $N\in\N$
\begin{align*}
\big\|u^{N,\eps,\delta}_{\kappa,m}\big\|_{L^\infty(\overline\cO_{m_0})}\leq K_3(1+|m_0|^2) \eqqcolon M_1(m_0).
\end{align*}
\end{lemma}

Our next goal is an a priori bound for the norm of $\partial_x u^{N,\eps,\delta}_{\kappa,m}$, required to obtain existence of the solution of Problem \ref{prb:penprob}. Differently from \cite{bovo2022variational} our boundary condition $u=g^N_m$ on $\partial_P\cO_m$ in Problem \ref{prb:penprob} is not identically zero. Therefore, we need to replace estimates from \cite[Lemma 3]{bovo2022variational} with new ones. These new estimates leverage on ideas from linear-quadratic control and Fleming logarithmic transform (cf.\ \cite[Ch.\ III.7, Ch.\ VI.5--9]{fleming2006controlled}). This approach has no analogue in \cite{bovo2022variational}. 

For the next results we use properties of the function $H_{m}^{N,\eps}$: for arbitrary $N$, $\eps$, $m$ and $(t,x)\in\overline\cO$, 
\begin{itemize}
	\item[ (i)] $H_{m}^{N,\eps}(t,x,0)=0$ and $H_{m}^{N,\eps}(t,x,y)\ge 0$ for any $y\in\R$ (pick $p=0$ in \eqref{eq:hmltn}); 
	\item[(ii)] Choosing $p=\eps y/2$ in \eqref{eq:hmltn} yields
\begin{align}\label{eq:lowbnHepsm}
H_{m}^{N,\eps}(t,x,y)\geq \tfrac{\eps}{2}|y|^2-\psi_\eps\Big({\tfrac{\eps^2}{4}|y|^2-(\bar{\alpha}^N_m(t,x))^2}\Big)\geq \tfrac{\eps}{2}|y|^2-\psi_\eps\Big({\tfrac{\eps^2}{4}|y|^2}\Big)\geq \tfrac{\eps}{4}|y|^2.
\end{align}
\end{itemize}

\begin{lemma}\label{lem:fntbndr}
Let $u^{N,\eps,\delta}_{\kappa,m}$ be a solution of Problem \ref{prb:penprob}. There is $M_2=M_2(N,\kappa,\eps,m)>0$ such that 
\begin{align}\label{eq:grdb}
\sup_{(t,x)\in \partial_P\cO_m}|\partial_x u^{N,\eps,\delta}_{\kappa,m} (t,x)|\leq M_2.
\end{align}
\end{lemma}
\begin{proof}
For simplicity we denote $u=u^{N,\eps,\delta}_{\kappa,m}$. Let $(s,y)\in\partial_P\cO_m$. If $s=T$, then $u(T,y)=g^N_m(T,y)$ for $y\ge \overline\Theta^{N}_{\kappa}$ and the result is trivial. If $y=m$, then $u(t,m)=g^N_m(t,m)=0$ for $t\in[0,T)$ and a direct application of \cite[Lemma\ 3]{bovo2022variational} leads to the result.
It remains to prove the bound along the boundary $(t,\zeta_m(t))$ for $t\in[0,T)$.
Using that $u(t,\zeta_m(t))=g_m^N(t,\zeta_m(t))$ for $t\in[0,T)$, and setting $y=\zeta_m(t)$, we have
\begin{align}\label{eq:limxtoz}
|u(t,x)-u(t,y)|
\le |u(t,x)-g_m^N(t,x)|+\|\partial_x g_m^N\|_{L^\infty(\overline{\cO}_m)}|x-y|.
\end{align}
Therefore, it is enough to show that $|u(t,x)-g_m^N(t,x)|=O(|x-y|)$ for $x\to y$.

Let $t\in[0,T)$ and fix $x\in (\zeta_m(t),m)$. An upper bound for $u(t,x)-g_m^N(t,x)$ can be obtained as in the proof of \cite[Lemma\ 3]{bovo2022variational}. Thanks to Dynkin's formula, for arbitrary $(\nu,w)\in\cA_{t,x}^\circ\times\cT^\delta_t$ we have
\begin{align}\label{eq:gmNdynkin}
\begin{aligned}
g_m^N(t,x)=\E_x\Big[&R^w_{\rho_m}g_m^N(t\!+\!\rho_m,X_{\rho_m}^{\nu,\kappa})\\
&-\!\int_0^{\rho_{m}}\!\!R^w_s\big((\partial_t\!+\!\cL_\kappa\!-\!r)g_m^N\!-\!w_sg_m^N\!+\!\partial_xg_m^N\dot{\nu}_s\big)(t\!+\!s,X_s^{\nu,\kappa})\,\ud s\Big].
\end{aligned}
\end{align}
Taking $\nu\equiv 0$, recalling $H_{m}^{N,\eps}(t,x,0)=0$ and using \eqref{eq:Jpen} and \eqref{eq:gmNdynkin}, yields 
\begin{align*}
u(t,x)\!-\!g_m^N(t,x)\!\le\! \sup_{w\in\cT^\delta_t}\E_x\Big[\!\int_{0}^{\rho_{m}}\!\!\!R^w_s\Theta^{N}_{\kappa,m}(t\!+\!s,X_s^{0,\kappa})\,\ud s\Big]\!\le\! \E_x\Big[\!\int_{0}^{\rho_{m}}\!\!\!(\Theta^{N}_{\kappa,m})^+(t\!+\!s,X_s^{0,\kappa})\,\ud s\Big]\eqqcolon\pi(t,x).
\end{align*}
Since the boundary of the domain $\cO_m$ is smooth, \cite[Thm.\ 3.7]{friedman2008partial} and a classical verification argument guarantee that $\pi\in C^{1,2;\alpha}(\overline\cO_m)$, for some $\alpha\in(0,1)$, and it solves 
\begin{align*}
\begin{cases}
\big(\partial_t \pi+\cL_\kappa \pi)(t,x)=-(\Theta^{N}_{\kappa,m})^+(t,x), &\text{for $(t,x)\in \cO_m$,}\\
\pi(t,m)=\pi(t,\zeta_m(t))=0,&\text{for $t\in[0,T)$},\\
\pi(T,x)=0,&\text{for $x\in[\overline{\Theta}^{N}_{\kappa},m]$}.
\end{cases}
\end{align*}
Moreover, $\pi(t,x)\ge 0$ for $(t,x)\in\cO_m$. Then, $0\!\le\! \pi(t,x)\!-\!\pi(t,y)\!\le\! L_{N,\kappa,m}|x\!-\!y|$ for $(t,x)\in\cO_m$ and $y\in\{\zeta_m(t),m\}$, for some $L_{N,\kappa,m}>0$. It follows 
\begin{align}\label{eq:limxtozUPP}
u(t,x)-g_m^N(t,x)\le L_{N,\kappa,m}|x-y|.
\end{align}

Next we prove the lower bound. 
Take $w_s\equiv 0$ in \eqref{eq:Jpen}. Using \eqref{eq:gmNdynkin} and \eqref{eq:lowbnHepsm} we obtain
\begin{align*}
&u(t,x)-g_m^N(t,x)\\
&\ge \inf_{\nu\in\cA_{t,x}^\circ}\E_{x}\Big[\int_0^{\rho_{m}}\!\!\e^{-rs}\big(\Theta^{N}_{\kappa,m}(t\!+\!s,X_s^{\nu,\kappa})\!+\!\partial_xg_m^N(t\!+\!s,X_s^{\nu,\kappa})\dot{\nu}_s\!+\! H_{m}^{N,\eps}(t\!+\!s,X_s^{\nu,\kappa},\dot{\nu}_s)\big)\ud s\Big]\\
&\ge \tfrac{\eps }{4}\inf_{\nu\in\cA^\circ_{t,x}}\E_x\Big[\int_{0}^{\rho_{m}}\!\!\e^{-rs}\Big(\tfrac{4}{\eps}\Theta^{N}_{\kappa,m}(t\!+\!s,X_s^{\nu,\kappa})\!+\!\tfrac{4}{\eps}\partial_xg_m^N(t\!+\!s,X_s^{\nu,\kappa})\dot{\nu}_s\!+\!|\dot{\nu}_s|^2\Big)\ud s\Big].
\end{align*}
It is convenient to introduce $\dot{\upsilon}_s\coloneqq\dot{\nu}_s/\sigma_\kappa(X^{\nu,\kappa}_s)$, which is well-defined because $\sigma_\kappa(x)\ge \kappa$. Then $\upsilon_t=\int_0^t\dot \upsilon_s\ud s$ defines an admissible control $\upsilon\in\cA^\circ_{t,x}$. With a slight abuse of notation we relabel $X^{\nu,\kappa}=Y^{\upsilon}$. Optimising over $\nu$ is equivalent to optimising over $\upsilon$, then
\begin{align}\label{eq:preHJB}
\begin{aligned}
&u(t,x)\!-\!g_m^N(t,x)\\
&\ge \tfrac{\eps}{4}\inf_{\upsilon\in\cA^\circ_{t,x}}\E_x\Big[\int_{0}^{\rho_{m}}\!\!\e^{-rs}\Big(\tfrac{4}{\eps}\Theta^{N}_{\kappa,m}(\cdot,\cdot)\!+\!\tfrac{4}{\eps}\partial_xg_m^N(\cdot,\cdot)\sigma_\kappa(\cdot)\dot{\upsilon}_s\!+\!|\sigma_\kappa(\cdot)\dot{\upsilon}_s|^2\Big)(t\!+\!s,Y^{\upsilon}_s)\,\ud s\Big]\\
&\ge \tfrac{\eps\kappa^2}{4}\inf_{\upsilon\in\cA^\circ_{t,x}}\E_x\Big[\int_{0}^{\rho_{m}}\!\!\e^{-rs}\Big(\tfrac{4}{\eps\kappa^2}\Theta^{N}_{\kappa,m}(\cdot,\cdot)+\tfrac{4}{\eps\kappa^2}\partial_xg_m^N(\cdot,\cdot)\sigma_\kappa(\cdot)\dot{\upsilon}_s\!+\!|\dot{\upsilon}_s|^2\Big)(t\!+\!s,Y^{\upsilon}_s)\,\ud s\Big]
\end{aligned}
\end{align}
where the second inequality holds by $|\sigma_\kappa(x)|\ge \kappa$ for all $x\in[0,\infty)$. Applying the inequality $b^2+ab\ge\tfrac{1}{2}b^2-\tfrac{1}{2}a^2$ with $a=\tfrac{4}{\eps\kappa^2}\partial_xg_m^N(t\!+\!s,Y^{\upsilon}_s)\sigma_\kappa(Y^{\upsilon}_s)$ and $b=\dot{\upsilon}_s$ in \eqref{eq:preHJB}, we obtain
\begin{align*}
\begin{aligned}
&u(t,x)\!-\!g_m^N(t,x)\\
&\ge \tfrac{\eps\kappa^2}{8}\inf_{\upsilon\in\cA^\circ_{t,x}}\E_x\Big[\int_{0}^{\rho_{m}}\!\!\e^{-rs}\Big(\tfrac{8}{\eps\kappa^2}\Theta^{N}_{\kappa,m}(t\!+\!s,Y^{\upsilon}_s)\!-\!\big(\tfrac{4}{\eps\kappa^2}\partial_xg_m^N(t\!+\!s,Y^{\upsilon}_s)\sigma_\kappa(Y^{\upsilon}_s)\big)^2\!+\!|\dot{\upsilon}_s|^2\Big)\ud s\Big].
\end{aligned}
\end{align*}
In order to simplify notation, we relabel 
$\e^{-rT}\vartheta(t,x)\coloneqq \frac{8}{\eps\kappa^2}[\Theta^{N}_{\kappa,m}(t,x)-\frac{2}{\eps\kappa^2}(\partial_x g^N_m(t,x)\sigma_\kappa(x))^2]$,
and denote $\vartheta^-=\max\{0,-\vartheta\}$. Then
\begin{align*}
\begin{aligned}
u(t,x)-g_m^N(t,x)
&\ge \tfrac{\eps\kappa^2}{8}\inf_{\upsilon\in\cA^\circ_{t,x}}\E_x\Big[\int_{0}^{\rho_{m}}\!\!\e^{-rs}\Big(|\dot{\upsilon}_s|^2-\e^{-rT}\vartheta^-(t+s,Y^{\upsilon}_s)\Big)\,\ud s\Big]\\
&\ge \tfrac{\eps\kappa^2}{8}\e^{-rT}\inf_{\upsilon\in\cA^\circ_{t,x}}\E_x\Big[\int_{0}^{\rho_{m}}\!\!\Big(|\dot{\upsilon}_s|^2-\vartheta^-(t+s,Y^{\upsilon}_s)\Big)\,\ud s\Big].
\end{aligned}
\end{align*}
Setting
\begin{align}\label{eq:defw}
w(t,x)\coloneqq \inf_{\upsilon\in\cA^\circ_{t,x}}\E_x\Big[\int_{0}^{\rho_{m}}\!\!\Big(|\dot{\upsilon}_s|^2-\vartheta^-(t+s,Y^{\upsilon}_s)\Big)\,\ud s\Big],
\end{align}
we reach $(u-g^N_m)(t,x)\ge (\eps\kappa^2/8)\e^{-rT} w(t,x)$. It is clear that $w(t,y)=0$ for $y\in\{\zeta_m(t),m\}$. We claim that $w(t,x)\ge -c |x-y|$ for $(t,x)\in\cO_m$ and $y\in \{\zeta_m(t),m\}$, and some $c=c(N,\kappa,\eps,m)>0$. Then, combining the latter with \eqref{eq:limxtoz} and \eqref{eq:limxtozUPP} we obtain \eqref{eq:grdb}.

A proof of the bound $w(t,x)\ge -c |x-y|$ is given in Lemma \ref{lem:lbw} for the ease of exposition.
\end{proof}
\begin{lemma}\label{lem:lbw}
Let $w(t,x)$ be defined as in \eqref{eq:defw}. Then, there is $c=c(N,\kappa,\eps,,m)>0$ such that for every $(t,x)\in\cO_m$ and $y\in\{\zeta(t),m\}$ it holds $w(t,x)\ge -c|x-y|$.
\end{lemma}
\begin{proof}
Denoting 
\begin{align}\label{eq:ham}
\cH[f](t,x)\coloneqq \inf_{\xi\in\R}\big(\partial_x f(t,x)\sigma_\kappa(x)\xi+\xi^2\big),
\end{align}
and recalling 
\begin{align*}
\ud Y^\upsilon_s=\big[\mu_\kappa(Y^\upsilon_s)+\sigma_\kappa(Y^\upsilon_s)\dot\upsilon_s\big]\ud s+\sigma_\kappa(Y^\upsilon_s)\ud W_s,\quad Y_0=x,
\end{align*}
the HJB equation associated to $w$ is formally given by 
\begin{align}\label{eq:quadlinprob}
\begin{cases}
\big(\partial_t f\!+\!\cL_\kappa f\!+\!\cH[f]\!-\!\vartheta^-\big)(t,x)=0, &\text{for $(t,x)\in \cO_m$,}\\
f(t,m)=f(t,\zeta_m(t))=0,&\text{for $t\in[0,T)$},\\
f(T,x)=0,&\text{for $x\in[\overline{\Theta}^{N}_{\kappa},m]$}.
\end{cases}
\end{align}
If we find a sufficiently regular solution $f$ of \eqref{eq:quadlinprob}, then by a standard verification argument $f=w$ (cf.\ \cite[Ch.\ III.7]{fleming2006controlled}, see (7.5) therein). It then remains to show that \eqref{eq:quadlinprob} admits a classical solution with bounded derivatives in $\overline\cO_m$.

We first prove the result under the additional assumption that
\begin{align}\label{eq:g_xmn=0}
\partial_x g^N_m(T,\overline{\Theta}^{N}_{\kappa})=0.
\end{align}
The infimum in the Hamiltonian \eqref{eq:ham} is achieved in $\xi=-\tfrac{\sigma_\kappa(x)}{2}\partial_x f(t,x)$, so that $\cH[f](t,x)=-\frac14(\sigma_\kappa(x)\partial_x f(t,x))^2$. Plugging that into \eqref{eq:quadlinprob} we get
\begin{align}\label{eq:quadlinpro2}
\begin{cases}
\big(\partial_tf\!+\!\cL_\kappa f\!-\!\tfrac14(\sigma_\kappa \partial_x f)^2-\!\vartheta^-\big)(t,x)=0, &\text{for $(t,x)\in \cO_m$,}\\
f(t,m)=f(t,\zeta_m(t))=0,&\text{for $t\in[0,T)$},\\
f(T,x)=0,&\text{for $x\in[\overline{\Theta}^{N}_{\kappa},m]$}.
\end{cases}
\end{align}
This type of PDEs can be linearised via a logarithmic transformation (cf.\ \cite[Ch.\ VI]{fleming2006controlled}). Letting $\Psi(t,x)=\exp(-\frac{1}{2}f(t,x))$, we obtain
\begin{align}\label{eq:PDEPsi}
\begin{cases}
\big(\partial_t\Psi\!+\! \cL_\kappa\Psi\!+\!\tfrac{1}{2}\vartheta^-\Psi\big)(t,x)=0, &\text{for $(t,x)\in \cO_m$,}\\
\Psi(t,m)=\Psi(t,\zeta_m(t))=1,&\text{for $t\in[0,T)$},\\
\Psi(T,x)=1,&\text{for $x\in[\overline{\Theta}^{N}_{\kappa},m]$}.
\end{cases}
\end{align}

Under \eqref{eq:g_xmn=0} the compatibility condition at $\{T\}\times\{\overline \Theta^{N}_{\kappa},m\}$ holds, i.e., $\Psi|_{\partial_P\cO_m}\eqqcolon\Gamma=1$ satisfies 
\begin{align}\label{eq:comp}
\lim_{s\to T}(\partial_t \Gamma+\cL_\kappa \Gamma+\frac12\vartheta^-\Gamma)(s,x)=\frac12\vartheta^-(T,x)=0,\quad\text{for $x\in\{\overline\Theta^{N}_{\kappa},m\}$}.
\end{align}
Then \cite[Thm.\ 3.7]{friedman2008partial} yields a unique solution $\Psi\in C^{1,2;\alpha}_b(\overline\cO_m)$ of \eqref{eq:PDEPsi}, for some $\alpha\in(0,1)$. Then, there is $K=K(N,\kappa,\eps,m)>0$ such that
$\|\Psi\|_{L^\infty(\overline{\cO}_m)}+\|\partial_x\Psi\|_{L^\infty(\overline{\cO}_m)}\leq K$, where the dependence on $\eps$ is due to the definition of $\vartheta$.
Finally, $\Psi\ge 1$ can be seen as follows: applying Dynkin's formula, for all $(t,x)\in\overline\cO_m$,
\begin{align*}
\Psi(t,x)=\E_{x}\Big[&\e^{\int_0^{\rho_m}\frac{1}{2}\vartheta^-(t+\lambda,X_\lambda^{0,\kappa})\ud \lambda}\Psi(t\!+\!\rho_m,X_{\rho_m}^{0,\kappa})\!\\
&-\!\int_0^{\rho_m}\!\e^{\int_0^{s}\frac{1}{2}\vartheta^-(t+\lambda,X_\lambda^{0,\kappa})\ud \lambda}\big(\partial_t\Psi\!+\!\cL_\kappa\Psi\!+\!\tfrac{1}{2}\vartheta^-\Psi\big)(t\!+\!s,X_{s}^{0,\kappa})\,\ud s\Big]\\
=\E_{x}\Big[&\e^{\int_0^{\rho_m}\frac{1}{2}\vartheta^-(t+\lambda,X_\lambda^{0,\kappa})\ud \lambda}\Big]\ge 1.
\end{align*}

Taking $f(t,x)=-2\ln(\Psi(t,x))$ it is not hard to verify that $f\le 0$ is the unique classical solution of \eqref{eq:quadlinpro2} (uniqueness follows by uniqueness of the solution to \eqref{eq:PDEPsi}). As explained, a standard verification argument yields $w=f$. Since $w=-2\ln \Psi$ and $\Psi\ge 1$, we also obtain
$\|\partial_x w\|_{L^\infty(\overline{\cO}_m)}\leq 2\|\partial_x \Psi\|_{L^\infty(\overline{\cO}_m)}\le 2K$. This concludes the proof under the additional assumption \eqref{eq:g_xmn=0}.

When condition \eqref{eq:g_xmn=0} is not satisfied, a priori we cannot guarantee the compatibility condition \eqref{eq:comp} at $y=\zeta_m(t)$. Instead, the compatibility condition at $y=m$ always holds because $\partial_x g^N_m(T,m)=0$. In order to restore the compatibility condition, we show in Appendix \ref{app:gtilde} that there is $\tilde{g}_m^N\in C^{1,2}(\overline\cO_m)$ such that: 
\begin{align}\label{eq:tildegmN}
\begin{aligned}
& g_m^N(t,y)=\tilde{g}_m^N(t,y)\quad \text{for all $t\in[0,T]$ and $y\in\{\zeta_m(t),m\}$};\\
& \partial_x\tilde{g}_m^N(T,\overline{\Theta}^{N}_{\kappa})=\partial_x\tilde{g}_m^N(T,m)=0;\\
&\!\big(\partial_t \tilde{g}_m^N+\cL_\kappa \tilde{g}_m^N-r \tilde{g}_m^N\big)(T,y)+h_m^N(T,y)\ge 0,\quad\text{for $y\in\{\overline\Theta^{N}_{\kappa},m\}$}.
\end{aligned}
\end{align}
Since $\tilde{g}_m^N$ agrees with $g_m^N$ on the boundary $(t,\zeta_m(t))$ for $t\in[0,T)$, then \eqref{eq:limxtoz} holds with $g_m^N$ therein replaced by $\tilde{g}_m^N$. The rest of the proof can be repeated verbatim but replacing everywhere $g^N_m$ with $\tilde g^N_m$. In particular, $\Theta^{N}_{\kappa,m}(t,x)$ is replaced by $\widetilde \Theta^{N}_{\kappa,m}(t,x)\coloneqq \big(\partial_t \tilde{g}_m^N+\cL_\kappa \tilde{g}_m^N-r \tilde{g}_m^N\big)(t,x)+h_m^N(t,x)$,
so that also $\vartheta(t,x)$ is replaced by
$\tilde \vartheta(t,x)\coloneqq \e^{rT}\frac{8}{\eps\kappa^2}[\widetilde \Theta^{N}_{\kappa,m}(t,x)-\frac{2}{\eps\kappa^2}(\partial_x \tilde g^N_m(t,x)\sigma_\kappa(x))^2]$.
Then, using the second and third conditions in \eqref{eq:tildegmN} it is easy to verify that \eqref{eq:comp} is restored.
\end{proof}

Lemma \ref{lem:fntbndr} gives us a bound on the gradient of $u^{N,\eps,\delta}_{\kappa,m}$ along the parabolic boundary $\partial_P\cO_m$. Therefore we can repeat line by line the proofs of \cite[Prop.\ 3 and 4]{bovo2022variational} to obtain a bound on the $C^{1,2;\gamma}(\overline{\cO}_m)$-norm of $u^{N,\eps,\delta}_{\kappa,m}$. Since the proofs are identical, we omit further details and refer the interested reader to the original paper.
\begin{proposition}[{\cite[Prop.\ 3 and 4]{bovo2022variational}}]\label{prop:locgradp}
Let $u^{N,\eps,\delta}_{\kappa,m} $ be a solution of Problem \ref{prb:penprob}. Then, there is $M_3=M_3(N,\kappa,\eps,m)$ such that $\|\partial_x u^{N,\eps,\delta}_{\kappa,m}\|_{L^\infty(\overline\cO_m)}\le M_3$.
Moreover, for any $\beta\in(0,1)$ there is $M_4=M_4(N,\kappa,\eps,\delta,m,\beta)$ such that
$\|u^{N,\eps,\delta}_{\kappa,m} \|_{C^{0,1;\beta}(\overline{\cO}_m)}\leq M_4$.
\end{proposition}

Having established the gradient bound on $\overline\cO_m$ we can proceed to prove existence and uniqueness of a solution to Problem \ref{prb:penprob} via Schaefer's fixed point theorem. The proof is given in \cite[Thm.\ 2]{bovo2022variational} and we therefore omit it here.
\begin{theorem}[{\cite[Thm.\ 2]{bovo2022variational}}]\label{thm:exisolpenprb}
There exists a unique solution $u^{N,\eps,\delta}_{\kappa,m}$ of Problem \ref{prb:penprob}.
\end{theorem}

\section{Penalised problem on unbounded domain}\label{sec:penalised-b}

In this section we want to pass to the limit as $m\to\infty$ and show that $u^{N,\eps,\delta}_{\kappa,m}$ converges in a suitable sense to a function $u^{N,\eps,\delta}_{\kappa}$ that solves the penalised PDE on unbounded domain. This requires two main ingredients: (i) a uniform bound for the $W^{1,2;p}_{\ell oc}$-norm of $u^{N,\eps,\delta}_{\kappa,m}$ (away from $[0,T]\times\{0\}$) and (ii) equi-continuity of functions $u^{N,\eps,\delta}_{\kappa,m}$ on $[0,T]\times[0,m]$, uniformly in $m$.

For the $W^{1,2;p}_{\ell oc}$-bounds we can rely on analogous results from \cite{bovo2022variational}: 
\begin{proposition}[{\cite[Prop.\ 5 and 6]{bovo2022variational}}]\label{prop:gradbndUa}
For any closed rectangle $\cK\subset\cO$ there is $m_\cK\in\N$ such that $\cK\subset\mathrm{int}\,\cO_{m}$ for all $m\ge m_\cK$. Let $u^{N,\eps,\delta}_{\kappa,m}$ be the unique solution of Problem \ref{prb:penprob} on $\cO_m$, with $m\ge m_\cK$. 

There is $c_1=c_1(\cK)>0$ independent of $N$, $\kappa$, $\eps$, $\delta$ and $m$, such that $\|\partial_x u^{N,\eps,\delta}_{\kappa,m}\|_{L^\infty(\cK)}\le c_1$.
Moreover, for any $p\in(3,\infty)$, we have
$\|u^{N,\eps,\delta}_{\kappa,m}\|_{W^{1,2;p}(\cK)}+\|u^{N,\eps,\delta}_{\kappa,m}\|_{C^{0,1;\beta}(\cK)}\le c_2$,
with $\beta=1-\tfrac{3}{p}$ and for a constant $c_2=c_2(\cK,\eps,\delta,p)>0$.
\end{proposition}
The statement of the first part of the proposition above is slightly different from the one for \cite[Prop.\ 5]{bovo2022variational}. That is due to the fact that in \cite{bovo2022variational} we could consider closed cylinders in $[0,T]\times\R^d$. Here instead we must stay away from $T$ in order to avoid intersecintg the boundary $\zeta_m(t)$. 

A few comments are in order to clarify why the proof of \cite[Prop.\ 5]{bovo2022variational} can be repeated verbatim. 
It is clear that for any compact $\cK\subset\cO$, there exists $m_\cK\in\N$ such that $\cK\subset\mathrm{int}\cO_m$ for all $m\ge m_\cK$. Then, in the proof of \cite[Prop.\ 5]{bovo2022variational} we must replace the cut-off function $\xi$ with another cut-off function $\xi_\cB$ that is equal to one in $\cK$ and it vanishes on the boundary of an open set $\cB$ such that $\cK\subset\cB\subset \cO_{m_\cK}$. The rest of the proof is unchanged and we can choose sufficiently large $m$ and $N$ so that $g^N_m=g$, $h^N_m=h$ and $\bar{\alpha}^N_m=\bar{\alpha}^N$ on $\cB$ and $\bar{\alpha}^N$ is arbitrarily close to $\bar{\alpha}$. We also notice that, with no loss of generality, $0<\min_{z\in\overline\cB}\sigma(z)\le \sigma_\kappa(x)\le \max_{z\in\overline\cB}[\sigma(z)+1]$ and $\sup_{z\in\overline\cB}|\mu_\kappa(x)|\le 1+\sup_{z\in\overline\cB}|\mu(x)|$, by uniform convergence on compacts of $(\mu_\kappa,\sigma_\kappa)$ to $(\mu,\sigma)$. That explains why the constants $c_1,c_2$ are independent of $N$, $\kappa$, $m$.

\subsection{Equi-continuity uniformly in \texorpdfstring{$m$}{m} and \texorpdfstring{$\delta$}{d}}\label{sec:equi_m}

The second ingredient in our recipe is the equi-continuity of functions $u^{N,\eps,\delta}_{\kappa,m}$ uniformly in $m$. 
The proof of the next theorem requires some technical lemmas, which are stated and proven after the theorem for the ease of exposition. Notice that
\begin{align}\label{eq:gmNDynk}
\begin{aligned}
\E[R^{w}_{\tau}g_m^N(t\!+\!\tau,X_{\tau}^{\nu,\kappa;x})]&=\E\Big[R^w_{\rho\wedge \tau}g_m^N(t\!+\!\rho\wedge \tau,X_{\rho\wedge \tau}^{\nu,\kappa;x})\!\\
&\qquad+\!\int_{\rho\wedge \tau}^{\tau}\!R^{w}_{s}(\Theta^{N}_{\kappa,m}\!-\!h_m^N\!-\!w_s\cdot g_m^N\!+\!\dot{\nu}_s\cdot\partial_x g_m^N)(t\!+\!s,X_{s}^{\nu,\kappa;x})\,\ud s\Big],
\end{aligned}
\end{align}
for arbitrary $(w,\nu)\in\cT_t^\delta\times\cA^{\circ}_{t,x}$ and $\rho,\tau\in\cT_t$, with $\Theta^{N}_{\kappa,m}$ as in \eqref{eq:ThetaNkm}. For random times $\rho\le \tau$ we use $[\![\rho,\tau]\!]$ and $(\!(\rho,\tau)\!)$ to denote closed and open random-time intervals.
\begin{theorem}\label{thm:equic_m}
For $p>1$, there is $c_0=c_0(p,N,\kappa,\eps)>0$ such that for all $m\in\N$
\begin{align*}
\big|u^{N,\eps,\delta}_{\kappa,m}(t_1,x_1)-u^{N,\eps,\delta}_{\kappa,m}(t_2,x_2)\big|\le c_0\big(|t_2-t_1|^{\frac{\gamma}{2}\wedge\frac{p-1}{8p}}+|x_2-x_1|^{\frac{p-1}{4p}}\big),
\end{align*}
for any $(t_1,x_1),(t_2,x_2)\in\cO_m$ with $|t_2-t_1|\vee|x_2-x_1|\le 1$.\end{theorem}
\begin{proof}
Fix $m\in\N$. Let $(t,x)\in\cO_m$ and take $\eta\in(0,1)$ such that $x+\eta\le m$. Set $u=u^{N,\eps,\delta}_{\kappa,m}$ and $\cJ_{t,x}=\cJ_{t,x}^{N,\kappa,\eps,\delta,m}$ for notational simplicity. Now we divide the proof into four steps. In Step 1 and Step 2, we prove a lower and an upper bound for $u(t,x+\eta)-u(t,x)$, respectively. In Step 3 and Step 4, we prove an upper and a lower bound for $u(t+\eta,x)-u(t,x)$, respectively. The controlled dynamics $(t+s,X^{\nu,\kappa}_s)$ is always extended to $s\in[0,\infty)$ with zero control after the exit from $\cO_m$.
\medskip

{\bf Step 1.} Let $\nu\in\cA_{t,x+\eta}^\circ$ be optimal for $u(t,x+\eta)$. Recalling \eqref{eq:rhom} we denote $\rho^\eta_m\coloneqq\rho_m(t,x+\eta;\nu,\kappa)$ and $\rho^0_m=\rho_m(t,x;\nu,\kappa)$. For the ease of exposition, we denote $X^{\nu,\kappa;x+\eta}$ and $X^{\nu,\kappa;x}$ by $X^\eta$ and $X^0$, respectively. We also use the compact notation $\bar\rho=\rho_m^0\wedge\rho_m^\eta$. 

Letting $w\in\cT_t^\delta$ be optimal for $u(t,x)$ and extending it to be $w_s=0$ for $s\in(\!(\rho_m^0,\infty)\!)$ we have
\begin{align}\label{eq:umtxeta1}
\begin{aligned}
u(t,x\!+\!\eta)&\ge \cJ_{t,x+\eta}(\nu,w)\\
&=\E\Big[R_{\bar{\rho}}^{w}g_m^N(t\!+\!\bar{\rho},X_{\bar{\rho}}^{\eta})\!+\!\!\int_0^{\bar{\rho}}\!\!R^{w}_s\big(h_m^N\!+\!w_sg_m^N\!+\! H_{m}^{N,\eps}(\cdot,\dot{\nu}_s)\big)(t\!+\!s,X_s^{\eta})\,\ud s\\
&\qquad +\!\int_{\bar \rho}^{\rho_m^\eta}\!R_s^{ w}\big(\Theta^{N}_{\kappa,m}\!+\!\dot{\nu}_s\cdot\partial_xg_m^N\!+\!H_{m}^{N,\eps}(\cdot,\dot{\nu}_s)\big)(t\!+\!s,X_s^{\eta})\ud s\Big],
\end{aligned}
\end{align}
where we used \eqref{eq:gmNDynk} with $(\tau,\rho,X^{\nu,\kappa;x})$ therein replaced by $(\rho_m^\eta,\rho_m^0,X^{\eta})$.
Similarly, we have
\begin{align}\label{eq:umtxeta2}
\begin{aligned}
u(t,x)
&\le \cJ_{t,x}(\nu,w)\\
&=\E\Big[R^{w}_{\bar{\rho}}g_m^N(t\!+\!\bar{\rho},X_{\bar{\rho}}^{0})\!+\!\!\int_0^{\bar{\rho}}\!\!R^{w}_s\big(h_m^N\!+\!w_sg_m^N\!+\! H_{m}^{N,\eps}(\cdot,\dot{\nu}_s)\big)(t\!+\!s,X_s^{0})\ud s\\
&\qquad +\!\int_{\bar \rho}^{\rho_m^0}\!R_s^w \Theta^{N}_{\kappa,m}(t\!+\!s,X_s^{0}) \ud s\Big],
\end{aligned}
\end{align}
where, on $\{\rho^\eta_m<\rho^0_m\}$ we have $\dot \nu_s=0$ and $ H_{m}^{N,\eps}(t+s,X_s^{0},\dot{\nu}_s)=H_{m}^{N,\eps}(t+s,X_s^{0},0) =0$ for $s\in(\!(\bar\rho,\rho^0_m]\!]$.

Combining \eqref{eq:umtxeta1} and \eqref{eq:umtxeta2} yields
\begin{align}\label{eq:equicspadiff}
\begin{aligned}
u(t,x+\eta)-u(t,x)&\ge \E\Big[R_{\bar{\rho}}^wg_m^N(t+\bar{\rho},X_{\bar{\rho}}^\eta)-R_{\bar{\rho}}^wg_m^N(t+\bar{\rho},X_{\bar{\rho}}^0)\\
&\qquad+ \int_0^{\bar{\rho}}R_s^w\big(h_m^N(t+s,X_s^\eta)-h_m^N(t+s,X_s^0)\big)\ud s
\\
&\qquad +\int_0^{\bar{\rho}}R_s^w\big(H_{m}^{N,\eps}(t+s,X_s^\eta,\dot{\nu}_s)-H_{m}^{N,\eps}(t+s,X_s^0,\dot{\nu}_s)\big)\ud s\\
&\qquad +\int_0^{\bar{\rho}}R_s^ww_s\big(g_m^N(t+s,X_s^\eta)-g_m^N(t+s,X_s^0)\big)\ud s\\
&\qquad+\mathds{1}_{\{\rho_m^\eta\ge\rho_m^0\}}\int_{\rho_m^0}^{\rho_m^\eta}R_s^w\big(\Theta^{N}_{\kappa,m}+\dot{\nu}_s\cdot\partial_xg_m^N+H_{m}^{N,\eps}(\cdot,\dot{\nu}_s)\big)(t+s,X_s^\eta)\ud s\\
&\qquad -\mathds{1}_{\{\rho_m^0\ge\rho_m^\eta\}}\int_{\rho_m^\eta}^{\rho_m^0}R_s^w\Theta^{N}_{\kappa,m}(t+s,X_s^0)\ud s\Big].
\end{aligned}
\end{align}
By definition of $H_{m}^{N,\eps}$ (see \eqref{eq:hmltn}) and since $x\mapsto (\bar{\alpha}^N_m)^2(t,x)$ is a $L_N$-Lipschitz function, 
\begin{align*}
&\big|H_{m}^{N,\eps}(t\!+\!s,X_s^\eta,\dot{\nu}_s)\!-\!H_{m}^{N,\eps}(t\!+\!s,X_s^0,\dot{\nu}_s)\big|\\
&\le \tfrac{1}{\eps}\big|(\bar{\alpha}^N_m)^2(t\!+\!s,X_s^\eta)\!-\!(\bar{\alpha}^N_m)^2(t\!+\!s,X_s^0)\big|\le \tfrac{1}{\eps}L_N|X_s^\eta-X_s^0|.
\end{align*}
Functions $h_m^N$ and $g_m^N$ are Lipschitz with constants $L_N$ and $\Lambda_N$, respectively. Using $R^w_s\le 1$, yields
\begin{align}\label{eq:equicngha}
\begin{aligned}
&\E\Big[R_{\bar{\rho}}^wg_m^N(t+\bar{\rho},X_{\bar{\rho}}^\eta)-R_{\bar{\rho}}^wg_m^N(t+\bar{\rho},X_{\bar{\rho}}^0)+ \int_0^{\bar{\rho}}R_s^w\big(h_m^N(t+s,X_s^\eta)-h_m^N(t+s,X_s^0)\big)\ud s
\\
&\quad +\int_0^{\bar{\rho}}R_s^w\big(H_{m}^{N,\eps}(t+s,X_s^\eta,\dot{\nu}_s)-H_{m}^{N,\eps}(t+s,X_s^0,\dot{\nu}_s)\big)\ud s\Big]\\
&\ge- \Lambda_N\E\Big[\big|X_{\bar{\rho}}^\eta-X_{\bar{\rho}}^0\big|\Big]-(1+\eps^{-1})L_N\int_0^T\E\Big[ \big|X_{s\wedge\bar{\rho}}^\eta-X_{s\wedge\bar{\rho}}^0\big| \Big]\ud s\\
&\ge- \big(\Lambda_N+(1+\eps^{-1})L_N T\big)\sup_{s\in[0,T]}\E\Big[\big|X_{s\wedge\bar{\rho}}^\eta-X_{s\wedge\bar{\rho}}^0\big|\Big].
\end{aligned}
\end{align}

For the fourth term under expectation in \eqref{eq:equicspadiff}
\begin{align*}
&\E\Big[\int_0^{\bar{\rho}}R_s^ww_s\big(g_m^N(t+s,X_s^\eta)-g_m^N(t+s,X_s^0)\big)\ud s\Big]\ge- \E\Big[\Lambda_N\int_0^{\bar{\rho}}R_s^w(r+w_s)|X_s^\eta-X_s^0|\ud s\Big].
\end{align*}
Using integration by parts with $-\ud R^w_s=R^w_s(r+w_s)\ud s$ and Tanaka's formula yields 
\begin{align*}
&\E\Big[\int_0^{\bar{\rho}}R_s^ww_s\big(g_m^N(t+s,X_s^\eta)-g_m^N(t+s,X_s^0)\big)\ud s\Big] \\
&\ge \Lambda_N\E\Big[R_{\bar{\rho}}^w\big|X_{\bar{\rho}}^\eta-X_{\bar{\rho}}^0\big|\!-\!\eta\!-\!\int_0^{\bar{\rho}}\!\!R_s^w\sign(X_s^\eta\!-\!X_s^0)(\mu_\kappa(X_s^\eta)\!-\!\mu_\kappa(X_s^0))\ud s\!-\!\int_0^{\bar{\rho}}\!\!R_s^w\ud L^0_s(X^\eta\!-\!X^0)\Big],
\end{align*}
where $(L^0_t)_{t\in[0,T]}$ is the local time at zero of the continuous semi-martingale $(X^\eta_t-X^0_t)_{t\in[0,T]}$. Using \cite[Cor.\ 3 of Thm.\ 75]{protter2005stochastic}, we know
\begin{align}\label{eq:loctime}
\begin{aligned}
L_{\bar \rho}^0(X^\eta-X^0)&=\lim_{\lambda\to0}\frac{1}{2\lambda}\int_0^{\bar \rho}\mathds{1}_{\{0\le |X^\eta_s-X^0_s|\le \lambda\}}(\sigma_\kappa(X^\eta_s)-\sigma_\kappa(X^0_s))^2\ud s\\
&\le \lim_{\lambda\to0}\frac{1}{2\lambda}\int_0^{\bar \rho}\mathds{1}_{\{0\le |X^\eta_s-X^0_s|\le \lambda\}}D_\gamma^2|X^\eta_s-X^0_s|^{2\gamma}\ud s\le \tfrac12 TD^2_\gamma\lim_{\lambda\to 0}\lambda^{2\gamma-1}=0,
\end{aligned}
\end{align}
where the first inequality holds because $\sigma_\kappa$ is $\gamma$-H\"older continuous with parameter $\gamma > 1/2$ by Assumption \ref{ass:gen1}.
Therefore, we have
\begin{align}\label{eq:equicngha1}
\begin{aligned}
&\E\Big[\int_0^{\bar{\rho}}R_s^ww_s\big(g_m^N(t+s,X_s^\eta)-g_m^N(t+s,X_s^0)\big)\ud s\Big]\\
&\ge-\Lambda_N\Big(\eta+\E\Big[\big|X_{\bar{\rho}}^\eta-X_{\bar{\rho}}^0\big| \Big]+\int_0^T\E\Big[\big|\mu_\kappa(X_{s\wedge\bar{\rho}}^\eta)-\mu_\kappa(X_{s\wedge\bar{\rho}}^0)\big|\Big]\ud s \Big)\\
&\ge-\Lambda_N\Big(\eta+\E\Big[\big|X_{\bar{\rho}}^\eta-X_{\bar{\rho}}^0\big| \Big]+D_1\int_0^T\E\Big[\big|X_{s\wedge\bar{\rho}}^\eta-X_{s\wedge\bar{\rho}}^0\big|\Big]\ud s \Big)\\
&\ge-\Lambda_N\Big(\eta+(1+TD_1)\sup_{s\in[0,T]}\E\Big[\big|X_{s\wedge\bar{\rho}}^\eta-X_{s\wedge\bar{\rho}}^0\big|\Big] \Big).
\end{aligned}
\end{align}

Plugging \eqref{eq:equicngha} and \eqref{eq:equicngha1} into \eqref{eq:equicspadiff} yields
\begin{align}\label{eq:equicm_a01}
\begin{aligned}
u(t,x\!+\!\eta)\!-\!u(t,x)&\ge -\Lambda_N\eta\!-\!\Big(\Lambda_N\!+\!(1+\eps^{-1})L_NT\!+\!\Lambda_N(1+TD_1)\Big)\sup_{s\in[0,T]}\E\Big[\big|X_{s\wedge\bar{\rho}}^\eta\!-\!X_{s\wedge\bar{\rho}}^0\big|\Big]\\
&\quad+\E\Big[\mathds{1}_{\{\rho_m^\eta\ge\rho_m^0\}}\int_{\rho_m^0}^{\rho_m^\eta}\!\!R_s^w\big(\Theta^{N}_{\kappa,m}+\dot{\nu}_s\cdot\partial_xg_m^N+H_{m}^{N,\eps}(\cdot,\dot{\nu}_s)\big)(t+s,X_s^\eta)\ud s\\
&\qquad\quad -\mathds{1}_{\{\rho_m^0\ge\rho_m^\eta\}}\int_{\rho_m^\eta}^{\rho_m^0}\!\!R_s^w \Theta^{N}_{\kappa,m}(t+s,X_s^0)\ud s\Big].
\end{aligned}
\end{align}
Now we need lower bounds on the remaining two integrals. Let $\beta,\ell>0$ be arbitrary constants, set $M_5=C_\Theta+2\eps^{-1}\Lambda_N^2$ and $M_6= 8\e^{rT}\eps^{-1}TM_5$ (with $C_\Theta$ as in \eqref{eq:ThetaNkmbnd}) and let 
\begin{align}\label{eq:taum6}
\tau_{M_6}\coloneqq\inf\Big\{s\ge \rho^0_m \,\Big|\,\int_{\rho^0_m}^s|\dot \nu_\lambda|^2\ud \lambda>M_6\Big\}\wedge(T-t).
\end{align}
We claim here and we will prove it in Lemmas \ref{lem:equicon_m} and \ref{lem:equicon_m1} that
\begin{align}\label{eq:equicon_1m}
\begin{aligned}
&\mathds{1}_{\{\rho_m^\eta\ge\rho_m^0\}}\!\!\int_{\rho_m^0}^{\rho_m^\eta}\!\!R_s^w\big(\Theta^{N}_{\kappa,m}+\dot{\nu}_s\cdot\partial_xg_m^N+H_{m}^{N,\eps}(\cdot,\dot{\nu}_s)\big)(t+s,X_s^\eta)\ud s \\
&\ge-M_5T\Big(\tfrac{\ell}{T}\!+\!\mathds{1}_{\big\{X^\eta_{\rho_m^0}-X^0_{\rho_m^0}\ge \beta\big\}}\mathds{1}_{\{\rho_m^\eta\ge\rho_m^0\}}\!+\!\mathds{1}_{\{\rho_m^0+\ell \le \tau_{M_6}\wedge\rho^\eta_m\}}\mathds{1}_{\big\{\inf_{0\leq s\le \ell}\big(X_{\rho_m^0+s}^{\eta}-X_{\rho_m^0}^{\eta}\big)>-\beta\big\}}\Big)
\end{aligned}
\end{align}
and
\begin{align}\label{eq:equicon_2m}
\begin{aligned}
\mathds{1}_{\{\rho_m^0\ge\rho_m^\eta\}}\int_{\rho_m^\eta}^{\rho_m^0}\!\!R_s^w \Theta^{N}_{\kappa,m}(t+s,X_s^0)\ud s&\le \!C_\Theta\ell \!+\!C_\Theta T\mathds{1}_{\big\{X_{\rho_m^\eta}^{\eta}-X_{\rho_m^\eta}^{0}\ge \beta\big\}}\mathds{1}_{\{\rho_m^0\ge\rho_m^\eta\}}\\
&\quad+\!C_\Theta T\mathds{1}_{\big\{\sup_{0\leq s\leq \ell}\big(X_{\rho_m^\eta+s}^{0}-X_{\rho_m^\eta}^{0}\big)\le \beta\big\}}\mathds{1}_{\{\rho^0_m\ge \rho^\eta_m+\ell\}}.
\end{aligned}
\end{align}

From now on $c_0>0$ is a constant that may vary from line to line and it may depend on $N$, $\kappa$ and $\eps$ but is independent of $\delta$ and $m$. Combining \eqref{eq:equicon_1m} and \eqref{eq:equicon_2m} with \eqref{eq:equicm_a01}, we get
\begin{align}\label{eq:equicm_a}
\begin{aligned}
u(t,x\!+\!\eta)\!-\!u(t,x)&\ge -c_0\Big[\ell\!+\!\eta\!+\!\sup_{s\in[0,T]}\E\big[|X_{s\wedge\bar{\rho}}^\eta\!-\!X_{s\wedge\bar{\rho}}^0|\big]\!+\!\P\Big(X^\eta_{\bar\rho}\ge X^0_{\bar\rho}+\beta\Big)\\
&\qquad\quad+\!\P\Big(\inf_{0\leq s\le \ell}\big(X_{\rho_m^0+s}^{\eta}\!-\!X_{\rho_m^0}^{\eta}\big)>-\beta,\rho_m^0+\ell \le \tau_{M_6}\wedge\rho^\eta_m\Big)\\
&\qquad\quad+\P\Big(\sup_{0\leq s\leq \ell}\big(X_{\rho_m^\eta+s}^{0}-X_{\rho_m^\eta}^{0}\big)\le \beta,\rho^0_m\ge \rho^\eta_m+\ell\Big)\Big].
\end{aligned}
\end{align}

In the next paragraphs, we study separately the terms on the right-hand side of \eqref{eq:equicm_a}. By It\^o-Tanaka formula and recalling $L_{s}^0(X^\eta\!-\!X^0)=0$ for $s\in[\![0,\bar{\rho}]\!]$ (cf.\ \eqref{eq:loctime}), we get 
\begin{align*}
\E\big[|X_{s\wedge\bar\rho}^\eta\!-\!X_{s\wedge\bar\rho}^0|\big]=\eta\!+\!\E\Big[\!\int_0^{s\wedge\bar{\rho}}\!\!\sign(X_t^\eta\!-\!X_t^0)\big(\mu_\kappa(X_t^\eta)\!-\!\mu_\kappa(X_t^0)\big)\ud t\Big]
\le \eta\!+\!D_1\int_0^{s}\!\! \E\big[|X_{t\wedge\bar\rho}^\eta\!-\!X_{t\wedge\bar\rho}^0|\big]\ud t
\end{align*}
where we used the Lipschitz property of $\mu_\kappa$. Gronwall's lemma yields
\begin{align}\label{eq:equicm_b}
\E\big[|X_{s\wedge\bar{\rho}}^\eta-X_{s\wedge\bar{\rho}}^0|\big]\le C_1 \eta ,\quad \text{for all $s\in[0,T]$},
\end{align}
where $C_1>0$ is a constant independent of $\eta$, $s$ and the control $\nu$. Similarly, by Markov's inequality
\begin{align}\label{eq:equicm_c}
\begin{aligned}
\P\Big(X^\eta_{\bar\rho}\ge X^0_{\bar\rho}\!+\!\beta\Big)\leq \P\Big(|X_{\bar{\rho}}^\eta-X_{\bar{\rho}}^0|\ge \beta\Big)\leq \tfrac{1}{\beta}\E\Big[|X_{\bar{\rho}}^\eta-X_{\bar{\rho}}^0|\Big]\leq C_1 \tfrac{\eta}{\beta}.
\end{aligned}
\end{align}

The term on the second line of \eqref{eq:equicm_a} is bounded using a measure-change and a time-change. For the measure-change let $(\phi_s)_{s\in[0,T]}$ be the adapted process
\begin{align}\label{eq:phichangemeas}
\phi_s\coloneqq-\frac{\mu_\kappa(X_s^\eta)+\dot{\nu}_s}{\sigma_\kappa(X_s^\eta)}\mathds{1}_{\{\rho_m^0\le s \le (\rho_m^0+\ell)\wedge\tau_{M_6}\}}
,\quad\text{for $s\in[0,T]$}.
\end{align}
Using $|\mu_\kappa(x)/\sigma_\kappa(x)|\le \kappa^{-2}$ and $|(\sigma_\kappa)^{-1}(x)|\le \kappa^{-1}$ for all $x\in[0,\infty)$, we have, $\P$-a.s.,
\begin{align}\label{eq:phinov}
\begin{aligned}
\frac12\int_0^{T-t}\!\!\phi_s^2\,\ud s
\le T\kappa^{-4}+\kappa^{-2}\int_{\rho^0_m}^{\tau_{M_6}\vee \rho^0_m}|\dot{\nu}_s|^2\ud s\le T\kappa^{-4}+M_6\kappa^{-2}.
\end{aligned}
\end{align}
Then, the Radon-Nikodym derivative
$\frac{\ud\Q}{\ud\P}|_{\cF_{T-t}}=\e^{\int_0^{T-t}\phi_s\,\ud W_s-\frac{1}{2}\int_0^{T-t}\phi_s^2\ud s}$, defines a measure-change and 
$W^\Q_s=W_s-\int_0^s \phi_\lambda\ud \lambda$ is a $\Q$-Brownian motion for $s\in[0,T-t]$. Notice that on the event $\{\tau_{M_6}\ge \rho^0_m+\ell\}$ we have for every $s\in[0,\ell]$,
\begin{align*}
X_{\rho_m^0+s}^\eta-X_{\rho_m^0}^\eta&=\int_{\rho_m^0}^{\rho_m^0+s}\!\!\big(\mu_\kappa(X_v^{\eta})\!+\!\dot\nu_v\big)\ud v\!+\!\int_{\rho_m^0}^{\rho_m^0+s}\!\!\sigma_\kappa(X_v^{\eta})\ud W_v=\int_{\rho_m^0}^{\rho_m^0+s}\!\!\sigma_\kappa(X_v^{\eta})\ud W^\Q_v.
\end{align*}
For $p>1$, we can rewrite the term on the second line of \eqref{eq:equicm_a} as
\begin{align}\label{eq:equicm_d}
\begin{aligned}
&\P\Big(\rho_m^0+\ell \le \tau_{M_6}\wedge\rho^\eta_m,\inf_{\rho^0_m\leq s\le\rho^0_m+ \ell}\Big(\int_{\rho_m^0}^{s}\!\!\big(\mu_\kappa(X_v^{\eta})\!+\!\dot\nu_v\big)\ud v\!+\!\int_{\rho_m^0}^{s}\!\!\sigma_\kappa(X_v^{\eta})\ud W_v\Big)>-\beta\Big)
\\
&=\E^{\Q}\Big[\Big(\tfrac{\ud \P}{\ud \Q}\Big|_{\cF_{T-t}}\Big)\mathds{1}_{\{\rho_m^0+\ell \le \tau_{M_6}\wedge\rho^\eta_m\}}\mathds{1}_{\big\{\inf_{0\leq s\leq \ell}\int_{\rho_m^0}^{\rho_m^0+s}\sigma_\kappa(X_v^{\eta})\ud W^\Q_v>-\beta\big\}}\Big]\\
&\leq\E^{\Q}\Big[\Big(\tfrac{\ud \P}{\ud \Q}\Big|_{\cF_{T-t}}\Big)^p\Big]^{\frac{1}{p}}\Q\Big(\inf_{0\leq s\leq \ell}\int_{\rho_m^0}^{\rho_m^0+s}\sigma_\kappa(X_v^{\eta})\ud W^\Q_v>-\beta\Big)^{\frac{p-1}{p}},
\end{aligned}
\end{align}
where $X^\eta_s$ is extended with $\dot \nu_s=0$ for $s\!\in\![\![\rho^\eta_m,\infty)\!)$. Setting $C_2\!=\!T\frac{1}{\kappa^{4}}\!+\!M_6\frac{1}{\kappa^{2}}$, from \eqref{eq:phinov} we get
\begin{align}\label{eq:equicm_d1}
\begin{aligned}
\E^{\Q}\Big[\Big(\tfrac{\ud \P}{\ud \Q}\big|_{\cF_{T-t}}\Big)^p\Big]&=\E^{\Q}\Big[\e^{-\int_0^{T-t}p \phi_s\ud W^\Q_s-\tfrac{p}{2}\int_0^{T-t}(\phi_s)^2\ud t}\Big]\\
&= \E^{\Q}\Big[\e^{-\int_0^{T-t}p \phi_s\ud W^\Q_s-\tfrac{1}{2}\int_0^{T-t}(p\phi_s)^2\ud t}
\e^{\tfrac{p}{2}(p-1)\int_0^{T-t}(\phi_s)^2\ud t}\Big]\\
&\le \e^{C_2p^2}\E^{\Q}\Big[\e^{-\int_0^{T-t}p \phi_s\ud W^\Q_s-\tfrac{1}{2}\int_0^{T-t}(p\phi_s)^2\ud t}
\Big]
\le\e^{C_2p^2}.
\end{aligned}
\end{align} 

By standard time-change (\cite[Thm.\ 8.2]{baldi2017stochastic}) we have $\int_{\rho_m^0}^{\rho_m^0+s}\sigma_\kappa(X^\eta_v)\ud W_v^\Q=B^\Q_{A_s}$, for $s\ge 0$, where $A_{s}\coloneqq\int_{\rho_m^0}^{\rho_m^0+s}(\sigma_\kappa(X^\eta_v))^2\ud v$ and $B^\Q$ is another $\Q$-Brownian motion (for a time-changed filtration). Since $\sigma_\kappa(x)\ge \kappa$, then $A_\ell\ge \kappa^2 \ell$. It follows that
\begin{align}\label{eq:equicm_d2}
\begin{aligned}
&\Q\Big(\inf_{0\leq s\leq \ell}\int_{\rho_m^0}^{\rho_m^0+s}\sigma_\kappa(X_v^{\eta})\ud W^\Q_v>-\beta\Big)=\Q\Big(\inf_{0\leq s\leq A_\ell} B^\Q_s >-\beta\Big)\\
&\leq \Q\big( \inf_{0\leq s\leq \kappa^2\ell}\! B^\Q_s>-\beta\big)=\Q(|B^\Q_{\kappa^2\ell}|<\beta)=\Q\big(|B^\Q_{1}|\le\tfrac{\beta}{\kappa \sqrt{\ell}}\big)\leq 2\tfrac{\beta}{\kappa \sqrt{\ell}},
\end{aligned}
\end{align}
where we used $\inf_{0\leq s\leq \kappa^2\ell}\! B^\Q_s=-|B^\Q_{\kappa^2\ell}|$ in law.
Plugging \eqref{eq:equicm_d1} and \eqref{eq:equicm_d2} into \eqref{eq:equicm_d}, we obtain
\begin{align}\label{eq:equicm_d3}
\P\Big(\rho_m^0+\ell \le \tau_{M_6}\wedge\rho^\eta_m,\inf_{0\leq s\le \ell}\big(X_{\rho_m^0+s}^{\eta}-X_{\rho_m^0}^{\eta}\big)>-\beta \Big)\le \e^{C_2p}\Big(2\tfrac{\beta}{\kappa\sqrt{\ell}}\Big)^{\frac{p-1}{p}}.
\end{align}

It remains to bound the term on the third line of \eqref{eq:equicm_a}. Since $\dot{\nu}_s=0$ for $s\in[\![\rho_m^\eta,\infty)\!)$, then 
\begin{align*}
X_{\rho_m^\eta+s}^0\!-\!X_{\rho_m^\eta}^0&=\int_{\rho_m^\eta}^{\rho_m^\eta+s}\!\! \mu_\kappa(X_v^{0})\ud v\!+\!\int_{\rho_m^\eta}^{\rho_m^\eta+s}\!\!\sigma_\kappa(X_v^{0})\ud W_v=\int_{\rho_m^\eta}^{\rho_m^\eta+s}\!\!\sigma_\kappa(X_v^{0})\ud\widetilde W^\Q_v=\widetilde B^\Q_{C_s},
\end{align*}
where $\widetilde{W}_s^\Q=W_s-\int_0^s \tilde{\phi}_\lambda\ud \lambda$ is a Brownian motion under the measure $\widetilde\Q$ defined via the Dol\'eans-Dade exponential of the adapted process 
$\tilde{\phi}_s\coloneqq-(\mu_\kappa/\sigma_\kappa)(X_s^0)\mathds{1}_{\{\rho_m^\eta\le s\le \rho_m^\eta+\ell\}}$. 
The final expression holds with $C_s\coloneqq\int_{\rho_m^\eta}^{\rho_m^\eta+s}(\sigma_\kappa(X^0_v))^2\ud v$ and $\widetilde B^{\Q}$ is another $\widetilde{\Q}$-Brownian motion. Since $\sigma_\kappa(x)\ge \kappa$, then similar calculations as in \eqref{eq:equicm_d2} yield
\begin{align}\label{eq:equicm_e}
\begin{aligned}
\P\Big(\sup_{0\leq s\leq \ell}\big(X_{\rho_m^\eta+s}^{0}\!-\!X_{\rho_m^\eta}^{0}\big)\!\le\! \beta,\rho^0_m\ge \rho^\eta_m\!+\!\ell\Big)\le\! \e^{C_2p}\widetilde{\Q}\Big(|\widetilde B^{\Q}_1|\!\le\! \tfrac{\beta}{\kappa\sqrt{\ell}}\Big)^{\frac{p-1}{p}}\!\!\le\! \e^{C_2p}\Big(2\tfrac{\beta}{\kappa\sqrt{\ell}}\Big)^{\frac{p-1}{p}}.
\end{aligned}
\end{align}

Plugging \eqref{eq:equicm_b}, \eqref{eq:equicm_c}, \eqref{eq:equicm_d3} and \eqref{eq:equicm_e} into \eqref{eq:equicm_a}, and changing constant $c_0>0$ as needed, we obtain
$u(t,x\!+\!\eta)\!-\!u(t,x)\ge -c_0\big(\eta+\ell+\eta/\beta+(\beta/\sqrt{\ell})^{\frac{p-1}{p}}\big)$.
Taking $\beta=\ell=\sqrt \eta$ and $\eta\le 1$, gives 
\begin{align}\label{eq:equicontm1}
u(t,x+\eta)-u(t,x)\ge -c_0 \eta^{1/4-1/4p}
\end{align}
with suitable $c_0=c_0(p,N,\kappa,\eps)>0$ independent of $m$ and $t$.
\medskip

{\bf Step 2.} An upper bound for $u(t,x+\eta)-u(t,x)$ is obtained with similar methods as in the step above. Let $\nu\in\cA_{t,x}^\circ$ be optimal for $u(t,x)$ and let $\rho_m^0\coloneqq\rho_m(t,x;\nu,\kappa)$ as in \eqref{eq:rhom}. We extend $\nu$ as $\dot\nu_s=0$ for $s\in[\![\rho_m^0,\infty)\!)$. Given the dynamics $(t+s,X^{\nu,\kappa;x+\eta}_s)$, we set $\rho^\eta_m\coloneqq\rho_m(t,x+\eta;\nu,\kappa)$.
Let $w\in\cT_t^\delta$ be optimal for $u(t,x+\eta)$ and let us extend it as $w_s=0$ for $s\in(\!(\rho_m^\eta,\infty)\!)$. For simplicity, we let $X^{\nu,\kappa;x+\eta}=X^{\eta}$, $X^{\nu,\kappa;x}=X^0$ and $\bar{\rho}=\rho_m^0\wedge\rho_m^\eta$. 

We have $u(t,x\!+\!\eta)\!-\!u(t,x)\le \cJ_{t,x+\eta}(\nu,w)-\cJ_{t,x}(\nu,w)$ and thus
\begin{align}\label{eq:equicm_b01}
\begin{aligned}
&u(t,x\!+\!\eta)\!-\!u(t,x)\\
&\le \E\Big[R_{\bar{\rho}}^wg_m^N(t+\bar{\rho},X_{\bar{\rho}}^\eta)-R_{\bar{\rho}}^wg_m^N(t+\bar{\rho},X_{\bar{\rho}}^0)\\
&\qquad+ \int_0^{\bar{\rho}}\!\!R_s^w\big[h_m^N(t\!+\!s,X_s^\eta)\!-\!h_m^N(t\!+\!s,X_s^0)\!+\!w_s\big(g_m^N(t\!+\!s,X_s^\eta)\!-\!g_m^N(t\!+\!s,X_s^0)\big)\big]\ud s
\\
&\qquad +\int_0^{\bar{\rho}}\!\!R_s^w\big(H_{m}^{N,\eps}(t\!+\!s,X_s^\eta,\dot{\nu}_s)\!-\!H_{m}^{N,\eps}(t\!+\!s,X_s^0,\dot{\nu}_s)\big)\ud s\!+\!\mathds{1}_{\{\rho_m^\eta\ge\rho_m^0\}}\!\!\int_{\rho_m^0}^{\rho_m^\eta}\!\!R_s^w\Theta^{N}_{\kappa,m}(t\!+\!s,X_s^\eta)\ud s\\
&\qquad -\mathds{1}_{\{\rho_m^0\ge\rho_m^\eta\}}\int_{\rho_m^\eta}^{\rho_m^0}\!\!R_s^w\big(\Theta^{N}_{\kappa,m}+\dot{\nu}_s\cdot\partial_xg_m^N+H_{m}^{N,\eps}(\cdot,\dot{\nu}_s)\big)(t+s,X_s^0)\ud s \Big].
\end{aligned}
\end{align}

By similar calculations as in \eqref{eq:equicngha} and \eqref{eq:equicngha1}, we find an upper bound 
\begin{align*}
\begin{aligned}
u(t,x\!+\!\eta)\!-\!u(t,x)&\le \Lambda_N\eta\!+\!\Big(\Lambda_N\!+\!(1\!+\!\eps^{-1})L_NT\!+\!\Lambda_N(1\!+\!TD_1)\Big)\sup_{s\in[0,T]}\E\Big[|X_{s\wedge\bar{\rho}}^\eta\!-\!X_{s\wedge\bar{\rho}}^0|\Big]\\
&\quad+\!\E\Big[\mathds{1}_{\{\rho_m^\eta\ge\rho_m^0\}}\int_{\rho_m^0}^{\rho_m^\eta}\!\!R_s^w\Theta^{N}_{\kappa,m}(t\!+\!s,X_s^\eta)\ud s \\
&\qquad\quad-\mathds{1}_{\{\rho_m^0\ge\rho_m^\eta\}}\int_{\rho_m^\eta}^{\rho_m^0}\!\!R_s^w\big(\Theta^{N}_{\kappa,m}\!+\!\dot{\nu}_s\cdot\partial_xg_m^N\!+\!H_{m}^{N,\eps}(\cdot,\dot{\nu}_s)\big)(t\!+\!s,X_s^0)\ud s \Big].
\end{aligned}
\end{align*}
Similarly to \eqref{eq:equicm_a01}, we need two bounds on the remaining two integrals. We let 
\begin{align}\label{eq:taum6_2}
\tau_{M_6}'\coloneqq\inf\Big\{s\ge \rho^\eta_m\,\Big|\, \int_{\rho^\eta_m}^s|\dot \nu_\lambda|^2\ud \lambda>M_6\Big\}\wedge(T-t).
\end{align}
This stopping time plays the same role as $\tau_{M_6}$ in Step 1 (cf.\ \eqref{eq:taum6}). Recalling also $M_5$ from the paragraph above \eqref{eq:taum6}, for arbitrary $\beta,\ell>0$ we have
\begin{align}\label{eq:equicon_3m}
\begin{aligned}
\mathds{1}_{\{\rho_m^\eta\ge\rho_m^0\}}\int_{\rho_m^0}^{\rho_m^\eta}\!\!R^w_s\Theta^{N}_{\kappa,m}(t\!+\!s,X_s^{\eta})\ud s 
&\le C_\Theta \ell \!+\!C_\Theta T\mathds{1}_{\big\{X_{\rho_m^0}^{\eta}-X_{\rho_m^0}^{0}\ge\beta\big\}}\mathds{1}_{\{\rho_m^\eta\ge \rho_m^0\}}
\\
&\quad +\!C_\Theta T\mathds{1}_{\{\rho_m^\eta>\rho_m^0+\ell\}}\mathds{1}_{\big\{\inf_{0\leq s\leq \ell}\big(X^{\eta}_{\rho_m^0+s}-X^{\eta}_{\rho_m^0}\big)>- \beta\big\}}
\end{aligned}
\end{align}
and
\begin{align}\label{eq:equicon_4m}
\begin{aligned}
&\mathds{1}_{\{\rho_m^0\ge\rho_m^\eta\}}\int_{\rho_m^\eta}^{\rho_m^0}\!\! R^w_s\big(\Theta^{N}_{\kappa,m}\!+\!\dot{\nu}_s\cdot\partial_xg_m^N\!+\!H_{m}^{N,\eps}(\cdot,\dot{\nu}_s)\big)(t\!+\!s,X_s^{0})\ud s\\
&\ge-M_5T\Big(\tfrac{\ell}{T}\!+\!\mathds{1}_{\{ X_{\rho_m^\eta}^{\eta}-X_{\rho_m^\eta}^{0}\ge \beta\}}\mathds{1}_{\{\rho_m^0\ge \rho_m^\eta\}}
\!+\mathds{1}_{\{\rho_m^\eta+\ell \le \tau_{M_6}'\wedge\rho^0_m\}}\mathds{1}_{\big\{\sup_{0\leq s\le \ell}\big(X^{0}_{\rho_m^\eta+s}-X^{0}_{\rho_m^\eta}\big)<\beta\big\}}\Big).
\end{aligned}
\end{align}
The two claims above will be proven in Lemmas \ref{lem:equicon_m} and \ref{lem:equicon_m1}. 

Combining \eqref{eq:equicon_3m} and \eqref{eq:equicon_4m} with \eqref{eq:equicm_b01}, and taking $c_0>0$ with the same conventions as in \eqref{eq:equicm_a}, we obtain
\begin{align}\label{eq:equicm_a1}
\begin{aligned}
u(t,x\!+\!\eta)\!-\!u(t,x)
&\le c_0\!\Big[\ell\!+\!\eta\!+\sup_{s\in[0,T]}\E\Big[|X_{s\wedge\bar{\rho}}^\eta\!-\!X_{s\wedge\bar{\rho}}^0|\Big]\!+\!\P\Big(X^\eta_{\bar\rho}\ge X_{\bar\rho}^0+\beta\Big)\\
&\quad+\!\P\Big(\rho_m^\eta+\ell \le \tau_{M_6}'\wedge\rho^0_m,\sup_{0\leq s\le \ell}\big(X_{\rho_m^\eta+s}^{0}-X_{\rho_m^\eta}^{0}\big)\le\beta\Big)\\
&\quad+ \P\Big(\inf_{0\leq s\leq \ell}\big(X_{\rho_m^0+s}^{\eta}-X_{\rho_m^0}^{\eta}\big)>- \beta,\rho^\eta_m\ge \rho^0_m+\ell\Big)\Big].
\end{aligned}
\end{align}
The last two expressions on the right-hand side of \eqref{eq:equicm_a1} are the analogue of the last two in \eqref{eq:equicm_a}, with $(\rho_m^0, X^\eta,\tau_{M_6}\!\wedge\!\rho^\eta_m)$ therein replaced by $(\rho_m^\eta, X^0,\tau_{M_6}'\!\wedge\!\rho^0_m)$ in \eqref{eq:equicm_a1}. Thus, similar calculations as those leading to \eqref{eq:equicm_d3} and \eqref{eq:equicm_e} allow us to obtain
\begin{align*}
\begin{aligned}
&\P\Big(\rho_m^\eta\!+\!\ell \le \tau_{M_6}'\wedge\rho^0_m,\sup_{0\leq s\le \ell}\!\big(X_{\rho_m^\eta+s}^{0}\!-\!X_{\rho_m^\eta}^{0}\big)\le\beta\Big)\\
&+\! \P\Big(\inf_{0\leq s\leq \ell}\big(X_{\rho_m^0+s}^{\eta}\!-\!X_{\rho_m^0}^{\eta}\big)>- \beta,\rho^\eta_m\ge \rho^0_m\!+\!\ell\Big)\le 
c_0\e^{C_2p}\Big(2\tfrac{\beta}{\kappa\sqrt{\ell}}\Big)^{\frac{p-1}{p}}.
\end{aligned}
\end{align*}
Since the constant $C_1>0$ in \eqref{eq:equicm_b} does not depend on the choice of $\nu$, as in \eqref{eq:equicm_c} we have $\P\big(X_{\bar\rho}^{\eta}\!-\!X_{\bar\rho}^{0}\ge\beta\big)\le C_1 \eta/\beta$. Plugging the latter two bounds into \eqref{eq:equicm_a1} and choosing $\beta=\ell=\sqrt{\eta}$ we get $u(t,x\!+\!\eta)\!-\!u(t,x)\le c_0\eta^{\frac14-\frac{1}{4p}}$ with the same $c_0$ as in \eqref{eq:equicontm1}. Thus $|u(t,x+\eta)-u(t,x)|\le c_0 \eta^{\frac14-\frac{1}{4p}}$.
\smallskip

{\bf Step 3.} In this and in the next step we are going to slightly abuse the notation by relabelling stopping times from steps 1 and 2 with different meaning. This avoids introducing heavier notations and it should cause no confusion, because steps 3 and 4 are independent of the previous ones.

Let $\eta\in(0,1)$ be such that $t+\eta\le T$. Let $\nu\in\cA_{t,x}^\circ$ be optimal for $u(t,x)$. Let $\rho_m^0\coloneqq\rho_m(t,x;\nu,\kappa)$ and $\rho_m^\eta=\rho_m(t+\eta,x;\nu,\kappa)$ be as in \eqref{eq:rhom}.
Let $w\in\cT_{t+\eta}^\delta$ be optimal for $u(t+\eta,x)$ and extend it to be $w_s=0$ for $s\in(\!(\rho_m^\eta,\infty)\!)$. For the ease of exposition, we denote $X^{\nu,\kappa;x}=X$. Due to the geometry of the set $\cO_m$ it is easy to verify that $\rho_m^\eta\le\rho_m^0$, $\P$-a.s. We have
\begin{align}
\begin{aligned}
u(t+\eta,x)&\le \cJ_{t+\eta,x}(\nu,w)\\
&= \E\Big[R_{\rho_m^\eta}^wg_m^N(t+\eta+\rho_m^\eta,X_{\rho_m^\eta})+\int_0^{\rho_m^\eta}\!\!R_s^w\big(h_m^N+w_sg_m^N+H^{N,\eps}_m(\cdot,\dot{\nu}_s)\big)(t+\eta+s,X_s)\ud s\Big]
\end{aligned}
\end{align}
and
\begin{align}
\begin{aligned}
u(t,x)&\ge \cJ_{t,x}(\nu,w) \\
&=\E\Big[R_{\rho_m^\eta}^wg_m^N(t+\rho_m^\eta,X_{\rho_m^\eta})+\int_0^{\rho_m^\eta}\!\!R_s^w\big(h_m^N+w_sg_m^N+H^{N,\eps}_m(\cdot,\dot{\nu}_s)\big)(t+s,X_s)\ud s
\\
&\qquad+\int_{\rho_m^\eta}^{\rho_m^0}\!\!R^w_s\big(\Theta^{N}_{\kappa,m}+\dot{\nu}_s\cdot\partial_x g_m^N+H^{N,\eps}_m(\cdot,\dot{\nu}_s)\big)(t+s,X_s)\ud s\Big],
\end{aligned}
\end{align}
where we used \eqref{eq:gmNDynk} with $(\tau,\sigma,X^{\nu,\kappa;x})=(\rho_m^0,\rho_m^\eta,X)$. Combining the two, we obtain
\begin{align}\label{eq:equitime2a}
\begin{aligned}
u(t+\eta,x)-u(t,x)&\le \E\Big[R_{\rho_m^\eta}^wg_m^N(t+\eta+\rho_m^\eta,X_{\rho_m^\eta})-R_{\rho_m^\eta}^wg_m^N(t+\rho_m^\eta,X_{\rho_m^\eta})\\
&\qquad+\int_0^{\rho_m^\eta}\!\!R_s^w\big(h_m^N(t+\eta+s,X_s)-h_m^N(t+s,X_s)\big)\ud s
\\
&\qquad+\int_0^{\rho_m^\eta}\!\!R_s^w\big(H^{N,\eps}_m(t+\eta+s,X_s,\dot{\nu}_s)-H^{N,\eps}_m(t+s,X_s,\dot{\nu}_s)\big)\ud s
\\
&\qquad+\int_0^{\rho_m^\eta}\!\!R_s^w w_s\big(g_m^N(t+\eta+s,X_s)-g_m^N(t+s,X_s)\big)\ud s
\\
&\qquad-\int_{\rho_m^\eta}^{\rho_m^0}\!\!R^w_s\big(\Theta^{N}_{\kappa,m}+\dot{\nu}_s\cdot\partial_x g_m^N+H^{N,\eps}_m(\cdot,\dot{\nu}_s)\big)(t+s,X_s)\ud s
\Big].
\end{aligned}
\end{align}
As functions of time, $g^N_m$, $h^N_m$ are Lipschitz and $H^{N,\eps}_m$ is $\gamma/2$-H\"older, with constants $\Lambda_N$, $L_N$ and $\eps^{-1}L_N$, respectively. Moreover, $R^w_s\le 1$ and $\int_0^tR^w_s w_s\ud s\le 1$ for any $t\in[0,T]$, $\P$-a.s. Recalling $\eta\in(0,1)$ we obtain an upper bound for the first four lines on the right-hand side of \eqref{eq:equitime2a} as
\begin{align}\label{eq:equitime2ab}
\begin{aligned}
\E\Big[&R_{\rho_m^\eta}^wg_m^N(t\!+\!\eta\!+\!\rho_m^\eta,X_{\rho_m^\eta})\!-\!R_{\rho_m^\eta}^wg_m^N(t\!+\!\rho_m^\eta,X_{\rho_m^\eta})\\
&+\!\!\int_0^{\rho_m^\eta}\!\!R_s^w\big[h_m^N(t\!+\!\eta\!+\!s,X_s)\!-\!h_m^N(t\!+\!s,X_s)\!+\!w_s\big(g_m^N(t\!+\!\eta\!+\!s,X_s)\!-\!g_m^N(t\!+\!s,X_s)\big)\big]\ud s
\\
&+\!\!\int_0^{\rho_m^\eta}\!\!R_s^w\big(H^{N,\eps}_m(t\!+\!\eta\!+\!s,X_s,\dot{\nu}_s)\!-\!H^{N,\eps}_m(t\!+\!s,X_s,\dot{\nu}_s)\big)\ud s\Big]\le \big(2\Lambda_N\!+\!(1\!+\!\eps^{-1})L_NT\big)\eta^\frac{\gamma}{2}.
\end{aligned}
\end{align}
Plugging \eqref{eq:equitime2ab} into \eqref{eq:equitime2a} and letting $c_0=c_0(N,\kappa,\eps)>0$ that may vary from line to line, we get
\begin{align}\label{eq:equitime2bc}
\begin{aligned}
&u(t\!+\!\eta,x)\!-\!u(t,x)\le c_0\eta^\frac{\gamma}{2}\!-\!\E\Big[\!\int_{\rho_m^\eta}^{\rho_m^0}\!\!R^w_s\big(\Theta^{N}_{\kappa,m}\!+\!\dot{\nu}_s\cdot\partial_x g_m^N\!+\!H^{N,\eps}_m(\cdot,\dot{\nu}_s)\big)(t\!+\!s,X_s)\ud s\Big].
\end{aligned}
\end{align}

Let $\beta,\ell>0$ be arbitrary constants with $t+\eta+\ell<T$, and similarly to \eqref{eq:taum6_2} let
\begin{align}\label{eq:taum6_3}
\tau_{M_6}'\coloneqq\inf\Big\{s\ge \rho^\eta_m\,\Big|\, \int_{\rho^\eta_m}^s|\dot \nu_\lambda|^2\ud \lambda>M_6\Big\}\wedge(T-t).
\end{align}
We claim here and we will prove it in Lemma \ref{lem:equicon_m2} that 
\begin{align}\label{eq:equicon_5m}
\begin{aligned}
&\int_{\rho_m^\eta}^{\rho_m^0}\!\!R^w_s\big(\Theta^{N}_{\kappa,m}\!+\!\dot{\nu}_s\cdot\partial_xg_m^N\!+\!H_{m}^{N,\eps}(\cdot,\dot{\nu}_s)\big)(t\!+\!s,X_s)\ud s \\
&\ge-M_5(\eta\!+\!\ell)\!-\!M_5T\mathds{1}_{\{\rho_m^\eta+\eta+\ell \le \tau_{M_6}'\wedge\rho^0_m\}}\mathds{1}_{\big\{X_{\rho_m^\eta+\eta}-X_{\rho_m^\eta}\ge\beta\big\}}\! \\
&\quad-\!M_5T\mathds{1}_{\{\rho_m^\eta+\eta+\ell \le \tau_{M_6}'\wedge\rho^0_m\}}\mathds{1}_{\big\{\inf_{0\leq s\le \ell}\big(X_{\rho_m^\eta+\eta+s}-X_{\rho_m^\eta+\eta}\big)>-\beta\big\}}.
\end{aligned}
\end{align}
Plugging \eqref{eq:equicon_5m} into \eqref{eq:equitime2bc} yields
\begin{align}\label{eq:equitime2b}
\begin{aligned}
u(t\!+\!\eta,x)\!-\!u(t,x)\le c_0\Big[&\eta^{\frac{\gamma}{2}}\!+\!\ell\!+\!\P\Big(\rho_m^\eta\!+\!\eta\!+\!\ell\le \tau_{M_6}'\wedge\rho^0_m, X_{\rho_m^\eta+\eta}\!-\!X_{\rho_m^\eta}\ge \beta\Big)\\
&+\!\P\Big(\rho_m^\eta\!+\!\eta\!+\!\ell \le \tau_{M_6}'\wedge\rho^0_m,\inf_{0 \leq s\le \ell}\big(X_{\rho_m^\eta+\eta+s}\!-\!X_{\rho_m^\eta+\eta}\big)>-\beta\Big)\Big].
\end{aligned}
\end{align}

First we show that 
\begin{align}\label{eq:claimtoprove}
\P\Big(\rho_m^\eta\!+\!\eta\!+\!\ell\le \tau_{M_6}'\wedge\rho^0_m, X_{\rho_m^\eta+\eta}\!-\!X_{\rho_m^\eta}\ge \beta\Big)\leq C_3\frac{\eta}{\beta^2},
\end{align}
for a constant $C_3>0$ independent of $\eta$, $\beta$ and $\ell$. We can assume with no loss of generality that 
$\P(\rho_m^\eta\!+\!\eta\!+\!\ell \le \tau_{M_6}'\wedge\rho^0_m)>0$,
as otherwise \eqref{eq:claimtoprove} holds trivially. We have
\begin{align*}
&\P\Big(\rho_m^\eta+\eta+\ell \le \tau_{M_6}'\wedge\rho^0_m,X_{\rho_m^\eta+\eta}-X_{\rho_m^\eta}\ge\beta\Big)\\
&\le \P\Big(\big|X_{(\rho_m^\eta+\eta)\wedge\tau'_{M_6}}-X_{\rho_m^\eta}\big|\ge\beta\Big) \le\tfrac{1}{\beta^2}\E\Big[\big|X_{(\rho_m^\eta+\eta)\wedge\tau_{M_6}'}-X_{\rho_m^\eta}\big|^2\Big].
\end{align*}
Since $\mu_\kappa$ and $\sigma_\kappa$ are bounded (cf.\ Section \ref{sec:approx}) by $\kappa^{-1}$, \eqref{eq:claimtoprove} holds because It\^o isometry yields
\begin{align*}
&\E\Big[\big|X_{(\rho^\eta_m+\eta)\wedge\tau_{M_6}'}-X_{\rho_m^\eta}\big|^2\Big]\\
&=\E\Big[\Big|\int_{\rho_m^\eta}^{(\rho^\eta_m+\eta)\wedge\tau_{M_6}'}\!\mu_\kappa(X_\lambda)\ud \lambda\!+\!\int_{\rho_m^\eta}^{(\rho^\eta_m+\eta)\wedge\tau_{M_6}'}\!\sigma_\kappa(X_\lambda)\ud W_\lambda+\int_{\rho_m^\eta}^{(\rho^\eta_m+\eta)\wedge\tau_{M_6}'}\!\dot{\nu}_\lambda\ud \lambda\Big|^2\Big]\\
&\le 3\kappa^{-2}(\eta^2\!+\!\eta)\!+\!3\eta\E\Big[\int_{\rho_m^\eta}^{\tau_{M_6}'}\!|\dot{\nu}_\lambda|^2\ud \lambda\Big]\le 3\kappa^{-2}(\eta^2\!+\!\eta)+3M_6\eta.
\end{align*}

The last term on the right-hand side of \eqref{eq:equitime2b} is the analogue of the one in the second line of \eqref{eq:equicm_a}, with $(\rho_m^0,X^\eta,\tau_{M_6}\!\wedge\!\rho^\eta_m)$ therein replaced by $(\rho_m^\eta+\eta,X,\tau_{M_6}'\!\wedge\!\rho^0_m)$ in \eqref{eq:equitime2b}. Thus, similar calculations lead to 
\begin{align*}
\P\Big(\rho_m^\eta+\eta+\ell \le \tau_{M_6}'\wedge\rho^0_m,\inf_{0 \leq s\le \ell}\big(X_{\rho_m^\eta+\eta+s}-X_{\rho_m^\eta+\eta}\big)>-\beta\Big)\le \e^{C_2p}\Big(2\tfrac{\beta}{\kappa\sqrt{\ell}}\Big)^{\frac{p-1}{p}}.
\end{align*}
In particular, for the change of measure in \eqref{eq:phichangemeas} we must replace $(\rho_m^0, X^\eta,\tau_{M_6})$ with $(\rho_m^\eta+\eta, X,\tau_{M_6}')$.
Plugging the equation above with \eqref{eq:claimtoprove} into \eqref{eq:equitime2b}, we obtain
\begin{align}\label{eq:equitimeupper}
u(t\!+\!\eta,x)\!-\!u(t,x)\le c_0\big(\eta^{\frac{\gamma}{2}}+\ell+\eta/\beta^2+(\beta/\sqrt{\ell})^{1-\frac{1}{p}}\big).
\end{align}

{\bf Step 4.} Let $\nu\in\cA_{t+\eta,x}^\circ$ be optimal for $u(t+\eta,x)$ and let $\rho_m^\eta=\rho_m(t+\eta,x;\nu,\kappa)$ as in \eqref{eq:rhom}. We extend $\nu$ with $\dot{\nu}_s=0$ for $s\in[\![\rho_m^\eta,\infty)\!)$. Let $\rho_m^0=\rho_m(t,x;\nu,\kappa)$ and take $w\in\cT_t^\delta$ optimal for $u(t,x)$. Set $X^{\nu,\kappa;x}\!=\!X$ for simplicity. The geometry of $\cO_m$ implies $\rho^0_m\!\ge\! \rho^\eta_m$ and by arguments as in Step 3
\begin{align*}
&u(t\!+\!\eta,x)\!-\!u(t,x)\\
&=\E\Big[R_{\rho_m^\eta}^wg_m^N(t\!+\!\eta\!+\!\rho_m^\eta,X_{\rho_m^\eta})\!-\!R_{\rho_m^\eta}^wg_m^N(t\!+\!\rho_m^\eta,X_{\rho_m^\eta})\\
&\qquad+\!\int_0^{\rho_m^\eta}\!\!R_s^w\big[h_m^N(t\!+\!\eta\!+\!s,X_s)\!-\!h_m^N(t\!+\!s,X_s)\!+\!w_s\big(g_m^N(t\!+\!\eta\!+\!s,X_s)\!-\!g_m^N(t\!+\!s,X_s)\big)\big]\ud s\\
&\qquad+\!\int_0^{\rho_m^\eta}\!\!R_s^w\big(H^{N,\eps}_m(t\!+\!\eta\!+\!s,X_s,\dot{\nu}_s)\!-\!H^{N,\eps}_m(t\!+\!s,X_s,\dot{\nu}_s)\big)\ud s\!-\!\int_{\rho_m^\eta}^{\rho_m^0}\!\!R^w_s\Theta^{N}_{\kappa,m}(t\!+\!s,X_s)\ud s\Big],
\end{align*}
where we used that $\dot{\nu}_s=0$ for $s\in[\![\rho_m^\eta,\rho_m^0]\!]$. Using the Lipschitz/H\"older property in time of $g_m^N$, $h_m^N$, $H^{N,\eps}_m$ as in \eqref{eq:equitime2ab}, we have with the same $c_0>0$ as in \eqref{eq:equitime2bc}
\begin{align}\label{eq:equicon_5m-}
u(t+\eta,x)-u(t,x)\ge-c_0\eta^\frac{\gamma}{2}-\E\Big[\int_{\rho_m^\eta}^{\rho_m^0}\!\!R^w_s\Theta^{N}_{\kappa,m}(t+s,X_s)\ud s \Big].
\end{align}

We claim here and we will prove it in Lemma \ref{lem:equicon_m2} that for $\ell,\beta>0$ (with $t+\eta+\ell<T$) 
\begin{align}\label{eq:equicon_6m}
\begin{aligned}
\int_{\rho_m^\eta}^{\rho_m^0}\!\! R^w_s\Theta^{N}_{\kappa,m}(t\!+\!s,X_s)\ud s &\leq C_\Theta(\eta\!+\!\ell)\!+\!C_\Theta T\mathds{1}_{\big\{X_{\rho_m^\eta+\eta}-X_{\rho_m^\eta}>\beta\big\}}\\
&\quad+\!C_\Theta T\mathds{1}_{\big\{\inf_{0 \leq s\leq \ell}\big(X_{\rho_m^\eta+\eta+s}-X_{\rho_m^\eta+\eta}\big)>- \beta\big\}}\mathds{1}_{\{\rho^0_m\ge \rho^\eta_m+\eta+\ell\}}.
\end{aligned}
\end{align}
Plugging \eqref{eq:equicon_6m} into \eqref{eq:equicon_5m-} yields
\begin{align}\label{eq:equicon_5m+}
\begin{aligned}
u(t\!+\!\eta,x)\!-\!u(t,x)\ge -c_0\Big[&\eta^\frac{\gamma}{2}\!+\!\ell\!+\!\P\Big( X_{\rho_m^\eta+\eta}\!-\!X_{\rho_m^\eta}\ge \beta\Big)\\
&+\!\P\Big(\inf_{0 \leq s\le \ell }\big(X_{\rho_m^\eta+\eta+s}\!-\!X_{\rho_m^\eta+\eta}\big)>\!-\beta,\rho^0_m\ge \rho^\eta_m+\eta+\ell\Big)\Big].
\end{aligned}
\end{align}

The second line in \eqref{eq:equicon_5m+} is analogue to \eqref{eq:equicm_e}, but with the supremum and $(\rho_m^\eta,X^0)$ therein replaced by the infimum and $(\rho_m^\eta+\eta,X)$. Thus, similar calculations lead to
\begin{align}\label{eq:equicon_5m2}
\P\Big(\inf_{0 \leq s\le \ell }\big(X_{\rho_m^\eta+\eta+s}-X_{\rho_m^\eta+\eta}\big)>-\beta,\rho^0_m\ge \rho^\eta_m+\eta+\ell\Big)\le \e^{C_2p}\Big(2\tfrac{\beta}{\kappa \sqrt{ \ell}}\Big)^{\frac{p-1}{p}}.
\end{align}
Recall that $X$ is uncontrolled on $[\![\rho_m^\eta,\rho_m^\eta+\eta]\!]$ so that Markov's inequality and standard estimates for SDEs yield
$\P\big( X_{\rho_m^\eta+\eta}-X_{\rho_m^\eta}\ge \beta\big)\leq \beta^{-2}\E\big[\big|X_{\rho_m^\eta+\eta}-X_{\rho_m^\eta}\big|^2\big]\le C_3\frac{\eta}{\beta^2}$,
with $C_3>0$ as in \eqref{eq:claimtoprove}.
Plugging those two bounds into \eqref{eq:equicon_5m+} yields 
$u(t+\eta,x)\!-\!u(t,x)\ge -c_0\big(\eta^\frac{\gamma}{2}\!+\!\ell\!+\!\eta/\beta^2\!+\!(\beta/\sqrt{\ell})^{1-\frac1p}\big)$. Combining with \eqref{eq:equitimeupper}
and taking $\beta=\ell=\eta^{1/4}$ yield $\big|u(t+\eta,x)-u(t,x)\big|\le c_0 \eta^{(\gamma/2)\wedge(1/8-1/8p)}$ for a constant $c_0=c_0(p,N,\kappa,\eps)>0$ independent of $m$ and $t$.
\end{proof}
We now prove the bounds in \eqref{eq:equicon_1m}, \eqref{eq:equicon_2m}, \eqref{eq:equicon_3m}, \eqref{eq:equicon_4m}, \eqref{eq:equicon_5m} and \eqref{eq:equicon_6m}. 
\begin{lemma}\label{lem:equicon_m}
The bounds in \eqref{eq:equicon_2m} and \eqref{eq:equicon_3m} hold.
\end{lemma}
\begin{proof}
Let us first prove \eqref{eq:equicon_3m}. Recall that $X^{\eta}$ and $X^0$ denote $X^{\nu,\kappa;x+\eta}$ and $X^{\nu,\kappa;x}$, respectively, and $\dot{\nu}_s=0$ for $s\in[\![\rho_m^0\wedge\rho_m^\eta,\rho_m^\eta]\!]$.
On the event $\{\rho_m^\eta\ge \rho_m^0\}$ we have for $\ell>0$
\begin{align}\label{eq:equicon_m3a}
\begin{aligned}
\int_{\rho_m^0}^{\rho_m^\eta}\!\!R^w_s\Theta^{N}_{\kappa,m}(t+s,X_s^{\eta})\ud s 
\le C_\Theta(\rho_m^\eta\!-\!\rho_m^0) \le C_\Theta\ell\!+\!C_\Theta T\mathds{1}_{\{\rho_m^\eta>\rho_m^0+\ell\}},
\end{aligned}
\end{align}
where the first inequality uses that the function $\Theta^{N}_{\kappa,m}$ is bounded from above by $C_\Theta$ (see \eqref{eq:ThetaNkmbnd}) and the function $R^w_s\leq 1$ for any $w\in\cT_t^\delta$. For any $\beta>0$ 
\begin{align}\label{eq:equicon_m3b}
\begin{aligned}
\mathds{1}_{\{\rho_m^\eta>\rho_m^0+\ell\}}&= \mathds{1}_{\big\{X_{\rho_m^0}^{\eta}-X_{\rho_m^0}^{0}\ge \beta\big\}}\mathds{1}_{\{\rho_m^\eta>\rho_m^0+\ell\}}+\!\mathds{1}_{\big\{X_{\rho_m^0}^{\eta}-X_{\rho_m^0}^{0}<\beta\big\}}\mathds{1}_{\{\rho_m^\eta>\rho_m^0+\ell\}}\\
&\leq \mathds{1}_{\big\{X_{\rho_m^0}^{\eta}-X_{\rho_m^0}^{0}\ge\beta\big\}}\!+\!\mathds{1}_{\big\{X_{\rho_m^0}^{\eta}-X_{\rho_m^0}^{0}<\beta\big\}}\mathds{1}_{\{\rho_m^\eta>\rho_m^0+\ell\}}.
\end{aligned}
\end{align}
On the event $\{\rho_m^\eta>\rho_m^0+\ell\}$ we have $\rho_m^0<T-t$ and $X^0_{\rho_m^0}=\zeta_m(t+\rho_m^0)$ because $X^\eta\ge X^0$ and both processes are bounded between $s\mapsto \zeta_m(t+s)$ and $m$. Moreover, it must be 
$X^\eta_{\rho^0_m+s}\!-\zeta_m(t+\rho^0_m+s)>0$, for all $s\in[0,\ell]$ on $\{\rho_m^\eta>\rho_m^0+\ell\}$.
The latter inequality implies that for all $s\in[0,\ell]$ we also have $X^\eta_{\rho_m^0+s}-X^\eta_{\rho_m^0}>- \beta$ on the event $\{X_{\rho_m^0}^{\eta}-X_{\rho_m^0}^{0}<\beta\}\cap\{\rho_m^\eta>\rho_m^0+\ell\}$, because 
\begin{align*}
0<X^\eta_{\rho^0_m+s}-\zeta_m(t\!+\!\rho^0_m\!+\!s)&=X^\eta_{\rho^0_m+s}-X^\eta_{\rho^0_m}+(X^\eta_{\rho^0_m}-X^0_{\rho^0_m})+\zeta_m(t\!+\!\rho^0_m)-\zeta_m(t\!+\!\rho^0_m\!+\!s)\\
&\le X^\eta_{\rho^0_m+s}-X^\eta_{\rho^0_m}+\beta,
\end{align*}
by monotonicity of $\zeta_m$.
In conclusion
\begin{align}\label{eq:inf0gamma}
\Big\{X_{\rho_m^0}^{\eta}-X_{\rho_m^0}^{0}<\beta,\,\rho_m^\eta>\rho_m^0+\ell\Big\}\subseteq \Big\{\inf_{0\leq s\leq \ell}\big(X^\eta_{\rho_m^0+s}-X^\eta_{\rho_m^0}\big)\ge- \beta,\,\rho_m^\eta>\rho_m^0+\ell\Big\}.
\end{align}
Combining inequalities \eqref{eq:equicon_m3a}, \eqref{eq:equicon_m3b} and the inclusion \eqref{eq:inf0gamma}, we get \eqref{eq:equicon_3m}.

Similarly, we prove \eqref{eq:equicon_2m}. We must recall that now $\dot{\nu}_s=0$ for $s\in[\![\rho^0_m\wedge\rho_m^\eta,\rho_m^0]\!]$. Repeating the same estimates as in \eqref{eq:equicon_m3a} and \eqref{eq:equicon_m3b} but with $\rho_m^0$ in place of $\rho_m^\eta$, we obtain for $\ell,\beta>0$
\begin{align}\label{eq:equicon_m3c}
\begin{aligned}
\int_{\rho_m^\eta}^{\rho_m^0}\!\!R^w_s\Theta^{N}_{\kappa,m}(t\!+\!s,X_s^{0})\ud s \leq C_\Theta\Big( \ell\!+\! T\mathds{1}_{\big\{X_{\rho_m^\eta}^{\eta}-X_{\rho_m^\eta}^{0}\ge \beta\big\}}\!+\! T\mathds{1}_{\{\rho_m^0>\rho_m^\eta+\ell\}}\mathds{1}_{\big\{X_{\rho_m^\eta}^{\eta}-X_{\rho_m^\eta}^{0 }<\beta\big\}}\Big).
\end{aligned}
\end{align}
On the event $\{\rho_m^0>\rho_m^\eta+\ell\}$ we have $\rho^\eta_m<T-t$ and it must be $X^\eta_{\rho_m^\eta}=m$ because $m\ge X^\eta_s\ge X^0_s\ge \zeta_m(t+s)$ for all $s\in[\![0,\rho^0_m\wedge\rho^\eta_m]\!]$. Moreover, for all $s\in[0,\ell]$ it must be $X^0_{\rho^\eta_m+s}<m$ which, on the event $\{X_{\rho_m^\eta}^{\eta}-X_{\rho_m^\eta}^{0}<\beta\}\cap\{\rho_m^0>\rho_m^\eta+\ell\}$ implies also 
\[
0>X^0_{\rho^\eta_m+s}-m=X^0_{\rho^\eta_m+s}-X^0_{\rho^\eta_m}+X^0_{\rho^\eta_m}-X^\eta_{\rho^\eta_m}\ge X^0_{\rho^\eta_m+s}-X^0_{\rho^\eta_m}-\beta,
\]
for all $s\in[0,\ell]$. Therefore, 
\begin{align}\label{eq:inf0gamma2}
\Big\{X_{\rho_m^\eta}^{\eta}-X_{\rho_m^\eta}^{0}<\beta,\,\rho_m^0>\rho_m^\eta+\ell\Big\}\subseteq \Big\{\sup_{0\leq s\leq \ell}\big(X^0_{\rho_m^\eta+s}-X^0_{\rho_m^\eta}\big)\le \beta,\,\rho_m^0>\rho_m^\eta+\ell\Big\},
\end{align}
and we get \eqref{eq:equicon_2m} using \eqref{eq:inf0gamma2} into \eqref{eq:equicon_m3c}.
\end{proof}

\begin{lemma}\label{lem:equicon_m1}
The bounds in \eqref{eq:equicon_1m} and \eqref{eq:equicon_4m} hold.
\end{lemma}
\begin{proof}
For \eqref{eq:equicon_1m} we recall $X^\eta=X^{\nu,\kappa;x+\eta}$, $X^0=X^{\nu,\kappa;x}$, $\rho^\eta_m$ and $\rho_m^0$. Using \eqref{eq:ThetaNkmbnd} and \eqref{eq:lowbnHepsm}, we get
\begin{align}\label{eq:bndgxmN}
\begin{aligned}
&\Theta^{N}_{\kappa,m}(t\!+\!s,X_s^{\eta})+\dot{\nu}_s\cdot\partial_xg_m^N(t\!+\!s,X_s^{\eta}) +H_{m}^{N,\eps}(t+s,X_s^{\eta},\dot{\nu}_s)\\
&\ge -C_\Theta-|\dot{\nu}_s|\cdot\|\partial_xg_m^N\|_{L^\infty(\overline{\cO}_m)} +\tfrac{\eps}{4}|\dot{\nu}_s|^2\\
&= -C_\Theta\!+\!\Big(\sqrt{\tfrac{\eps}{8}}|\dot{\nu}_s|\!-\!\tfrac{1}{2}\sqrt{\tfrac{8}{\eps}}\|\partial_xg_m^N\|_{L^\infty(\overline{\cO}_m)}\Big)^2\!-\!\tfrac{2}{\eps}\|\partial_xg_m^N\|_{L^\infty(\overline{\cO}_m)}^2\!+\!\tfrac{\eps}{8}|\dot{\nu}_s|^2\ge-M_5\!+\!\tfrac{\eps}{8}|\dot{\nu}_s|^2,
\end{aligned}
\end{align}
with $M_5=C_\Theta+2\eps^{-1}\Lambda_N^2$ and $\Lambda_N\ge \|\partial_x g^N_m\|_{L^\infty(\overline{\cO}_m)}$ as in \eqref{eq:alphaNMbnd}. For $\ell>0$, on the event $\{\rho^0_m\le \rho^\eta_m\}$ 
\begin{align}\label{eq:rhometa0}
\begin{aligned}
&\int_{\rho_m^0}^{\rho_m^\eta}\!\! R^w_s\big(\Theta^{N}_{\kappa,m}\!+\!\dot{\nu}_s\cdot\partial_xg_m^N\!+\!H_{m}^{N,\eps}(\cdot,\dot{\nu}_s)\big)(t\!+\!s,X_s^{\eta})\ud s \\
&\ge -\!M_5\ell\!+\!\mathds{1}_{\{\rho_m^\eta-\rho_m^0>\ell\}}\int_{\rho_m^0}^{\rho_m^\eta}\!\! R^w_s\big(\tfrac{\eps}{8}|\dot{\nu}_s|^2\!-\!M_5\big)\ud s,
\end{aligned}
\end{align}
where we used \eqref{eq:bndgxmN} in the inequality. 

For the last term on the right-hand side of \eqref{eq:rhometa0} recall $M_6= 8\e^{rT}\eps^{-1}TM_5$ and $\tau_{M_6}$ from \eqref{eq:taum6}. Since $w_\lambda=0$ for $\lambda\in(\!(\rho_m^0,T-t]\!]$ then $\e^{r s}R^w_{s}=\e^{r \rho^0_m}R^w_{\rho^0_m}$ for $s\in[\![\rho_m^0,T-t]\!]$. Using also \eqref{eq:bndgxmN}, on the event $\{\rho_m^\eta-\rho_m^0>\ell\}\cap\{\rho_m^0+\ell > \tau_{M_6}\}$ we have
\begin{align}\label{eq:equic_mK}
\begin{aligned}
\int_{\rho_m^0}^{\rho_m^\eta}\!\! R^w_s\big(\tfrac{\eps}{8}|\dot{\nu}_s|^2\!-\!M_5\big)\ud s
\ge\e^{r \rho^0_m}R^w_{\rho^0_m}\Big(\tfrac{\eps}{8}\e^{-rT}\int_{\rho_m^0}^{\tau_{M_6}}\!\!|\dot{\nu}_s|^2\ud s -M_5T\Big).
\end{aligned}
\end{align}
On $\{\rho_m^\eta\!-\!\rho_m^0>\ell\}\cap\{\rho_m^0\!+\!\ell > \tau_{M_6}\}$ it must be $\tau_{M_6}\!<\!T\!-\!t$ and thus $\tfrac{\eps}{8}\e^{-rT}\int_{\rho_m^0}^{\tau_{M_6}}\!\!|\dot{\nu}_s|^2\ud s=\tfrac{\eps}{8}\e^{-rT}M_6=M_5T$. Then, the right-hand side of \eqref{eq:equic_mK} is non-negative on $\{\rho_m^\eta-\rho_m^0>\ell\}\cap\{\rho_m^0+\ell > \tau_{M_6}\}$. Plugging \eqref{eq:equic_mK} into \eqref{eq:rhometa0} yields
\begin{align}\label{eq:rhometa0b}
\begin{aligned}
&\int_{\rho_m^0}^{\rho_m^\eta}\!\! R^w_s\big(\Theta^{N}_{\kappa,m}\!+\!\dot{\nu}_s\cdot\partial_xg_m^N\!+\!H_{m}^{N,\eps}(\cdot,\dot{\nu}_s)\big)(t\!+\!s,X_s^{\eta})\ud s \\
&\ge -\!M_5\ell\!-\!M_5T\mathds{1}_{\{\rho_m^\eta-\rho_m^0>\ell\}}\mathds{1}_{\{ \rho_m^0+\ell\le \tau_{M_6}\}}.
\end{aligned}
\end{align}
For any $\beta>0$, we then obtain
\begin{align*}
\begin{aligned}
&\int_{\rho_m^0}^{\rho_m^\eta}\!\! R^w_s\big(\Theta^{N}_{\kappa,m}\!+\!\dot{\nu}_s\cdot\partial_xg_m^N\!+\!H_{m}^{N,\eps}(\cdot,\dot{\nu}_s)\big)(t\!+\!s,X_s^{\eta})\ud s \\
&\ge -\!M_5\ell\!-\!M_5T\mathds{1}_{\{X_{\rho_m^0}^{\eta}\ge X_{\rho_m^0}^{0}\!+\beta\}}-\!M_5T\mathds{1}_{\{\rho_m^\eta-\rho_m^0>\ell\}}\mathds{1}_{\{\rho_m^0+\ell\le \tau_{M_6}\}}\mathds{1}_{\big\{X_{\rho_m^0}^{\eta}<X_{\rho_m^0}^{0}\!+\beta\big\}}\\
\end{aligned}
\end{align*}
Therefore, \eqref{eq:equicon_1m} holds, because by the same arguments as those leading to \eqref{eq:inf0gamma} we obtain
\begin{align}\label{eq:inf0gamma3}
\Big\{X_{\rho_m^0}^{\eta}<X_{\rho_m^0}^{0}+\beta,\,\rho_m^\eta>\rho_m^0+\ell\Big\}\subseteq \Big\{\inf_{0\leq s\leq \ell}\big(X^\eta_{\rho_m^0+s}-X^\eta_{\rho_m^0}\big)>- \beta,\,\rho_m^\eta>\rho_m^0+\ell\Big\}.
\end{align}

We show now that \eqref{eq:equicon_4m} holds. Repeating verbatim the first part of the proof until \eqref{eq:rhometa0b} but swapping the roles of the pairs $(\rho_m^\eta,X^{\eta})$ and $(\rho_m^0,X^{0})$ and replacing $\tau_{M_6}$ with $\tau_{M_6}'$ from \eqref{eq:taum6_2}, we have for $\ell>0$, on the event $\{\rho^0_m\ge \rho^\eta_m\}$ 
\begin{align*}
&\int_{\rho_m^\eta}^{\rho_m^0}\!\! R^w_s\big(\Theta^{N}_{\kappa,m}\!+\!\dot{\nu}_s\cdot\partial_xg_m^N\!+\!H_{m}^{N,\eps}(\cdot,\dot{\nu}_s)\big)(t\!+\!s,X_s^{0})\ud s \ge -M_5\ell\! -\! M_5T\mathds{1}_{\{\rho_m^0-\rho_m^\eta>\ell\}}\mathds{1}_{\{\rho_m^\eta+\ell \le \tau_{M_6}'\}}.
\end{align*}
For $\beta>0$, we have
\begin{align*}
&\int_{\rho_m^\eta}^{\rho_m^0}\!\! R^w_s\big(\Theta^{N}_{\kappa,m}\!+\!\dot{\nu}_s\cdot \partial_xg_m^N\!+\!H_{m}^{N,\eps}(\cdot,\dot{\nu}_s)\big)(t\!+\!s,X_s^{0})\ud s \\
&\ge -M_5\ell-M_5T\mathds{1}_{\big\{X_{\rho_m^\eta}^{\eta}\ge X_{\rho_m^\eta}^{0}+\beta\big\}} \!-\!M_5T\mathds{1}_{\{\rho_m^0-\rho_m^\eta>\ell\}}\mathds{1}_{\{\rho_m^\eta+\ell \le \tau'_{M_6}\}}\mathds{1}_{\big\{X_{\rho_m^\eta}^{\eta}<X_{\rho_m^\eta}^{0}+\beta\big\}}.
\end{align*}
Arguments as those leading to \eqref{eq:inf0gamma2} yield 
\begin{align*}
\Big\{\rho_m^0-\rho_m^\eta>\ell,\,X_{\rho_m^\eta}^{\eta}<X_{\rho_m^\eta}^{0}+\beta\Big\}\subseteq\Big\{\rho_m^0-\rho_m^\eta>\ell,\,\sup_{0\leq s\le \ell}\big(X_{\rho_m^\eta+s}^{0}-X_{\rho_m^\eta}^{0}\big)<\beta\Big\}
\end{align*}
and \eqref{eq:equicon_4m} holds.
\end{proof}
\begin{lemma}\label{lem:equicon_m2}
The bounds in \eqref{eq:equicon_5m} and \eqref{eq:equicon_6m} hold.
\end{lemma}

\begin{proof}
For the proof of \eqref{eq:equicon_5m} recall that $X=X^{\nu,\kappa;x}$ and $\rho_m^0$ and $\rho_m^\eta$ are the exit times of $(t+s,X_s)$ and $(t+\eta+s,X_s)$ from the set $\cO_m$, respectively. Due to the geometry of $\cO_m$ it holds $\rho_m^0\ge\rho_m^\eta$. Recalling \eqref{eq:bndgxmN}, for $\ell>0$ we have 
\begin{align*}
&\int_{\rho_m^\eta}^{\rho_m^0}\!\! R^w_s\big(\Theta^{N}_{\kappa,m}\!+\!\dot{\nu}_s\cdot\partial_xg_m^N\!+\!H_{m}^{N,\eps}(\cdot,\dot{\nu}_s)\big)(t\!+\!s,X_{s})\ud s\\
&\ge-M_5(\eta +\ell)\!+\!\mathds{1}_{\{\rho_m^0-\rho^\eta_m>\eta+\ell\}}\int_{\rho_m^\eta}^{\rho_m^0}\!\!R^w_s\big(\tfrac{\eps}{8}\big|\dot \nu_s\big|^2-M_5\big)\ud s.
\end{align*}
The same argument as in \eqref{eq:equic_mK}, but with $\tau_{M_6}$ therein replaced by $\tau_{M_6}'$ as in \eqref{eq:taum6_3}, yields
\begin{align*}
&\int_{\rho_m^\eta}^{\rho_m^0}\!\! R^w_s\big(\Theta^{N}_{\kappa,m}\!+\!\dot{\nu}_s\cdot\partial_xg_m^N\!+\!H_{m}^{N,\eps}(\cdot,\dot{\nu}_s)\big)(t\!+\!s,X_{s})\ud s\\
&\ge -M_5(\eta\!+\!\ell)\!-\!M_5 T\mathds{1}_{\{\rho_m^\eta+\eta+\ell \le \tau_{M_6}'\}}\mathds{1}_{\{\rho_m^0-\rho_m^\eta>\eta+\ell\}}.
\end{align*}
For $\beta>0$ we further obtain 
\begin{align}\label{eq:taumeta0b1}
\begin{aligned}
&\int_{\rho_m^\eta}^{\rho_m^0}\!\! R^w_s\big(\Theta^{N}_{\kappa,m}\!+\!\dot{\nu}_s\cdot\partial_xg_m^N\!+\!H_{m}^{N,\eps}(\cdot,\dot{\nu}_s)\big)(t\!+\!s,X_{s})\ud s\\
&\ge -M_5(\eta\!+\!\ell)\!-\!M_5T\mathds{1}_{\{\rho_m^\eta+\eta+\ell\le \tau_{M_6}'\wedge\rho^0_m\}}\Big(\mathds{1}_{\big\{X_{\rho_m^\eta+\eta}-X_{\rho_m^\eta}>\beta\big\}}\!+\!\mathds{1}_{\big\{X_{\rho_m^\eta+\eta}-X_{\rho_m^\eta}\le \beta\big\}}\Big).
\end{aligned}
\end{align}

Due to the geometry of $\cO_m$, on $\{\rho_m^0>\rho_m^\eta+\eta+\ell\}$ it must be $X_{\rho_m^\eta}=\zeta(t+\eta+\rho_m^\eta)$ and $X_{s}>\zeta_m(t+s)$ for all $s\in[\![\rho_m^\eta+\eta,\rho_m^\eta+\eta+\ell]\!]$. Then, on $\{\rho_m^0>\rho_m^\eta+\eta+\ell\}\cap\{X_{\rho_m^\eta+\eta}-X_{\rho_m^\eta}\le\beta\}$ we have $X_{s}-X_{\rho_m^\eta+\eta}\ge X_{s}-X_{\rho_m^\eta}-\beta>\zeta_m(t+s)-\zeta_m(t+\rho_m^\eta+\eta)-\beta\ge -\beta$ for all $s\in[\![\rho_m^\eta+\eta,\rho_m^\eta+\eta+\ell]\!]$, where we used that $\zeta_m$ is non-decreasing in the third inequality. Therefore, 
\begin{align}\label{eq:taumeta0b2}
\begin{aligned}
&\Big\{X_{\rho_m^\eta+\eta}\!-\!X_{\rho_m^\eta}\le \beta,\,\rho_m^0\!>\!\rho_m^\eta\!+\!\eta\!+\!\ell\Big\}\\
&\subseteq\! \Big\{\inf_{\eta\leq s\leq \eta +\ell}\big(X_{\rho_m^\eta+s}\!-\!X_{\rho_m^\eta+\eta}\big)\ge- \beta,\,\rho_m^0\!>\!\rho_m^\eta\!+\!\eta\!+\!\ell\Big\}.
\end{aligned}
\end{align}
Plugging the latter into \eqref{eq:taumeta0b1} we obtain \eqref{eq:equicon_5m}.

To prove \eqref{eq:equicon_6m} we recall $\rho^\eta_m\le \rho^0_m$ and $\dot{\nu}_s=0$ for $s\in[\![\rho_m^\eta,\rho_m^0]\!]$. For $\ell,\beta>0$, using \eqref{eq:ThetaNkmbnd} we have
\begin{align*}
&\int_{\rho_m^\eta}^{\rho_m^0}\!\! R^w_s\Theta^{N}_{\kappa,m}(t\!+\!s,X_{s})\ud s\le C_\Theta \big(\rho^0_m-\rho^\eta_m\big)\big(\mathds{1}_{\{\rho^0_m-\rho^\eta_m\le \eta+\ell\}}+\mathds{1}_{\{\rho^0_m-\rho^\eta_m\ge \eta+\ell\}}\big) \\
&\le C_\Theta(\eta\!+\!\ell) \!+\!C_\Theta T\mathds{1}_{\{\rho_m^0>\rho_m^\eta+\eta+\ell\}}\Big(\mathds{1}_{\big\{X_{\rho_m^\eta+\eta}-X_{\rho_m^\eta}>\beta\big\}}\!+\!\mathds{1}_{\big\{X_{\rho_m^\eta+\eta}-X_{\rho_m^\eta}\le \beta\big\}}\Big).
\end{align*}
An inclusion as \eqref{eq:taumeta0b2} holds here as well, which concludes the proof of \eqref{eq:equicon_6m}.
\end{proof}

\subsection{Penalised problem on unbounded domain}

Combining Proposition \ref{prop:gradbndUa} with Theorem \ref{thm:equic_m}, we obtain existence and uniqueness of the solution to the penalised problem on $\overline\cO$. We denote by $C^{1,2;\gamma}_{Loc}(\overline\cO)$ the subset of $C(\overline\cO)\cap C^{1,2}(\cO)$ of functions whose derivatives are $\gamma$-H\"older continuous in any compact $\cK\subset\cO$. Clearly $C^{1,2;\gamma}_{Loc}(\overline\cO)\subset W^{1,2;p}_{\ell oc}(\cO)$.

\begin{theorem}\label{thm:highreguued}
There exists $u=u^{N,\eps,\delta}_{\kappa}\in C^{1,2;\gamma}_{Loc}(\overline\cO)$, for any $p\in(1,\infty)$ and $\gamma\in(0,1)$ as in Assumption \ref{ass:gen2}, that solves:
\begin{align}\label{eq:pdeinRdeps}
\begin{cases}(\partial_t\!+\!\cL_\kappa \!-\!r)u=\!-h^N\!-\!\frac{1}{\delta}\big(g^N \!-\!u\big)^+\!+\psi_\eps\big(|\partial_x u|^2\!-\!(\bar{\alpha}^N)^2\big), &\text{on $\cO$}, \\
u(T,x)=g^N(T,x), &\text{for all $x\in[0,\infty)$},\\
u(t,0)=g^N(t,0), &\text{for all $t\in[0,T]$}.
\end{cases}
\end{align}
Moreover, $0\le u^{N,\eps,\delta}_\kappa(t,x)\le K_3(1+|x|^2)$ for $(t,x)\in\overline\cO$ with $K_3>0$ as in Lemma \ref{lem:polygrow}.
\end{theorem}
\begin{proof}
We extend $u^{N,\eps,\delta}_{\kappa,m}$ to be equal to $g^N$ on the set $\overline\cO\setminus \cO_m$. So all functions $(u^{N,\eps,\delta}_{\kappa,m})_{m\in\N}$ are defined on $\overline\cO$ and continuous. The family $(u^{N,\eps,\delta}_{\kappa,m})_{m\in\N}$ is equi-bounded and equi-continuous on any compact subset of $\overline\cO$ by Lemma \ref{lem:polygrow} and by Theorem \ref{thm:equic_m}, respectively. For any $n\in\N$ we can extract a subsequence $(u^{N,\eps,\delta}_{\kappa,m^n_j})_{j\in\N}$ that converges uniformly on $[0,T]\times[0,n]$ to a function $u^{N,\eps,\delta}_{\kappa;n}$. Such function is continuous on $[0,T]\times[0,n]$ and $u^{N,\eps,\delta}_{\kappa;n}=g^N$ on $([0,T]\times\{0\})\cup(\{T\}\times[0,n])$. Up to extracting a further subsequence, $(u^{N,\eps,\delta}_{\kappa,m^n_j})_{j\in\N}$ also converges to $u^{N,\eps,\delta}_{\kappa;n+1}$ on $[0,T]\times[0,n+1]$. Hence, by construction $u^{N,\eps,\delta}_{\kappa;n}=u^{N,\eps,\delta}_{\kappa;n+1}$ on $[0,T]\times[0,n]$ for every $n\in\N$. Thus, we define $u^{N,\eps,\delta}_{\kappa}$ on $\overline\cO$ by taking, e.g., $u^{N,\eps,\delta}_{\kappa}=u^{N,\eps,\delta}_{\kappa;n}$ on $[0,T]\times[0,n]$.

Fix $n\in\N$ and consider $\cK_n=[0,T-\frac1n]\times[\frac1n,n]$. The subsequence $(u^{N,\eps,\delta}_{\kappa,m^n_j})_{j\in\N}$ constructed above is bounded in $W^{1,2;p}(\cK_n)$, thanks to Proposition \ref{prop:gradbndUa}. Therefore, we can extract a further subsequence, which for simplicity we still denote by $(u^{N,\eps,\delta}_{\kappa,m^n_j})_{j\in\N}$, such that for $j\to\infty$ 
\begin{align}\label{eq:weakconv}
\begin{array}{l}
u^{N,\eps,\delta}_{\kappa,m^n_j} \to u^{N,\eps,\delta}_{\kappa}\quad\text{and}\quad \partial_x u^{N,\eps,\delta}_{\kappa,m^n_j} \to \partial_x u^{N,\eps,\delta}_{\kappa}\quad\text{ in } C^{\alpha}(\overline{\cK}_n), \\[+5pt]
\partial_t u^{N,\eps,\delta}_{\kappa,m^n_j} \to \partial_t u^{N,\eps,\delta}_{\kappa}\quad\text{and}\quad \partial_{xx}u^{N,\eps,\delta}_{\kappa,m^n_j} \to \partial_{xx}u^{N,\eps,\delta}_{\kappa}\quad\text{weakly in }L^p(\cK_n),
\end{array}
\end{align}
where $u^{N,\eps,\delta}_{\kappa}$ is the same limit constructed in the first paragraph. This shows that, for any $n\in\N$, $u^{N,\eps,\delta}_{\kappa}\in W^{1,2;p}(\cK_n)$ and $\partial_x u^{N,\eps,\delta}_{\kappa}\in C^{\alpha}(\overline{\cK}_n)$. It follows that $u^{N,\eps,\delta}_{\kappa}\in W^{1,2;p}_{\ell oc}(\cO)$ and it is not hard to show that $u^{N,\eps,\delta}_{\kappa}$ is a strong solution of \eqref{eq:pdeinRdeps} with boundary conditions taken continuously from the interior of the domain. Thus, we can lift the regularity to $u^{N,\eps,\delta}_{\kappa}\in C^{1,2;\gamma}_{Loc}(\overline\cO)$ by a standard procedure (see, e.g., the proof of \cite[Thm.\ 3]{bovo2022variational}). 

Finally, the quadratic growth condition follows immediately from Lemma \ref{lem:polygrow}.
\end{proof}

We now give a probabilistic representation for $u^{N,\eps,\delta}_{\kappa}$ analogue of \eqref{eq:probrap} but on unbounded domain. Such representation implies that $u^{N,\eps,\delta}_\kappa$ is indeed the unique solution of \eqref{eq:pdeinRdeps}. Let 
\begin{align}\label{eq:hmltn2}
H^{N,\eps}(y)\coloneqq \sup_{p\in\mathbb{R}}\big\{y p-\psi_\eps\big(|p|^2- (\alpha^{N})^2\big)\big\}
\end{align}
and, for $\nu\in\cA^\circ_{t,x}$,
\begin{align}\label{eq:rhoknu}
\rho=\rho(t,x;\nu,\kappa)=\inf\big\{s\ge 0\,\big|\, X_s^{\nu,\kappa;x}\le 0\big\}\wedge(T-t). 
\end{align}
For arbitrary $(\nu,w)\in\cA^\circ_{t,x}\times \cT^\delta_t$, let us denote by $\cJ^{N,\kappa,\eps,\delta}_{t,x}(\nu,w)$ a payoff analogue of \eqref{eq:Jpen} but with $g_m^N,h_m^N,H^{N,\eps}_{m}$ and $\rho_{m}$ therein replaced by $g^N,h^N,H^{N,\eps}$ and $\rho$, respectively. 

\begin{proposition}\label{lem:prbraprRD}
Let $u^{N,\eps,\delta}_{\kappa}$ be a solution of \eqref{eq:pdeinRdeps} as in Theorem \ref{thm:highreguued}. Then 
\begin{align}\label{eq:valueunbdd}
u^{N,\eps,\delta}_{\kappa}(t,x)=\inf_{\nu\in\cA^\circ_{t,x}}\sup_{w\in\cT^\delta_t}\cJ^{N,\kappa,\eps,\delta}_{t,x}(\nu,w)=\sup_{w\in \cT^\delta_t}\inf_{\nu\in\cA^\circ_{t,x}}\cJ^{N,\kappa,\eps,\delta}_{t,x}(\nu,w).
\end{align}
The controlled SDE \eqref{eq:SDEcntrll} admits a unique solution $(X_{s\wedge\rho}^{\nu^*})_{s\in[0,T-t]}$ when $\nu=\nu^*$ is defined by 
\begin{align}\label{eq:dfnoptw*2b}
\dot{\nu}_s^*\coloneqq -2\psi_{\eps}'\Big(|\partial_x u^{N,\eps,\delta}_{\kappa}(t+s,X_s^*)|^2- (\bar{\alpha}^N)^2\Big)\partial_x u^{N,\eps,\delta}_{\kappa}(t+s,X_s^*),\quad s\in[\![0,\rho)\!),
\end{align}
$\dot \nu^*_s=0$ for $s\in[\![\rho,T\!-\!t]\!]$, and $\nu^*$ is optimal for the minimiser. For any $\nu\in\cA_{t,x}^\circ$, the process
\begin{align}\label{eq:dfnoptw*2a}
w^*_s=w^*_s(\nu)\coloneqq\tfrac1\delta \mathds{1}_{\big\{u^{N,\eps,\delta}_{\kappa}(t+s,X_s^{\nu,\kappa})\leq g^N(t+s,X_s^{\nu,\kappa})\big\}},
\end{align}
is optimal for the maximiser.
\end{proposition}
\begin{proof}
The proof follows ideas from \cite[Prop.\ 7]{bovo2022variational} but the admissibility and optimality of $\nu^*$ require some refinements to the original arguments. Again we simplify notations by taking $u=u^{N,\eps,\delta}_{\kappa}$ and for arbitrary $\nu\in\cA^\circ_{t,x}$ setting $X^\nu=X^{\nu,\kappa;x}$. For $n,m, k\in\N$, let $\theta_{n,m}=\beta_m\wedge\lambda_n$, where $\beta_m=\inf\{s\ge 0\,|\, X^{\nu}_s\ge m\}$ and $\lambda_n=\inf\{s\ge 0\,|\, X^{\nu}_s\le n^{-1}\}$, and set $T_t=T-t$ and $T_k=T_t\!-\!k^{-1}$. Since $u\in C^{1,2;\gamma}([0,T_k]\times[n^{-1},m])$, by It\^o's formula 
\begin{align}\label{eq:dyn0}
\begin{aligned}
u(t,x)=\E_x\Big[&R^w_{\theta_{n,m}\wedge T_k}u(t+\theta_{n,m}\wedge T_k,X^\nu_{\theta_{n,m}\wedge T_k})\\
&+\int_0^{\theta_{n,m}\wedge T_k}\!R^w_s\big[-(\partial_t+\cL_\kappa-r)u+w_su-\partial_xu\cdot\dot{\nu}_s \big](t+s,X_s^\nu)\ud s\Big].
\end{aligned}
\end{align}
For $w_s=w^*_s$ we have $\frac{1}{\delta}(g^N-u)^++w^*_su= w^*_sg^N$ and using that $u$ solves \eqref{eq:pdeinRdeps} and 
\begin{align}\label{eq:Hineq}
-\partial_xu(t+s,X_s^\nu)\cdot\dot{\nu}_s-\psi_\eps\big(|\partial_xu(t+s,X_s^\nu)|^2-(\bar{\alpha}^N)^2\big)\le H^{N,\eps}(\dot{\nu}_s),\quad s\in[\![0,\rho)\!),
\end{align}
we obtain the upper bound
\begin{align*}
\begin{aligned}
u(t,x)\le \E\Big[R^{w^*}_{\theta_{n,m}\wedge T_k}u(t\!+\!\theta_{n,m}\wedge T_k,X^\nu_{\theta_{n,m}\wedge T_k})\!+\!\int_0^{\theta_{n,m}\wedge T_k}\!\!\!R^{w^*}_s\big[h^N\!+\!w^*_sg^N\!+\!H^{N,\eps}(\dot{\nu}_s)\big](t\!+\!s,X_s^\nu)\ud s\Big].
\end{aligned}
\end{align*}
Since $\nu\in\cA^\circ_{t,x}$, it is not hard to check that $\E[\sup_{0\le s\le T-t}|X^\nu_s|^2]<\infty$ by standard SDE estimates.
Letting $n,m,k\to\infty$ and using dominated convergence (justified by quadratic growth of $u$) and continuity of $u$ it is not difficult to verify that 
$u(t,x)\le \sup_{w\in\cT^\delta_t}\inf_{\nu\in\cA^\circ_{t,x}}\cJ^{N,\kappa,\eps,\delta}_{t,x}(\nu,w)$.

For the reverse inequality, we repeat the steps above with $\nu=\nu^*$ and arbitrary $w$. Notice that the dynamics $X^*$ is well-posed on $[\![0,\theta_{n,m}]\!]$ (cf.\ \cite[Thm.\ 2.5.7]{krylov1980controlled}). First we have \eqref{eq:dyn0} with $(w,\nu^*)$. Then we notice that equality holds in \eqref{eq:Hineq} whereas $\frac{1}{\delta}(g^N-u)^++w_s u\ge w_sg^N$. Thus, letting also $k\to \infty$ we obtain 
\begin{align}\label{eq:uNkde_opt}
\begin{aligned}
u(t,x)\ge \E\Big[&R^{w}_{\theta_{n,m}\wedge T_t}u(t\!+\!\theta_{n,m}\wedge T_t,X^*_{\theta_{n,m}\wedge T_t})\\
&+\!\int_0^{\theta_{n,m}\wedge T_t}\!\!R^{w}_s\big[h^N\!+\!w_sg^N\!+\!H^{N,\eps}(\dot{\nu}^*_s) \big](t\!+\!s,X_s^*)\ud s\Big].
\end{aligned}
\end{align}
It remains to justify taking limits as $n,m\to\infty$. The sequence $(\theta_{n,m})_{n,m}$ is increasing in both indexes so the limits $\theta_{n,\infty}\coloneqq\lim_{m\to\infty}\theta_{n,m}$ and $\theta_{\infty}\coloneqq\lim_{n\to\infty}\theta_{n,\infty}$ are well-defined with $\theta_{n,\infty}\le \lambda_n$. By \cite[Prop.\ 7, Eqs.\ (98)-(99)]{bovo2022variational}, there is $c(\eps)>0$ such that 
\begin{align}\label{eq:nubound}
\E[(\nu^{*,\pm}_{\theta_\infty\wedge T_t})^2]\le c(\eps)(1+|x|^2),
\end{align}
where the limits
$\nu^{*,\pm}_{\theta_\infty\wedge T_t}=\lim_{n,m\to\infty}\nu^{*,\pm}_{\theta_{n,m}\wedge T_t}$ are well-defined by monotonicity. 

In order to extend the dynamics of $X^*$ to $[\![0,\rho]\!]$, we have by classical SDE estimates
\begin{align*}
\P(\beta_m<T_t\wedge\lambda_n)&=\P\Big(\sup_{s\in[0,\theta_{n,m}\wedge T_t]}X^*_s\ge m\Big)\le \frac{1}{m^2}\E\Big[\sup_{s\in[0,\theta_{n,m}\wedge T_t]}|X^*_s|^2\Big]\\
&\le \frac{C}{m^2}\big(1+|x|^2+\E[(\nu^{*,+}_{\theta_{n,m}\wedge T_t})^2+(\nu^{*,-}_{\theta_{n,m}\wedge T_t})^2]\big)\le \frac{1}{m^2}c'(\eps)(1+|x|^2),
\end{align*}
where the final inequality holds by \eqref{eq:nubound} with some constant $c'(\eps)>0$.
Sending $m\to\infty$, we have that $\P(\beta_m<T_t\wedge\lambda_n)\uparrow \P(\beta_\infty\le T_t\wedge\lambda_n)=0$ so that $X^*$ does not explode in finite time (i.e., $\theta_{n,\infty}=\lambda_n$, $\P$-a.s.). Therefore, the process $(X^*_{s\wedge \lambda_n})_{s\in[0,T]}$ is well-defined on the stochastic interval $[\![0,\lambda_n\wedge T_t]\!]=[\![0,\theta_{n,\infty}\wedge T_t]\!]$. Next we let $n\to\infty$ so that $\theta_{n,\infty}\uparrow \theta_\infty$. Since $\theta_\infty\wedge T_t=\rho$, then $\nu^*$ is finite on $[\![0,\rho]\!]$ by \eqref{eq:nubound}. When $\theta_\infty(\omega)> T_t$ we have $\lim_{n\to\infty}X^*_{\theta_{n,\infty}\wedge T_t}(\omega)=X^*_{T_t}(\omega)>0$ and when $\theta_\infty(\omega)\le T_t$ we have $\theta_{n,\infty}(\omega)\le T_t$ for all $n\in\N$ so that $X^*_{\theta_{n,\infty}}(\omega)=n^{-1}$ and $X^*_{\theta_\infty}(\omega)=0$. Then, $\lim_{n\to\infty}X^*_{\theta_{n,\infty}\wedge T_t}=X^*_\rho$, $\P$-a.s., and letting $n\to\infty$ in \eqref{eq:uNkde_opt}, dominated convergence and $u(t\!+\!\rho,X^*_{\rho})=g^N(t\!+\!\rho,X^*_{\rho})$ yield
\begin{align*}
\begin{aligned}
u(t,x)&\ge \E\Big[R^w_{\rho}u(t\!+\!\rho,X^*_{\rho})\!+\!\int_0^{\rho}\!\!R^w_s\big[h^N\!+\!w_sg^N\!+\!H^{N,\eps}(\dot{\nu}^*_s) \big](t\!+\!s,X_s^*)\ud s\Big]=\cJ_{t,x}^{N,\kappa,\eps,\delta}(\nu^*,w).
\end{aligned}
\end{align*}
By the arbitrariness of $w$, we have $u(t,x)\ge\inf_{\nu\in\cA_{t,x}^\circ}\sup_{w\in\cT_t^\delta} \cJ_{t,x}^{N,\kappa,\eps,\delta}(\nu,w)$, which proves \eqref{eq:valueunbdd}. Optimality of the pair $(w^*,\nu^*)$ is easily deduced upon noticing that all inequalities become equalities when we repeat the arguments from \eqref{eq:dyn0} with $(w,\nu)=(w^*,\nu^*)$.
\end{proof}

We close this section recalling bounds for the penalty terms in the PDE \eqref{eq:pdeinRdeps}, uniformly in $N$, $\kappa$, $\eps$ and $\delta$. The proofs of the next two lemmas are omitted and can be found in \cite{bovo2022variational}.

\begin{lemma}[{\cite[Lem.\ 4 and 5]{bovo2022variational}}]\label{lem:bndobstpen2}
We have $\tfrac{1}{\delta}(g^N-u^{N,\eps,\delta}_{\kappa})^+\le K_2$ and $\partial_t u^{N,\eps,\delta}_{\kappa} \leq K_4$ on $\cO$, where $K_4>0$ depends only on $K_0$ and $K_2$ from Assumption \ref{ass:gen2}.
\end{lemma}

The lemma above is used to prove the next one.
\begin{lemma}[{\cite[Prop.\ 5, Lem.\ 6 and Thm.\ 4]{bovo2022variational}}]\label{lem:W12pbound}
For any compact $\cK\subset\cO$ and $p\in(1,\infty)$, there are constants $M_7=M_7(\cK)>0$ and $ M_8=M_8(\cK,p)>0$, independent of $N$, $\kappa$, $\eps$ and $\delta$, such that 
\begin{align}\label{eq:gradpeneq}
\| \partial_x u^{N,\eps,\delta}_{\kappa}\|_{L^\infty(\cK)}\leq M_7,\quad\big\|\psi_\eps\big(|\partial_x u^{N,\eps,\delta}_{\kappa}|^2-(\bar{\alpha}^N)^2\big)\big\|_{L^\infty(\cK)}\le M_7, 
\end{align} 
and $\|u^{N,\eps,\delta}_{\kappa}\|_{W^{1,2;p}(\cK)}\leq M_8$.
\end{lemma}

\section{Value of the original game}\label{sec:final}

In this section, we prove Theorem \ref{thm:usolvar}. First we prove it under {\bf A.1} and then under {\bf A.2}. 

\subsection{Value of the game with bounded data}\label{sec:vkN}
Lemma \ref{lem:W12pbound} guarantees boundedness of the family $(u^{N,\eps,\delta}_\kappa)_{\delta,\eps,\kappa,N}$ in $W^{1,2;p}_{\ell oc}(\cO)$. Moreover, the modulus of continuity in Theorem \ref{thm:equic_m} is independent of $\delta$. Then, repeating analogous limiting considerations as in the proof of Theorem \ref{thm:highreguued} we obtain a function $u^{N,\eps}_\kappa\in C(\overline \cO)\cap W^{1,2;p}_{\ell oc}(\cO)$ defined as
\begin{align}\label{eq:limu1}
u^{N,\eps}_\kappa\coloneqq \lim_{n\to\infty}u^{N,\eps,\delta_n}_\kappa,
\end{align}
where $(\delta_n)_{n\in\N}$ is a sequence (possibly depending on $N$, $\kappa$, $\eps$) converging to zero.
For $(t,x)\in\overline\cO$, let 
\begin{align*}
\cJ_{t,x}^{N,\kappa,\eps}(\tau,\nu)\coloneqq \E_x\Big[\e^{-r(\tau\wedge\rho)}g^N(t\!+\!\tau\wedge\rho,X_{\tau\wedge\rho}^{\nu,\kappa})\!+\!\int_0^{\tau\wedge\rho}\!\e^{-rs}\big(h^N(t\!+\!s,X_s^{\nu,\kappa})\!+\!H^{N,\eps}(\dot{\nu}_s)\big)\,\ud s\Big],
\end{align*}
with $(\tau,\nu)\in\cT_t\times \cA_{t,x}^\circ$ and $\rho$ as in \eqref{eq:rhoknu}.
 
\begin{lemma}\label{lem:uepsmuN}
We have $u^{N,\eps}_{\kappa}(t,x)\!=\!\sup_{\tau\in\cT_t}\inf_{\nu\in\cA_{t,x}^\circ}\cJ_{t,x}^{N,\kappa,\eps}(\tau,\nu)\!=\!\inf_{\nu\in\cA_{t,x}^\circ}\sup_{\tau\in\cT_t}\cJ_{t,x}^{N,\kappa,\eps}(\tau,\nu)$, for $(t,x)\in\overline\cO$. The pair $(\tau^{N,\kappa,\eps},\nu^{N,\kappa,\eps})$ with $\nu^{N,\kappa,\eps}$ as in \eqref{eq:dfnoptw*2b}, but with $u^{N,\eps}_\kappa$ instead of $u^{N,\eps,\delta}_\kappa$, and
\begin{align*}
\tau^{N,\kappa,\eps}=\tau^{N,\kappa,\eps}(t,x;\nu,\kappa)=\inf\big\{s\ge0\,\big|\, u^{N,\eps}_{\kappa}(t+s,X_s^{\nu,\kappa;x})=g^N(t+s,X_s^{\nu,\kappa;x})\big\},
\end{align*}
is a saddle point.
\end{lemma}
\begin{proof}
The proof follows very closely the one of \cite[Thm.\ 6]{bovo2022variational} and it is similar to our proof of Proposition \ref{lem:prbraprRD}. Therefore we only provide an outline. As in \eqref{eq:dyn0} and \eqref{eq:Hineq} with $w_s\equiv 0$ we obtain
\begin{align*}
\begin{aligned}
u^{N,\eps,\delta}_\kappa(t,x)&\le \E\Big[\e^{-r(\theta_{n,m}\wedge \tau)}u^{N,\eps,\delta}_\kappa(t\!+\!\theta_{n,m}\wedge \tau,X^\nu_{\theta_{n,m}\wedge \tau})\\
&\qquad\!+\!\int_0^{\theta_{n,m}\wedge \tau}\!\!\e^{-rs}\big[h^N\!+\!\tfrac1\delta(g^N\!-\!u^{N,\eps,\delta}_\kappa)^+\!+\!H^{N,\eps}(\dot{\nu}_s)\big](t\!+\!s,X_s^\nu)\ud s\Big],
\end{aligned}
\end{align*}
for any $\tau\in\cT_t$. Letting first $\delta \to 0$ along the sequence $(\delta_n)_{n\in\N}$ defined in \eqref{eq:limu1}, and using Lemma \ref{lem:bndobstpen2}, and then letting $n,m\to\infty$ we obtain $u^{N,\eps}_\kappa(t,x)\le \cJ^{N,\kappa,\eps}_{t,x}(\tau^{N,\kappa,\eps},\nu)$ (cf.\ details in \cite[p.\ 31]{bovo2022variational}). Thus $u^{N,\eps}_\kappa$
is smaller than the lower value of the game. 

For the lower bound, the PDE for $u^{N,\eps,\delta}_\kappa$ and \eqref{eq:dyn0} with $w_s\equiv 0$, $\nu\in\cA^\circ_{t,x}$, $\tau\in\cT_t$, yield 
\begin{align}\label{eq:uNkeps,delta}
\begin{aligned}
u^{N,\eps,\delta}_\kappa(t,x)\ge \E_x\Big[&\e^{-r(\theta_{n,m}\wedge \tau)}u^{N,\eps,\delta}_\kappa(t\!+\!\theta_{n,m}\wedge \tau,X^\nu_{\theta_{n,m}\wedge \tau})\\
&+\int_0^{\theta_{n,m}\wedge \tau}\!\!\!\e^{-rs}\big[h^N\!-\!\psi_\eps\big(|\partial_x u^{N,\eps,\delta}_\kappa|^2\!-\!(\bar \alpha^N)^2\big)\!-\!\partial_x u^{N,\eps,\delta}_\kappa\cdot\dot{\nu}_s \big](t\!+\!s,X_s^\nu)\ud s\Big],
\end{aligned}
\end{align}
where the inequality holds because we dropped the term $\frac1\delta(g^N-u^{N,\eps,\delta}_\kappa)^+$. Letting $\delta\to 0$ along the sequence $(\delta_n)_{n\in\N}$ and using Lemma \ref{lem:W12pbound} we obtain the same expression as in \eqref{eq:uNkeps,delta} but with $u^{N,\eps,\delta}_\kappa$ replaced everywhere by $u^{N,\eps}_\kappa$. Then we choose $\nu^{N,\kappa,\eps}$ as in \eqref{eq:dfnoptw*2b}, but replacing $u^{N,\eps,\delta}_\kappa$ with $u^{N,\eps}_\kappa$, and notice that $X^{\nu^{N,\kappa,\eps}}$ is well-posed because \eqref{eq:nubound} continues to hold and the rest of the argument is as in Proposition \ref{lem:prbraprRD}. Moreover, we notice that Lemma \ref{lem:bndobstpen2} implies $u^{N,\eps}_\kappa\ge g^N$. Combining these two observations yields $u^{N,\eps}_\kappa(t,x)\ge \cJ^{N,\kappa,\eps}_{t,x}(\tau,\nu^{N,\kappa,\eps})$. Hence, $u^{N,\eps}_\kappa$ is greater than the upper value of the game.
Thus the value exists and the pair $(\tau^{N,\kappa,\eps},\nu^{N,\kappa,\eps})$ is a saddle point. 
\end{proof}

Next we show equi-continuity of the family $(u^{N,\eps}_{\kappa})_{\eps\in(0,1)}$. 
\begin{proposition}\label{prop:equiuepsmuN}
There are $c_0=c_0(p,N,\kappa)>0$ and $\lambda=\lambda(N,\kappa)>0$ such that for all $\eps\in(0,1)$
\begin{align*}
|u^{N,\eps}_{\kappa}(t,x)-u^{N,\eps}_{\kappa}(s,y)|\le c_0|x-y|^{\frac{p-1}{4p}}+\lambda|t-s|
\end{align*}
for any $s,t\in[0,T]$ and $x,y\in[0,\infty)$ with $|s-t|\vee|x-y|\le 1$.
\end{proposition}

\begin{proof}
Fix $\eps\in(0,1)$. Let $(t,x)\in\cO$ and $\eta\in(0,1)$. Set $u=u^{N,\eps}_{\kappa}$ and $\cJ_{t,x}=\cJ_{t,x}^{N,\kappa,\eps}$ for simplicity. We first prove equi-continuity in space and then we prove equi-continuity in time. 

{\bf Step 1}. Let $\tau^\eta\in\cT_t$ be optimal for $u(t,x+\eta)$ and let $\nu\in\cA_{t,x}^\circ$ be optimal for $u(t,x)$. Recall $\rho=\rho(t,x;\nu,\kappa)$ as in \eqref{eq:rhoknu} and extend $\nu_s$ as $\dot{\nu}_s=0$ for $s\in[\![\rho,\infty)\!)$. Denote $\rho^\eta=\rho(t,x+\eta;\nu,\kappa)$ and notice that $\rho\le \rho^\eta$, $\P$-a.s. For simplicity, we denote $X^{\nu,\kappa;x}=X$ and $X^{\nu,\kappa;x+\eta}=X^\eta$. Since $u(t,x\!+\!\eta)\le \cJ_{t,x+\eta}(\nu,\tau^\eta)$ and $u(t,x)\ge \cJ_{t,x}(\nu,\tau^\eta)$, by Dynkin's formula on $[\![\rho,\rho^\eta]\!]$ (cf.\ \eqref{eq:gmNDynk}) we have 
\begin{align*}
u(t,x\!+\!\eta)&\le\E\Big[\e^{-r(\tau^\eta\wedge\rho)}g^N(t\!+\!\tau^\eta\wedge\rho,X^\eta_{\tau^\eta\wedge\rho})\!+\!\!\int_0^{\tau^\eta\wedge\rho}\!\!\e^{-rs}\big(h^N(t\!+\!s,X_s^\eta)\!+\!H^{N,\eps}(\dot{\nu}_s)\big)\ud s\\
&\qquad+\!\!\int_{\tau^\eta\wedge\rho}^{\tau^\eta\wedge\rho^\eta}\!\!\e^{-rs}\Theta_{\kappa}^N(t\!+\!s,X_s^\eta)\ud s\Big]
\end{align*}
and
\begin{align*}
u(t,x)\ge \E\Big[\e^{-r(\tau^\eta\wedge\rho)}g^N(t\!+\!\tau^\eta\wedge\rho,X_{\tau^\eta\wedge\rho})\!+\!\!\int_0^{\tau^\eta\wedge\rho}\!\!\e^{-rs}\big(h^N(t\!+\!s,X_s)\!+\!H^{N,\eps}(\dot{\nu}_s)\big)\ud s\Big].
\end{align*}
Combining the two inequalities, we obtain
\begin{align}
\begin{aligned}
u(t,x\!+\!\eta)\!-\!u(t,x)
&\le \E\Big[\e^{-r(\tau^\eta\wedge \rho)}\big(g^N(t\!+\!\tau^\eta\!\wedge\! \rho,X_{\tau^\eta\wedge \rho}^\eta)\!-\!g^N(t\!+\!\tau^\eta\!\wedge\! \rho,X_{\tau^\eta\wedge \rho})\big) \!\\
&\qquad+\!\!\int_0^{\tau^\eta\wedge \rho}\!\!\e^{-rs}\big(h^N(t\!+\!s,X_s^\eta)\!-\!h^N(t\!+\!s,X_s)\big)\ud s\! +\!\!\int_{\tau^\eta\wedge \rho}^{\tau^\eta\wedge \rho^\eta}\!\!\e^{-rs}\Theta^N_\kappa(t+s,X_s^\eta)\ud s\Big].
\end{aligned}
\end{align}
Since $g^N$ and $h^N$ are Lipschitz with constants $\bar{\alpha}^N$ and $L_N$, respectively, and $r\ge0$, we have
\begin{align}\label{eq:equic_eps_0}
\begin{aligned}
u(t,x\!+\!\eta)\!-\!u(t,x)&\le \E\Big[\bar{\alpha}^N|X_{\tau^\eta\wedge \rho}^\eta\!-\!X_{\tau^\eta\wedge \rho}|\!+\!L_N\!\int_0^{\tau^\eta\wedge \rho}\!\!|X_s^\eta\!-\!X_s|\ud s \!+\!\int_{\tau^\eta\wedge \rho}^{\tau^\eta\wedge \rho^\eta}\!\!\e^{-rs}\Theta^N_\kappa(t\!+\!s,X_s^\eta)\ud s\Big]\\
&\le(\bar{\alpha}^N\!+\!TL_N)\E\Big[\sup_{s\in[0,\rho]}|X_s^\eta\!-\!X_s|^2\Big]^{\frac{1}{2}}\!+\!\E\Big[\int_{\tau^\eta\wedge \rho}^{\tau^\eta\wedge \rho^\eta}\!\!\e^{-rs}\Theta^N_\kappa(t\!+\!s,X_s^\eta)\ud s\Big].
\end{aligned}
\end{align}
For $\ell,\beta>0$ and $C_\Theta>0$ as in \eqref{eq:ThetaNkmbnd}, similar calculations as in \eqref{eq:equicon_3m} from Lemma \ref{lem:equicon_m} lead to (recall that here $\rho^\eta\ge \rho$, $\P$-a.s.)
\begin{align}\label{eq:equic_eps_1}
\begin{aligned}
&\E\Big[\int_{\tau^\eta\wedge \rho}^{\tau^\eta\wedge \rho^\eta}\!\e^{-rs}\Theta^N_\kappa(t\!+\!s,X_s^\eta)\ud s\Big]\\
&\le C_\Theta\Big(\ell\!+\!T\P\big(X_{\rho}^\eta\!-\!X_{\rho} \ge\beta\big)\!+\!T\P\Big(\inf_{0\le s\le \ell}\big(X^{\eta}_{\rho+s}-X^{\eta}_{\rho}\big)>-\beta,\rho^\eta\ge \rho+\ell\Big)\Big)\\
&\le C_\Theta\Big(\ell\!+\!T\frac{1}{\beta^2}\E\Big[\sup_{s\in[0,\rho]}|X_s^\eta\!-\!X_s|^2\Big]\!+\!T\P\Big(\inf_{0\le s\le \ell}\big(X^{\eta}_{\rho+s}\!-\!X^{\eta}_{\rho}\big)>-\beta,\rho^\eta\ge \rho+\ell\Big)\Big).
\end{aligned}
\end{align}
Noticing that $X^\eta$ and $X^0$ are controlled by the same $\nu$, and $\mu_\kappa$, $\sigma_\kappa$ are Lipschitz, we obtain 
\begin{align}\label{eq:L2etaeps}
\E\Big[\sup_{s\in[0,\rho]}|X_s^\eta-X_s|^2\Big]\le C_1\eta^2
\end{align}
where $C_1$ is a constant independent of $\eta$ and $\nu$. The last term in \eqref{eq:equic_eps_1} is the analogue of the last term on the right-hand side of \eqref{eq:equicm_a1} but with $\rho$ in place of $\rho^0_m$. Thus by the same arguments based on measure-change and time-change we have
\begin{align*}
\P\Big(\inf_{0\le s\le \ell}\big(X^{\eta}_{\rho+s}-X^{\eta}_{\rho}\big)>-\beta,\rho^\eta\ge \rho+\ell\Big)\le \e^{C_2p}\Big(2\tfrac{\beta}{\kappa\sqrt{\ell}}\Big)^{\frac{p-1}{p}}.
\end{align*}
Combining these bounds with \eqref{eq:equic_eps_1} we get
\begin{align}\label{eq:equic_eps_3}
u(t,x+\eta)-u(t,x)\le c_0\big(\eta+\ell+\tfrac{\eta^2}{\beta^2}+\big(\beta/\sqrt{\ell}\big)^{1-\frac1p}\big).
\end{align}

Now we prove a lower bound. Let $\tau\in\cT_t$ be optimal for $u(t,x)$ and let $\nu^\eta\in\cA_{t,x+\eta}^\circ$ be optimal for $u(t,x+\eta)$. Denote $\rho=\rho(t,x;\nu^\eta,\kappa)$ and notice that $\dot \nu^\eta_s\,\mathds{1}_{\{s\le \rho\}}$ defines a control in $\cA^\circ_{t,x}$. Let $X^{\nu^\eta,\kappa;x+\eta}$ and $X^{\nu^\eta,\kappa;x}$ be denoted by $X^{\eta}$ and $X$, respectively. We have
\begin{align*}
&u(t,x\!+\!\eta)\!-\!u(t,x)\ge \cJ_{t,x+\eta}(\nu^\eta,\tau\!\wedge\!\rho)\!-\!\cJ_{t,x}(\nu^\eta,\tau)\\
&= \E\Big[\e^{-r(\tau\wedge \rho)}\big(g^N(t\!+\!\tau\!\wedge\! \rho,X_{\tau\wedge \rho}^\eta)\!-\!g^N(t\!+\!\tau\!\wedge\! \rho,X_{\tau\wedge \rho})\big)\! +\!\int_0^{\tau\wedge \rho}\!\!\e^{-rs}\big(h^N(t\!+\!s,X_s^\eta)\!-\!h^N(t\!+\!s,X_s)\big)\ud s\Big]\\
&\ge -(\bar{\alpha}^N+TL_N)\E\Big[\sup_{s\in[0,\rho]}|X_s^\eta-X_s|^2\Big]^{\frac12}\ge -(\bar{\alpha}^N+TL_N)\sqrt{C_1}\eta,
\end{align*}
by calculations as in \eqref{eq:equic_eps_0} and \eqref{eq:L2etaeps}. Combining with \eqref{eq:equic_eps_3} we have $|u(t,x+\eta)-u(t,x)|\le c_0\big(\eta+\ell+\frac{\eta^2}{\beta^2}+\big(\beta/\sqrt{\ell}\big)^{1-\frac1p}\big)$. Taking $\beta=\ell=\sqrt{\eta}$ and $\eta\le 1$ yields $|u(t,x+\eta)-u(t,x)|\le c_0\eta^{1/4-1/4p}$ with $c_0=c_0(\kappa,N)>0$ independent of $\eps$ and $t$.

{\bf Step 2}. We prove now Lipschitz continuity in time. Let $T\ge t>s\ge 0$. Let $\tau\in\cT_{t}\subset\cT_{s}$ be optimal for $u(t,x)$ and let $\nu\in\cA_{s,x}^\circ$ be optimal for $u(s,x)$. For $\rho_t=\rho(t,x;\nu,\kappa)$ and $\rho_s=\rho(s,x;\nu,\kappa)$ we have $\rho_s\wedge\tau=\rho_t\wedge\tau$. Then
\begin{align}\label{eq:ubut}
\begin{aligned}
&u(t,x)-u(s,x)\le \cJ_{t,x}(\nu,\tau)-\cJ_{s,x}(\nu,\tau)\\
&=\E_x\Big[\e^{-r(\rho_t\wedge\tau)}\big(g^N(t\!+\!\rho_t\wedge\tau,X_{\rho_t\wedge\tau}^\nu)\!-\!g^N(s\!+\!\rho_t \wedge\tau,X_{\rho_t\wedge\tau}^\nu)\big)\\
&\qquad\;\;+\!\int_0^{\rho_t\wedge\tau}\!\!\e^{-r v}\big(h^N(t\!+\!v,X_v^\nu)\!-\!h^N(s\!+\!v,X_v^\nu)\big)\ud v\Big] \\
&\le\E_x\Big[\e^{-r(\rho_t\wedge\tau)}\bar{\alpha}^N|t-s|+\int_0^{\rho_t\wedge\tau}\e^{-r v}L_N|t-s|\ud v\Big] \le(\bar{\alpha}^N+TL_N)|t-s|,
\end{aligned}
\end{align}
because $g^N$ and $h^N$ are Lipschitz by \eqref{eq:thetaN} with constant $\bar{\alpha}^N$ and $L_N$, respectively.

For the reverse inequality, let $\tau\in\cT_{s}$ be optimal for $u(s,x)$ and let $\nu\in\cA_{t,x}^\circ$ be optimal for $u(t,x)$. We set $\dot \nu_v=0$ for $v\in[\![\rho_t,\infty)\!)$ and $\tau_t=\tau\wedge(T-t)$. Since $u(t,x)-u(s,x)\ge \cJ_{t,x}(\nu,\tau_t)-\cJ_{s,x}(\nu,\tau)$, an application of Dynkin's formula on $[\![\tau_t,\rho_s\wedge\tau]\!]$ yields (cf.\ \eqref{eq:gmNDynk}) 
\begin{align}\label{eq:lbut}
\begin{aligned}
u(t,x)\!-\!u(s,x)
\ge\E_x\Big[&\e^{-r(\rho_t\wedge\tau_t)}\big(g^N(t\!+\!\rho_t\wedge\tau_t,X_{\rho_t\wedge\tau_t}^\nu)\!-\!g^N(s\!+\!\rho_t\wedge\tau_t,X_{\rho_t\wedge\tau_t}^{\nu})\big)\\
&+\!\int_0^{\rho_t\wedge\tau_t}\!\!\e^{-r v}\big(h^N(t\!+\!v,X_v^\nu)\!-\!h^N(s\!+\!v,X_v^{\nu})\big)\ud v\\
&-\mathds{1}_{\{T-t<\rho_t\wedge\tau\}}\!\int_{T-t}^{\rho_s\wedge\tau}\!\!\e^{-r v}\Theta^{N}_{\kappa}(s\!+\!v,X_v^{\nu})\ud v\Big]\ge -(\bar{\alpha}^N\!+\!L_NT\!+\!C_\Theta)|t\!-\!s|,
\end{aligned}
\end{align}
where we used again that $g^N$, $h^N$ are Lipschitz in time with constant $\bar{\alpha}^N$ and $L_N$, respectively, and $\Theta^{N}_{\kappa}$ is bounded by $C_\Theta$. Combining \eqref{eq:ubut} and \eqref{eq:lbut} we conclude.
\end{proof}

Lemma \ref{lem:W12pbound} and Proposition \ref{prop:equiuepsmuN} imply $(u^{N,\eps}_\kappa)_{\eps\in(0,1)}$ bounded in $W^{1,2;p}_{\ell oc}$ and equi-continuous on $\overline \cO$. We define, as in the proof of Theorem \ref{thm:highreguued} 
\begin{align}\label{eq:limu2}
u^{N}_\kappa\coloneqq\lim_{n\to\infty}u^{N,\eps_n}_\kappa,
\end{align}
where the limit is along a sequence $(\eps_n)_{n\in\N}$ decreasing to zero and it is understood weakly in $W^{1,2;p}(\cK)$ for any compact $\cK\subset\cO$ and strongly in $L^\infty(\cK)$ for any compact in $\overline \cO$. Then we obtain a probabilistic representation of $u^N_\kappa$, similarly to Lemma \ref{lem:uepsmuN}.
For $(t,x)\in\overline\cO$ and $(\nu,\tau)\in\cA_{t,x}\times\cT_t$, let 
\begin{align}\label{eq:5.1b}
\begin{aligned}
\cJ_{t,x}^{N,\kappa}(\nu,\tau)&=\E_x\Big[\e^{-r(\tau\wedge\rho)}g^N(t\!+\!\tau\wedge\rho,X_{\tau\wedge\rho}^{\nu,\kappa})\\
&\quad\qquad+\!\int_0^{\tau\wedge\rho}\!\!\!\!\e^{-rs} h^N(t\!+\!s,X_s^{\nu,\kappa})\ud s\!+\!\int_{[0,\tau\wedge\rho]}\!\!\!\!\e^{-rs} \bar{\alpha}^N \ud |\nu|_s\Big].
\end{aligned}
\end{align}
\begin{remark}\label{rem:Ak}
Notice that we should use the notation $\cA^\kappa_{t,x}$ for the class of admissible controls, because we want $X^{\nu,\kappa}_{\rho}=0$ and $\nu_{s}-\nu_{\rho }=0$ for $s\in[\![\rho,T-t]\!]$. However, we prefer to avoid adding new notations when no confusion shall arise. 
\end{remark}

\begin{lemma}\label{lem:uNkappagame}
We have~$u^{N}_{\kappa}(t,x)\!=\!\sup_{\tau\in\cT_t}\inf_{\nu\in\cA_{t,x}}\cJ_{t,x}^{N,\kappa}(\nu,\tau)\!=\!\inf_{\nu\in\cA_{t,x}}\sup_{\tau\in\cT_t}\cJ_{t,x}^{N,\kappa}(\nu,\tau)$,~for $(t,x)\in\overline \cO$.
Moreover,~the~stopping~time~$\sigma_*^{N,\kappa}\wedge\tau_*^{N,\kappa}$~with~$\tau_*^{N,\kappa}(t,x;\nu,\kappa)$~and~$\sigma_*^{N,\kappa}(t,x;\nu,\kappa)$~defined~as
\begin{align}\label{eq:tau*Nkappa}
\begin{aligned}
\tau_*^{N,\kappa}&\coloneqq\inf\big\{s\ge 0\,\big|\, u_\kappa^N(t+s,X_s^{\nu,\kappa;x})=g^N(t+s,X_s^{\nu,\kappa;x})\big\}\wedge (T-t),\\
\sigma_*^{N,\kappa}&\coloneqq\inf\big\{s\ge 0\,\big|\, u_\kappa^N(t+s,X_{s-}^{\nu,\kappa;x})=g^N(t+s,X_{s-}^{\nu,\kappa;x})\big\}\wedge (T-t),
\end{aligned}
\end{align}
is optimal for the stopper for any $\nu\in\cA_{t,x}$.
\end{lemma}
\begin{proof}
The proof follows very closely the one in \cite[Thm.\ 6]{bovo2022variational}. Therefore we only outline it here. First we observe that by \eqref{eq:limu1} and \eqref{eq:limu2} we can select a sequence $(\delta_i,\eps_i)_{i\in\N}$ so that $u^{N}_\kappa=\lim_{i\to\infty}u^{N,\eps_i,\delta_i}_\kappa$ weakly in $W^{1,2;p}(\cK)$ and strongly in $L^\infty(\cK)$. Take $(\tau,\nu)\in\cT_t\times\cA_{t,x}$. We can apply \eqref{eq:dyn0} with $w_s\equiv 0$ and replacing $T_k$ by $\tau$ and $\dot \nu_s\ud s$ by $\ud \nu_s$ (notice that we must account for jumps of the controlled dynamics). Then, using that $u^{N,\eps,\delta}_\kappa$ solves the PDE \eqref{eq:pdeinRdeps}, we obtain
\begin{align}\label{eq:ito}
\begin{aligned}
u^{N,\eps,\delta}_\kappa(t,x)=\E_x\Big[&\e^{-r(\theta_{n,m}\wedge \tau)}u^{N,\eps,\delta}_\kappa(t\!+\!\theta_{n,m}\!\wedge\! \tau,X^\nu_{\theta_{n,m}\wedge \tau-})\\
&+\int_0^{\theta_{n,m}\wedge \tau}\!\!\e^{-rs}\big[h^N\!+\!\tfrac1\delta\big(g^N\!-\!u^{N,\eps,\delta}_\kappa\big)^+\!-\!\psi_\eps\big(|\partial_x u^{N,\eps,\delta}_\kappa|^2\!-\!(\bar \alpha^N)^2\big)\big](t\!+\!s,X_s^\nu)\ud s\\
&-\int_0^{\theta_{n,m}\wedge \tau}\!\!\e^{-r s}\partial_xu^{N,\eps,\delta}_\kappa(t\!+\!s,X_s^\nu)\ud \nu^c_s-\!\!\!\!\sum_{s\in[0,\theta_{n,m}\wedge \tau)}\!\!\e^{-r s}\Delta u^{N,\eps,\delta}_\kappa(t\!+\!s,X_s^\nu)\Big],
\end{aligned}
\end{align}
where $\Delta u^{N,\eps,\delta}_\kappa(t+s,X_s^\nu)=u^{N,\eps,\delta}_\kappa(t+s,X_s^\nu)-u^{N,\eps,\delta}_\kappa(t+s,X_{s-}^\nu)$ and $\nu^c$ is the continuous part of $\nu$. Now we take $\tau=\tau_*^{N,\kappa}\wedge\sigma_*^{N,\kappa}$ and drop the term with $\psi_\eps$. Letting $(\delta,\eps)\to (0,0)$ along the sequence $(\delta_i,\eps_i)_{i\in\N}$ and finally taking $n,m\to\infty$ we can show (cf.\ \cite[Thm.\ 6, p.\ 31]{bovo2022variational})
\begin{align}\label{eq:uNkopt}
\begin{aligned}
u^{N}_\kappa(t,x)\le\E_x\Big[&\e^{-r(\tau_*^{N,\kappa}\wedge\sigma_*^{N,\kappa})}u^N_\kappa(t\!+\!\tau_*^{N,\kappa}\!\wedge\!\sigma_*^{N,\kappa},X^\nu_{\tau_*^{N,\kappa}\wedge\sigma_*^{N,\kappa}-})\!+\!\int_0^{\tau_*^{N,\kappa}\wedge\sigma_*^{N,\kappa}}\!\!\!\e^{-rs}h^N(t\!+\!s,X_s^\nu)\ud s\\
&-\int_0^{\tau_*^{N,\kappa}\wedge\sigma_*^{N,\kappa}}\!\!\e^{-r s}\partial_x u^{N}_\kappa(t+s,X_s^\nu)\ud \nu^c_s-\sum_{s\in[0,\tau_*^{N,\kappa}\wedge\sigma_*^{N,\kappa})}\e^{-r s}\Delta u^{N,\eps,\delta}_\kappa(t+s,X_s^\nu)\Big].
\end{aligned}
\end{align}
When taking limits as $(\eps,\delta)\to 0$, the penalty term $\frac1\delta(g^N-u^{N,\eps,\delta}_\kappa)^+$ can be shown to vanish thanks to Lemma \ref{lem:bndobstpen2}. Recalling $\theta_\infty=\lim_{n\to\infty}\lim_{m\to\infty}\theta_{n,m}$ we have $\theta_\infty\le \rho$ and we must ensure that $\theta_\infty\ge \tau_*^{N,\kappa}\wedge\sigma_*^{N,\kappa}$ to justify \eqref{eq:uNkopt}. On the event $\{\theta_\infty<T-t\}\cap\{X^{\nu}_{\theta_{\infty}-}=0\}$ we have $\theta_\infty\ge \sigma_*^{N,\kappa}$ because $u^{N}_\kappa(t,0)=g^N(t,0)$. Similarly, $\theta_\infty\ge \tau_*^{N,\kappa}$ on the event $\{\theta_\infty<T-t\}\cap\{X^{\nu}_{\theta_{\infty}}=0\}$. 

Finally, we show an upper bound for the right-hand side of \eqref{eq:uNkopt}. On $\{\sigma_*^{N,\kappa}< \tau_*^{N,\kappa}\}$ we have
\begin{align}\label{eq:sigma*opt}
\begin{aligned}
&u^N_\kappa(t\!+\!\tau_*^{N,\kappa}\!\wedge\!\sigma_*^{N,\kappa},X^\nu_{\tau_*^{N,\kappa}\wedge\sigma_*^{N,\kappa}-})=u^N_\kappa(t\!+\!\sigma_*^{N,\kappa},X^\nu_{\sigma_*^{N,\kappa}-})\\
&=g^N(t\!+\!\sigma_*^{N,\kappa},X^\nu_{\sigma_*^{N,\kappa}-})\le g^N(t\!+\!\sigma_*^{N,\kappa},X^\nu_{\sigma_*^{N,\kappa}})+\bar{\alpha}^N\Delta\nu_{\sigma_*^{N,\kappa}},
\end{aligned}
\end{align}
by using that $g^N$ is Lipschitz with constant smaller than $\bar{\alpha}^N$.
On $\{\sigma_*^{N,\kappa}\ge \tau_*^{N,\kappa}\}$
\begin{align}\label{eq:tau*opt}
\begin{aligned}
&u^N_\kappa(t\!+\!\tau_*^{N,\kappa}\!\wedge\!\sigma_*^{N,\kappa},X^\nu_{\tau_*^{N,\kappa}\wedge\sigma_*^{N,\kappa}-})=u^N_\kappa(t\!+\!\tau_*^{N,\kappa},X^\nu_{\tau_*^{N,\kappa}-})\\
&\le u^N_\kappa(t\!+\!\tau_*^{N,\kappa},X^\nu_{\tau_*^{N,\kappa}})+\bar{\alpha}^N\Delta\nu_{\tau_*^{N,\kappa}}= g^N(t\!+\!\tau_*^{N,\kappa},X^\nu_{\tau_*^{N,\kappa}})+\bar{\alpha}^N\Delta\nu_{\tau_*^{N,\kappa}},
\end{aligned}
\end{align}
where the inequality holds because $|\partial_x u^{N}_\kappa|\le \bar \alpha^N$ thanks to the second equation in \eqref{eq:gradpeneq}. 

Plugging \eqref{eq:sigma*opt} and \eqref{eq:tau*opt} into \eqref{eq:uNkopt} and using again that $|\partial_x u^{N}_\kappa|\le \bar\alpha_N$ to bound the final two terms in \eqref{eq:uNkopt} we obtain $u^{N}_\kappa(t,x)\le \cJ_{t,x}^{N,\kappa}(\nu,\tau_*^{N,\kappa}\wedge\sigma_*^{N,\kappa})$.
By the arbitrariness of $\nu$ and the sub-optimality of $\tau_*^{N,\kappa}\wedge\sigma_*^{N,\kappa}$, we obtain $u^{N}_{\kappa}(t,x)\le\sup_{\tau\in\cT_t}\inf_{\nu\in\cA_{t,x}}\cJ_{t,x}^{N,\kappa}(\nu,\tau)$.

We show now that $ u^{N}_{\kappa}$ is greater or equal than the upper value of the game. Let us go back to \eqref{eq:ito} and take $\nu=\nu^{N,\kappa,\eps}$ as in Lemma \ref{lem:uepsmuN} (so there are no jumps). Then we obtain \eqref{eq:uNkeps,delta}. Taking limits along the sequence from \eqref{eq:limu1} we obtain the same expression as \eqref{eq:uNkeps,delta} but with $u^{N,\eps}_\kappa$ instead of $u^{N,\eps,\delta}_\kappa$. By the choice of $\nu$ we have 
\[
-\psi_\eps\big(|\partial_x u^{N,\eps}_\kappa|^2-(\bar \alpha^N)^2\big)(t+s,X_s^\nu)-\partial_x u^{N,\eps}_\kappa(t+s,X_s^\nu)\cdot\dot{\nu}_s=H^{N,\eps}(\dot \nu_s).
\]

Moreover, $H^{N,\eps}(\dot{\nu}_s)\ge \bar{\alpha}^N|\dot{\nu}_s|$ and Lemma \ref{lem:bndobstpen2} implies $u^{N,\eps}_\kappa\ge g^N$. Then
\begin{align*}
\begin{aligned}
u^{N,\eps}_\kappa(t,x)
\ge \E_x\Big[&\e^{-r(\theta_{n,m}\wedge \tau)}g^{N}(t\!+\!\theta_{n,m}\wedge \tau,X^{\nu}_{\theta_{n,m}\wedge \tau})\!+\!\int_0^{\theta_{n,m}\wedge \tau}\!\e^{-rs}\big(h^N(t+s,X_s^{\nu})\!+\bar{\alpha}^N|\dot{\nu}_s|\big)\ud s\Big].
\end{aligned}
\end{align*}
Passing to the limit as $n,m\to\infty$ we have $\lim_n\lim_m\theta_{n,m}=\rho$, $\P$-a.s., because the control is continuous. Thus, we get $u^{N,\eps}_\kappa(t,x)\ge \cJ_{t,x}^{N,\kappa}(\nu,\tau)$.
By arbitrariness of $\tau$ and sub-optimality of $\nu=\nu^{N,\kappa,\eps}$, we have $u^{N,\eps}_{\kappa}(t,x)\ge\inf_{\nu\in\cA_{t,x}}\sup_{\tau\in\cT_t}\cJ_{t,x}^{N,\kappa}(\nu,\tau)$. Therefore, passing to the limit as $\eps\downarrow0$ yields $u^{N}_{\kappa}(t,x)\ge\inf_{\nu\in\cA_{t,x}}\sup_{\tau\in\cT_t}\cJ_{t,x}^{N,\kappa}(\nu,\tau)$, which completes the proof.
\end{proof}

\begin{remark}[{\bf Proof of Theorem \ref{thm:usolvar}}]\label{rem:thm}
The lemma above proves Theorem \ref{thm:usolvar} under the condition {\bf A1}, because the approximations indexed by $N$ and $\kappa$ are superfluous in that case. 

Thanks to Lemmas \ref{lem:bndobstpen2} and \ref{lem:W12pbound}, the same arguments of proof as in \cite[Thm.\ 5]{bovo2022variational} allow us to pass to the limit in the sequence $u^{N,\eps_n,\delta_n}_{\kappa}$ and obtain that $u^N_\kappa$
is a solution of Problem \ref{prb:varineq}. Repeating verbatim the proof from \cite[Thm.\ 6]{bovo2022variational} (cf.\ p.\ 32 therein) we obtain maximality of the solution.
\end{remark}

\subsection{Value of the game with unbounded data}\label{sec:valueunb}
Throughout this section we work under Assumption {\bf A.2} of Theorem \ref{thm:usolvar}. 
In this case, we show in the next lemma that optimal controls are non-increasing. Recall $\nu=\nu^+-\nu^-$. 

\begin{lemma}\label{lem:nu^-}
Let $\cA_{t,x}^-\coloneqq\{\nu\in\cA_{t,x}:\nu_{T-t}^+=0,\P\text{-a.s.}\}$. For any $(t,x)\in\overline\cO$, it holds
\begin{align*}
u^N_\kappa(t,x)=\inf_{\nu\in\cA_{t,x}^-}\sup_{\tau\in\cT_t}\cJ_{t,x}^{N,\kappa}(\nu,\tau)=\sup_{\tau\in\cT_t}\inf_{\nu\in\cA_{t,x}^-}\cJ_{t,x}^{N,\kappa}(\nu,\tau).
\end{align*}
\end{lemma}
\begin{proof}
Let $(t,x)\in\overline\cO$. Let $\nu\in\cA_{t,x}$ be $\eta$-optimal for $u^N_\kappa(t,x)$. Set $\xi\coloneqq-\nu^-$ and $\bar{\rho}=\rho(t,x;\xi,\kappa)$ as in \eqref{eq:rhoknu}. A priori $\xi\notin\cA_{t,x}$ because it could make the process $X^{\xi,\kappa}$ jump strictly below zero. However, the process $X^{\xi,\kappa}$ is well-defined. Then, we construct an admissible control $\tilde \nu\in\cA^-_{t,x}$ as $\tilde{\nu}_s=\xi_s\mathds{1}_{[0,\bar{\rho})}(s)+\big(\xi_{\bar{\rho}-}-|X^{\xi,\kappa}_{\bar{\rho}-}|\big)\mathds{1}_{[\bar{\rho},T-t]}(s)$. We have $X^{\tilde\nu,\kappa}_s\ge 0$ for all $s\in[\![0,\bar \rho]\!]$, $\P$-a.s. Moreover, simple comparison yields $X^{\tilde \nu,\kappa}_s\le X^{\nu,\kappa}_s$ for all $s\in[\![0,\bar \rho]\!]$ and, letting $\tilde{\rho}=\rho(t,x;\tilde{\nu},\kappa)$ and $\rho=\rho(t,x;\nu,\kappa)$, it is clear that $\tilde{\rho}=\bar{\rho}\le \rho$, $\P$-a.s.

Since $\nu$ is $\eta$-optimal for $u^N_\kappa$, letting $\sigma\in\cT_t$ be $\eta$-optimal for $\sup_{\tau\in\cT_t}\cJ_{t,x}^{N,\kappa}(\tilde \nu,\tau)$ we get
\begin{align*}
&u^N_\kappa(t,x)-\inf_{\xi\in\cA_{t,x}^-}\sup_{\tau\in\cT_t}\cJ_{t,x}^{N,\kappa}(\xi,\tau)\ge\cJ_{t,x}^{N,\kappa}(\nu,\sigma\wedge\tilde{\rho})-\cJ_{t,x}^{N,\kappa}(\tilde \nu,\sigma)-2\eta\\
&= \E_x\Big[\e^{-r(\tilde{\rho}\wedge\sigma)}\big(g^N(t+\tilde{\rho}\wedge\sigma,X_{\tilde{\rho}\wedge\sigma}^{\nu,\kappa})-g^N(t+\tilde{\rho}\wedge\sigma,X_{\tilde{\rho}\wedge\sigma}^{\tilde{\nu},\kappa})\big)\\
&\qquad\quad +\!\!\int_0^{\tilde{\rho}\wedge\sigma}\!\!\e^{-rs}\big(h^N(t\!+\!s,X_s^{\nu,\kappa})\!-\!h^N(t\!+\!s,X_s^{\tilde \nu,\kappa})\big)\ud s\!+\!\int_{[0,\tilde{\rho}\wedge\sigma]}\!\!\e^{-rs}\bar{\alpha}^N\big(\ud |\nu|_s\!-\!\ud |\tilde \nu|_s\big)\Big]-2\eta,
\end{align*}
where we used $\tilde{\rho}\le\rho$. Since $|\nu|_s\ge |\tilde \nu|_s$ and $X^{\nu,\kappa}_s\ge X^{\tilde \nu,\kappa}_s$, for all $s\in[\![0,\tilde\rho]\!]$, $\P$-a.s., then spatial monotonicity of $g^N$ and $h^N$ yields $u^N_\kappa(t,x)-\inf_{\xi\in\cA_{t,x}^-}\sup_{\tau\in\cT_t}\cJ_{t,x}^N(\xi,\tau)\ge -2\eta$. Since $\eta$ is arbitrary, 
\begin{align*}
u^N_\kappa(t,x)\ge \inf_{\xi\in\cA_{t,x}^-}\sup_{\tau\in\cT_t}\cJ_{t,x}^{N,\kappa}(\xi,\tau)\ge\sup_{\tau\in\cT_t}\inf_{\xi\in\cA_{t,x}^-}\cJ_{t,x}^{N,\kappa}(\xi,\tau)\ge u^N_\kappa(t,x),
\end{align*}
where the last inequality holds because $\cA_{t,x}^-\subset\cA_{t,x}$.
\end{proof}

Next we prove convergence of the process $X^{\nu,\kappa}$ to the solution $X^\nu$ of \eqref{eq:prcXcntrll} when $\kappa\to 0$. Here we must recall the notation from Remark \ref{rem:Ak} and distinguish $\cA_{t,x}$ and $\cA^\kappa_{t,x}$. Analogously, we draw a distinction between $\cA^-_{t,x}$ and $\cA^{\kappa,-}_{t,x}$ with obvious meaning of the latter notation.
\begin{lemma}\label{lem:convkappa}
Let $(t,x)\in\cO$ and set $\cB^\kappa_{t,x}\coloneqq \cA_{t,x}^-\cup \cA_{t,x}^{\kappa,-}$. Recall $\tau_0=\tau_0(t,x;\nu)$ as in \eqref{eq:tau0} and set $\tau_\kappa\coloneqq\rho(t,x;\nu,\kappa)$ as in \eqref{eq:rhoknu}. Then, for any compact $U\subset[0,\infty)$
\begin{align*}
\lim_{\kappa\downarrow 0}\sup_{x\in U}\sup_{\nu\in\cB^\kappa_{t,x}}\sup_{s\in[0,T-t]}\E_x\big[|X^{\nu,\kappa}_{s\wedge\tau_0\wedge\tau_\kappa}-X^{\nu}_{s\wedge\tau_0\wedge\tau_\kappa}|\big]=0,
\end{align*}
where $X^\nu_{0-}=X^{\nu,\kappa}_{0-}=x$ under $\P_x$.
\end{lemma}
\begin{proof}
Let $\nu\in\cB_{t,x}^\kappa$. Set $Y^{\kappa}_s\coloneqq X^{\nu,\kappa}_{s}-X^\nu_{s}$. Since $Y^{\kappa}$ is a continuous semi-martingale on $[\![0,\tau_0\wedge\tau_\kappa]\!]$, by It\^o-Tanaka's formula we have
\begin{align*}
|Y^{\kappa}_{s\wedge\tau_0\wedge\tau_\kappa}|= \int_0^{s\wedge\tau_0\wedge\tau_\kappa}\!\sign(Y^{\kappa}_\lambda)\ud Y^{\kappa}_\lambda+L^0_{s\wedge\tau_0\wedge\tau_\kappa}(Y^{\kappa}),
\end{align*}
where $L^0(Y^\kappa)$ is the local time at zero of the process $Y^\kappa$.
Taking expectation and using that the stochastic integral is a martingale (possibly up to a localisation procedure), we have
\begin{align}\label{eq:Ykappatau}
\begin{aligned}
\E_x\big[|Y^{\kappa}_{s\wedge\tau_0\wedge\tau_\kappa}|\big]=\E_x\Big[\int_0^{s\wedge\tau_0\wedge\tau_\kappa}\!\sign(Y^{\kappa}_\lambda)\big(\mu_\kappa(X^{\nu,\kappa}_\lambda)-\mu(X^\nu_\lambda)\big)\ud \lambda+L^0_{s\wedge\tau_0\wedge\tau_\kappa}(Y^{\kappa})\Big].
\end{aligned}
\end{align}
For any $\eps>0$, \cite[Lem.\ 5.1]{deangelis2019numerical} yields
\begin{align}\label{eq:L0Ykappatau}
\begin{aligned}
\E_x\Big[L^0_{s\wedge\tau_0\wedge\tau_\kappa}(Y^{\kappa})\Big]&\le 4\eps-2\E_x\Big[\int_0^{s\wedge\tau_0\wedge\tau_\kappa}\!\Big(\mathds{1}_{\{Y^{\kappa}_\lambda\in(0,\eps)\}}+\mathds{1}_{\{Y^{\kappa}_\lambda\ge\eps\}}\e^{1-\frac{Y_\lambda^{\kappa}}{\eps}}\Big)\ud Y_\lambda^{\kappa}\Big]\\
&\quad+\frac{1}\eps\E_x\Big[\int_0^{s\wedge\tau_0\wedge\tau_\kappa}\!\mathds{1}_{\{Y_\lambda^{\kappa}>\eps\}}\e^{1-\frac{Y_\lambda^{\kappa}}{\eps}}\ud\langle Y^{\kappa}\rangle_\lambda \Big],
\end{aligned}
\end{align}
where $\langle Y^\kappa\rangle$ is the quadratic variation of $Y^\kappa$. We need to bound the right-hand side above. For that we distinguish two cases: when $\nu\in\cA^{\kappa,-}_{t,x}$ and when $\nu\in\cA^-_{t,x}$. For $\nu\in\cA^{\kappa,-}_{t,x}$, we have $X^{0,\kappa}_s\ge X^{\nu,\kappa}_s\ge 0$ for all $s\in[\![0,\tau_\kappa]\!]$ and 
\begin{align}\label{eq:est1}
\E_x\Big[\sup_{\lambda\in[0,\tau_\kappa]} \big|X^{\nu,\kappa}_\lambda\big|^p\Big]\le \E_x\Big[\sup_{\lambda\in[0,T-t]}|X^{0,\kappa}_\lambda|^p\Big]\le C_p\big(1+|x|^p\big),
\end{align}
for some $C_p>0$ that can be taken independent of $\kappa$, by standard SDE estimates (see \cite[Cor.\ 2.5.10]{krylov1980controlled}). It may be $X^\nu_{\tau_0-}\ge0$ and $X^{\nu}_{\tau_0}<0$. In that case $X^{\nu}_{\tau_0}\ge -|\Delta\nu_{\tau_0}|$ but since $\nu\in\cA^{\kappa,-}_{t,x}$, then also $|\Delta\nu_{\tau_0}|\le X^{\nu,\kappa}_{\tau_0-}\le X^{0,\kappa}_{\tau_0-}$. Therefore, using $0\le X^\nu_\lambda\le X^0_\lambda$ for $\lambda\in[\![0,\tau_0)\!)$ and $|X^\nu_{\tau_0}|\le |\Delta \nu_{\tau_0}|$
\begin{align}\label{eq:est2}
\E_x\Big[\sup_{\lambda\in[0,\tau_0]}|X^\nu_\lambda|^p\Big]\le \E_x\Big[\sup_{\lambda\in[0,T-t]}|X^{0}_\lambda|^p\Big]+\E_x\Big[\sup_{\lambda\in[0,T-t]}|X^{0,\kappa}_\lambda|^p\Big]\le {C}_p\big(1+|x|^p\big),
\end{align}
where $C_p>0$ can be chosen as in \eqref{eq:est1} with no loss of generality.
Analogous considerations yield the same bounds when $\nu\in\cA^-_{t,x}$. 
 
With no loss of generality, we can assume $|\mu_\kappa(x)|\!\le\! |\mu(x)|$ and $\sigma_\kappa(x)\!\le\! \sigma(x)\!+\!\kappa$ for $x\!\in\![0,\infty)$ with $\mu_\kappa(x)\!=\!\mu(x)$ and $\sigma_\kappa(x)\!=\! \sigma(x)\!+\!\kappa$ for $x\!\in\![0,\kappa^{-1}]$. Linear growth of $\mu,\sigma$ and \eqref{eq:est1}--\eqref{eq:est2} imply 
\begin{align*}
\E_x\Big[\int_0^{s\wedge\tau_0\wedge\tau_\kappa}\!\!\Big(\big|f_\kappa(X^{\nu,\kappa}_\lambda)\big|^2\!+\!\big|f(X^\nu_\lambda)\big|^2\Big)\ud \lambda\Big]\!\le\! C \Big(1\!+\!\E_x\Big[\sup_{\lambda\in[0,s\wedge\tau_0\wedge\tau_\kappa]}\!\big(|X^{\nu,\kappa}_\lambda\big|^2 \!+\! \big|X^\nu_\lambda\big|^2\big)\Big]\Big)\!\le\! C(1\!+\!|x|^2),
\end{align*}
for $f_\kappa\in\{\mu_\kappa,\sigma_\kappa\}$, $f\in\{\mu,\sigma\}$ and some $C>0$ that may change from one inequality to the next but it is independent of $\kappa$ and $x$.

Setting $\bar X^\kappa_{s}\coloneqq\sup_{\lambda\in[0,s\wedge\tau_0\wedge\tau_\kappa]} |X^{\nu,\kappa}_\lambda|$, Markov's inequality combined with \eqref{eq:est1} and \eqref{eq:est2} yield $\P_x(\bar X^\kappa_{s}>\kappa^{-1})\le C(1+|x|^2)\kappa^2$.
For the second term on the right-hand side of \eqref{eq:L0Ykappatau} we have 
\begin{align}\label{eq:est4}
\begin{aligned}
&\,\Big|\E_x\Big[\int_0^{s\wedge\tau_0\wedge\tau_\kappa}\!\!\Big(\mathds{1}_{\{Y^{\kappa}_\lambda\in(0,\eps)\}}\!+\!\mathds{1}_{\{Y^{\kappa}_\lambda\ge\eps\}}\e^{1-\eps^{-1}Y_\lambda^{\kappa}}\Big)\ud Y_\lambda^{\kappa}\Big]\Big|\!\le\! \E_x\Big[\int_0^{s\wedge\tau_0\wedge\tau_\kappa}\!\big|\mu_\kappa(X^{\nu,\kappa}_\lambda)\!-\!\mu(X^\nu_\lambda)\big|\ud \lambda\Big]\\
&\,\le \E_x\Big[\mathds{1}_{\{\bar X^\kappa_s\le \kappa^{-1}\}}\!\int_0^{s\wedge\tau_0\wedge\tau_\kappa}\!\!\big|\mu(X^{\nu,\kappa}_\lambda)\!-\!\mu(X^\nu_\lambda)\big|\ud \lambda\Big]\\
&\quad+\!\E_x\Big[\mathds{1}_{\{\bar X^\kappa_s> \kappa^{-1}\}}\!\int_0^{s\wedge\tau_0\wedge\tau_\kappa}\!\!\Big(\big|\mu_\kappa(X^{\nu,\kappa}_\lambda)\big|\!+\!\big|\mu(X^\nu_\lambda)\big|\Big)\ud \lambda\Big]\\
&\,\le D_1\E_x\Big[\int_0^{s\wedge\tau_0\wedge\tau_\kappa}\!\big|Y^{\kappa}_\lambda\big|\ud \lambda\Big]\!+\!\E_x\Big[\int_0^{s\wedge\tau_0\wedge\tau_\kappa}\!\Big(\big|\mu_\kappa(X^{\nu,\kappa}_\lambda)\big|^2\!+\!\big|\mu(X^\nu_\lambda)\big|^2\Big)\ud \lambda\Big]^\frac12\P\big(\bar X^\kappa_s> \kappa^{-1}\big)^{\frac12}\\
&\,\le
D_1\E_x\Big[\int_0^{s\wedge\tau_0\wedge\tau_\kappa}\!\big|Y_\lambda^{\kappa}|\ud \lambda\Big]\!+\!C(1+|x|^2)\kappa,
\end{aligned}
\end{align}
where $D_1$ is the Lipschitz constant for $\mu$ and we used Cauchy-Schwarz for the third inequality. For the last term on the right-hand side of \eqref{eq:L0Ykappatau}, notice $\ud\langle Y^{\kappa}\rangle_\lambda=\big(\sigma_\kappa(X^{\nu,\kappa}_\lambda)-\sigma(X^\nu_\lambda)\big)^2\ud \lambda$ so that
\begin{align}\label{eq:est5}
\begin{aligned}
&\frac{1}\eps\E_x\Big[\int_0^{s\wedge\tau_0\wedge\tau_\kappa}\!\mathds{1}_{\{Y_\lambda^{\kappa}>\eps\}}\e^{1-\frac{Y_\lambda^{\kappa}}{\eps}}\ud\langle Y^{\kappa}\rangle_\lambda \Big]\\
&=\frac{1}\eps\E_x\Big[\mathds{1}_{\{\bar X^\kappa_s\le \kappa^{-1}\}}\int_0^{s\wedge\tau_0\wedge\tau_\kappa}\!\mathds{1}_{\{Y_\lambda^{\kappa}>\eps\}}\e^{1-\frac{Y_\lambda^{\kappa}}{\eps}}\big(\kappa+\sigma(X^{\nu,\kappa}_\lambda)-\sigma(X^\nu_\lambda)\big)^2\ud \lambda \Big]\\
&\quad+\frac{1}\eps\E_x\Big[\mathds{1}_{\{\bar X^\kappa_s> \kappa^{-1}\}}\int_0^{s\wedge\tau_0\wedge\tau_\kappa}\!\mathds{1}_{\{Y_\lambda^{\kappa}>\eps\}}\e^{1-\frac{Y_\lambda^{\kappa}}{\eps}}\big(\sigma_\kappa(X^{\nu,\kappa}_\lambda)-\sigma(X^\nu_\lambda)\big)^2\ud \lambda \Big]\\
&\le\frac{2T\kappa^2}{\eps}+\frac{2}\eps\E_x\Big[\int_0^{s\wedge\tau_0\wedge\tau_\kappa}\!\mathds{1}_{\{Y_\lambda^{\kappa}>\eps\}}\e^{1-\frac{Y_\lambda^{\kappa}}{\eps}}\big(\sigma(X^{\nu,\kappa}_\lambda)-\sigma(X^\nu_\lambda)\big)^2\ud \lambda \Big]\\
&\quad+\frac{1}\eps\E_x\Big[\mathds{1}_{\{\bar X^\kappa_s> \kappa^{-1}\}}\int_0^{s\wedge\tau_0\wedge\tau_\kappa}\!\mathds{1}_{\{Y_\lambda^{\kappa}>\eps\}}\e^{1-\frac{Y_\lambda^{\kappa}}{\eps}}\big(\sigma_\kappa(X^{\nu,\kappa}_\lambda)-\sigma(X^\nu_\lambda)\big)^2\ud \lambda \Big].
\end{aligned}
\end{align}
Since $\sigma\in C^\gamma([0,\infty))$ with $\gamma\in(\frac12,1)$, for any $p\in\big(\frac{1}{2\gamma},1\big)$, and using a constant $C>0$ that may vary from line to line but it is independent of $\eps$, $p$, $\kappa$ and $\nu$, we have
\begin{align}\label{eq:est5.1}
\begin{aligned}
&\frac{2}\eps\E_x\Big[\int_0^{s\wedge\tau_0\wedge\tau_\kappa}\!\!\!\mathds{1}_{\{Y_\lambda^{\kappa}>\eps\}}\e^{1-\frac{Y_\lambda^{\kappa}}{\eps}}\big(\sigma(X^{\nu,\kappa}_\lambda)\!-\!\sigma(X^\nu_\lambda)\big)^2\ud \lambda \Big]\\
&=\frac{2}\eps\E_x\Big[\int_0^{s\wedge\tau_0\wedge\tau_\kappa}\!\!\!\mathds{1}_{\{Y_\lambda^{\kappa}\in(\eps,\eps^p)\}}\e^{1-\frac{Y_\lambda^{\kappa}}{\eps}}\big(\sigma(X^{\nu,\kappa}_\lambda)\!-\!\sigma(X^\nu_\lambda)\big)^2\ud \lambda \Big]\\
&\quad+ \!\frac{2}\eps\E_x\Big[\int_0^{s\wedge\tau_0\wedge\tau_\kappa}\!\!\!\mathds{1}_{\{Y_\lambda^{\kappa}\ge \eps^p\}}\e^{1-\frac{Y_\lambda^{\kappa}}{\eps}}\big(\sigma(X^{\nu,\kappa}_\lambda)\!-\!\sigma(X^\nu_\lambda)\big)^2\ud \lambda \Big]\\
&\le \frac{2}\eps C\E_x\Big[\int_0^{s\wedge\tau_0\wedge\tau_\kappa}\!\!\!\mathds{1}_{\{Y_\lambda^{\kappa}\in(\eps,\eps^p)\}}\big|Y^\kappa_\lambda\big|^{2\gamma}\ud \lambda \Big]\!+ \!C\frac{2}\eps\e^{1-\frac{1}{\eps^{1-p}}}\E_x\Big[\int_0^{s\wedge\tau_0\wedge\tau_\kappa}\!\!\!\big(1\!+\!|X^{\nu,\kappa}_\lambda|^2\!+\!|X^\nu_\lambda|^2\big)\ud \lambda \Big]\\
&\le 2C\eps^{2\gamma p-1} T\!+\!C \frac{2}\eps\e^{1-\frac{1}{\eps^{1-p}}} \big(1\!+\!|x|^2\big) T,
\end{aligned}
\end{align}
where for the final inequality we used \eqref{eq:est1} and \eqref{eq:est2}. For the final term on the right-hand side of \eqref{eq:est5} we use Cauchy-Schwarz inequality to obtain
\begin{align}\label{eq:est5.3}
\begin{aligned}
&\frac{1}{\eps}\E_x\Big[\mathds{1}_{\{\bar X^\kappa_s> \kappa^{-1}\}}\int_0^{s\wedge\tau_0\wedge\tau_\kappa}\!\big(\sigma_\kappa(X^{\nu,\kappa}_\lambda)\!-\!\sigma(X^\nu_\lambda)\big)^2\ud \lambda \Big]\\
&\le\frac{2}{\eps}\E_x\Big[\mathds{1}_{\{\bar X^\kappa_s> \kappa^{-1}\}}\int_0^{s\wedge\tau_0\wedge\tau_\kappa}\!\big(|\sigma_\kappa(X^{\nu,\kappa}_\lambda)|^2\!+\!|\sigma(X^\nu_\lambda)|^2\big)\ud \lambda \Big]\\
&\le\frac{C T}{\eps}\P\big(\bar X^\kappa_s>\kappa^{-1}\big)^{\frac12}\E_x\Big[\Big(1\!+\!\sup_{\lambda\in[0,T-t]}|X^{0,\kappa}_\lambda|^2\!+\!\sup_{\lambda\in[0,T-t]}|X_\lambda^{0}|^2\Big)^2 \Big]^\frac12\le \frac{\kappa}{\eps}C T(1+|x|^2).
\end{aligned}
\end{align}
Plugging \eqref{eq:est5.1} and \eqref{eq:est5.3} into \eqref{eq:est5} and relabelling $C>0$, we obtain
\begin{align}\label{eq:est6}
\begin{aligned}
\quad\frac{1}\eps\E_x\Big[\int_0^{s\wedge\tau_0\wedge\tau_\kappa}\!\!\!\!\!\mathds{1}_{\{Y_\lambda^{\kappa}>\eps\}}\e^{1-\frac{Y_\lambda^{\kappa}}{\eps}}\ud\langle Y^{\kappa}\rangle_\lambda \Big]\!\le\! C\Big(\frac{\kappa^2}{\eps}\!+\!\eps^{2p\gamma-1}\!+\!\frac{\e^{1-\eps^{p-1}}}{\eps}\!+\!\frac{\kappa}{\eps}\Big)(1\!+\!|x|^2).
\end{aligned}
\end{align}
Arguments as those used in \eqref{eq:est4} yield a bound for the first term on the right-hand side of \eqref{eq:Ykappatau}, 
\begin{align}\label{eq:est7}
\begin{aligned}
\E_x\Big[ \int_0^{s\wedge\tau_0\wedge\tau_\kappa}\!\!\!\sign(Y^{\kappa}_\lambda)\big(\mu_\kappa(X^{\nu,\kappa}_\lambda)\!-\!\mu(X^\nu_\lambda)\big)\ud \lambda\Big]\le D_1\E_x\Big[\int_0^{s\wedge\tau_0\wedge\tau_\kappa}\!\big|Y_\lambda^{\kappa}|\ud \lambda\Big]\!+\!C(1\!+\!|x|^2)\kappa.
\end{aligned}
\end{align}
Therefore, plugging \eqref{eq:L0Ykappatau} with \eqref{eq:est4}, \eqref{eq:est6} and \eqref{eq:est7} into \eqref{eq:Ykappatau}, we obtain
\begin{align*}
\E\big[|Y^{\kappa}_{s\wedge\tau_0\wedge\tau_\kappa}|\big]&\le C\Big(\eps\!+\!\frac{\kappa^2}{\eps}\!+\!\eps^{2p\gamma-1}\!+\!\frac{\e^{1-\eps^{p-1}}}{\eps}\!+\!\frac{\kappa}{\eps}\!+\!\kappa\Big)(1\!+\!|x|^2)\!+\!C\E_x\Big[\int_0^{s\wedge\tau_0\wedge\tau_\kappa}\!\big|Y_\lambda^{\kappa}|\ud \lambda\Big].
\end{align*}
By Gronwall's inequality $\E[|Y^{\kappa}_{s\wedge\tau_0\wedge\tau_\kappa}|]\le C(\eps\!+\!\kappa^2\eps^{-1}\!+\!\eps^{2p\gamma-1}\!+\!\e^{1-\eps^{p-1}}\eps^{-1}\!+\!\kappa\eps^{-1}\!+\!\kappa)(1\!+\!|x|^2)$.
Now we let $\eps$ and $\kappa$ go to zero so that $\kappa/\eps \to 0$. Since $2p\gamma-1>0$ and $p<1$, we conclude the proof.
\end{proof}

Since $(u^{N}_\kappa)_{\kappa\in(0,1)}$ is bounded in $W^{1,2;p}_{\ell oc}$ by Lemma \ref{lem:W12pbound}, then we set
$u^{N}\coloneqq\lim_{n\to\infty}u^{N}_{\kappa_n}$,
where the limit is along a sequence $(\kappa_n)_{n\in\N}$ decreasing to zero and it is understood weakly in $W^{1,2;p}(\cK)$ for any compact in $\cK\subset\cO$ and strongly in $L^\infty(\cK)$ for any compact in $\cK\subset\overline\cO$, as in the proof of Theorem \ref{thm:highreguued}. In order to obtain a probabilistic representation for $u^N$, we introduce $\cJ^N_{t,x}(\nu,\tau)\coloneqq\cJ^{N,0}_{t,x}(\nu,\tau)$ (cf.\ \eqref{eq:5.1b}), i.e., the underlying dynamics in the game is $X^{\nu;x}=X^{\nu,0;x}$.

\begin{proposition}\label{prop:uNkappato0}
We have $u^N\in C(\overline\cO)$, with 
\[
u^N(t,x)=\underline u^N(t,x)\coloneqq\sup_{\tau\in\cT_t}\inf_{\nu\in\cA_{t,x}}\cJ_{t,x}^N(\nu,\tau)=\inf_{\nu\in\cA_{t,x}}\sup_{\tau\in\cT_t}\cJ_{t,x}^N(\nu,\tau)\eqqcolon \overline u^N(t,x),\quad\text{for $(t,x)\in\overline\cO$}.
\]
\end{proposition}
\begin{proof}
Let $U\subset[0,\infty)$ be compact. First we show that 
\begin{align}\label{eq:limsupumuN2}
\qquad\liminf_{\kappa\to0}\inf_{(t,x)\in[0,T]\times U}\big(u^{N}_{\kappa}(t,x)-\overline{u}^N(t,x)\big)\ge 0.
\end{align}
Let $\nu\in\cA^{\kappa,-}_{t,x}$ be an $\eta$-optimal for $u^{N}_{\kappa}(t,x)$. For simplicity, set $X^\kappa=X^{\nu,\kappa;x}$ and $\tau_\kappa=\rho(t,x;\nu,\kappa)$ as in \eqref{eq:rhoknu}. Similarly, let $X^{\nu}=X^{\nu;x}$ and let $\tau_0=\tau_0(t,x;\nu)$ be as in \eqref{eq:tau0}. It may be $\nu\notin\cA_{t,x}$ because $X^\nu_{\tau_0}<0$. 
Define a new control $\tilde{\nu}\in\cA_{t,x}$ as 
\begin{align}\label{eq:nuadmissible}
\quad\tilde{\nu}_s= \nu_s\mathds{1}_{[0, \tau_\kappa\wedge\tau_0)}(s)\!+\!(\nu_{(\tau_\kappa\wedge\tau_0)-}\!-\!X^{\nu}_{(\tau_\kappa\wedge\tau_0)-})\mathds{1}_{[\tau_\kappa\wedge\tau_0,T-t]}(s).
\end{align}
Setting $\tilde X=X^{\tilde{\nu}}$ and $\tilde\tau_0\coloneqq\tau_0(t,x;\tilde \nu)$, we have $\tilde\tau_0\le \tau_0 $, $\P$-a.s.\ and $\tilde X_s=X^\nu_s$ for $s\in[\![0,\tau_\kappa\wedge \tilde\tau_0)\!)$. On $\{\tau_\kappa\le \tau_0\}$ the process $\tilde X$ makes a downward jump at $\tau_\kappa$ of size $X^\nu_{\tau_\kappa-}\ge 0$ so that $\tau_\kappa=\tilde\tau_0$ and $\Delta\tilde{\nu}_{\tau_\kappa}-\Delta\nu_{\tau_\kappa}=-X_{\tau_\kappa}^\nu$. Instead, on $\{\tau_\kappa>\tau_0\}$, we have $\tau_0= \tilde\tau_0<\tau_\kappa$ and $\Delta\nu_{\tau_\kappa}-\Delta\tilde{\nu}_{\tau_\kappa}\ge 0$. Note for future reference that since $0=X^\kappa_{\tau_\kappa}-\tilde X_{\tau_\kappa}\le |X^\kappa_{\tau_\kappa}-X^\nu_{\tau_\kappa}|$, on $\{\tilde \tau_0=\tau_\kappa<\tau_0\}$, and $0\le X^\kappa_{\tilde \tau_0}-\tilde X_{\tilde \tau_0}\le X^\kappa_{\tilde\tau_0}-X^\nu_{\tilde \tau_0}$, on $\{\tau_\kappa\ge \tau_0=\tilde\tau_0\}$, then
\begin{align}\label{eq:XX}
|X^\kappa_{s}-\tilde X_{s}|\le |X^\kappa_s- X^\nu_s|,\quad\text{for $s\in[\![0,\tilde\tau_0]\!]$.}
\end{align}

Given $\eta\in(0,1)$ we choose $\tau\in\cT_t$ such that $\sup_{\sigma\in\cT_t}\cJ_{t,x}^N(\tilde{\nu},\sigma)\le \cJ_{t,x}^N(\tilde{\nu},\tau)+\eta$. Since $\tilde\tau_0\le \tau_\kappa$, then 
\begin{align}\label{eq:limsupumuN1a1}
\begin{aligned}
&u^{N}_{\kappa}(t,x)\!-\!\overline{u}^N(t,x)\ge \cJ_{t,x}^{N,\kappa}(\nu,\tau\!\wedge\!\tilde\tau_0)\!-\!\cJ_{t,x}^{N}(\tilde{\nu},\tau)\!-\!2\eta\\
&=\E_x\Big[\e^{-r(\tau\wedge\tilde \tau_0)}\Big(g^N(t\!+\!\tau\wedge\tilde \tau_0,X_{\tau\wedge\tilde \tau_0}^{\kappa})\!-\!g^N(t\!+\!\tau\wedge\tilde \tau_0,\tilde X_{\tau\wedge\tilde \tau_0})\Big)\\ 
&\quad+\!\int_0^{\tau\wedge\tilde \tau_0}\!\!\e^{-rs}\big(h^N(t\!+\!s,X_s^{\kappa})\!-\!h^N(t\!+\!s,\tilde X_s)\big)\ud s \!+\!\int_{[0,\tau\wedge\tilde \tau_0]}\!\!\!\e^{-rs}\bar{\alpha}^N\,\big(\ud|\nu|_s-\ud|\tilde{\nu}|_s\big)\Big]\!-\!2\eta.
\end{aligned}
\end{align}
Recalling $X^\kappa_{\tau_\kappa}=0$, notice that
\begin{align}\label{eq:jump1}
\begin{aligned}
&\int_{[0,\tau\wedge\tilde \tau_0]}\!\!\!\e^{-rs}\bar{\alpha}^N\,\big(\ud|\nu|_s-\ud|\tilde{\nu}|_s\big)=\e^{-r(\tau\wedge\tilde \tau_0)}\bar{\alpha}^N\,\big(\Delta\nu_{\tau\wedge\tilde \tau_0}-\Delta\tilde{\nu}_{\tau\wedge\tilde \tau_0}\big)\\
&\ge -\mathds{1}_{\{\tau_\kappa=\tilde\tau_0\}\cap\{\tau\ge \tilde\tau_0\}}\e^{-r(\tau\wedge\tilde \tau_0)}\bar{\alpha}^N |X_{\tau_\kappa}^\nu|=-\mathds{1}_{\{\tau_\kappa=\tilde\tau_0\}\cap\{\tau\ge \tilde\tau_0\}}\e^{-r(\tau\wedge\tilde \tau_0)}\bar{\alpha}^N |X_{\tau_\kappa}^\nu-X_{\tau_\kappa}^\kappa|.
\end{aligned}
\end{align}
Plugging \eqref{eq:jump1} into \eqref{eq:limsupumuN1a1} and using the Lipschitz property of $g^N$ and $h^N$ yield
\begin{align}\label{eq:limsupumuN1}
\begin{aligned}
u^{N}_{\kappa}(t,x)\!-\!\overline{u}^N(t,x)&\ge-\bar \alpha^N\E_x\big[\big|X_{\tau\wedge\tilde\tau_0}^{\kappa}\!-\!\tilde X_{\tau\wedge\tilde\tau_0}\big|\!+\!|X_{\tau\wedge\tilde\tau_0}^{\kappa}\!-\! X_{\tau\wedge\tilde\tau_0}^\nu|\big]\\
&\quad-L_N\int_0^{T-t}\E_x\big[\big|X_{s\wedge\tau\wedge\tilde\tau_0}^{\kappa}\!-\!\tilde X_{s\wedge\tau\wedge\tilde\tau_0}\big|\big]\ud s\!-\!2\eta\\
&\ge -(2\bar{\alpha}^N\!+\!TL_N)\sup_{s\in[0,T-t]}\sup_{x\in U}\sup_{\nu\in\cB^\kappa_{t,x}}\E_x\Big[|X_{s\wedge \tau\wedge\tilde\tau_0}^{\kappa}\!-\! X^\nu_{s\wedge \tau\wedge\tilde\tau_0}|\Big]\!-\!2\eta,
\end{aligned}
\end{align}
where the last inequality holds by \eqref{eq:XX} with $\cB^\kappa_{t,x}$ as in Lemma \ref{lem:convkappa}. 
Using Lemma \ref{lem:convkappa} and recalling that $\eta$ is arbitrary we get
\eqref{eq:limsupumuN2}.

Now we prove
\begin{align}\label{eq:limsupumuN0}
\qquad\limsup_{\kappa\to0}\sup_{(t,x)\in[0,T]\times U}\big(u^{N}_{\kappa}(t,x)- \underline{u}^N(t,x)\big)\le 0.
\end{align}
For $\nu\in\cA_{t,x}^-$, let $X^{\nu}=X^{\nu;x}$, $\tau_0\coloneqq\tau_0(t,x;\nu)$, $X^{\kappa}=X^{\nu,\kappa;x}$ and $\tau_\kappa=\rho(t,x;\nu,\kappa)$. Define $\tilde{\nu}$ as in \eqref{eq:nuadmissible} but replacing $(\tau_\kappa,\tau_0,X^{\nu})$ therein with $(\tau_0,\tau_\kappa,X^{\kappa})$ (mind the order), so that $\tilde{\nu}\in\cA_{t,x}^{\kappa,-}$. Finally, set $\tilde{X}^\kappa= X^{\tilde{\nu},\kappa;x}$ and $\tilde{\tau}_\kappa=\rho(t,x;\tilde{\nu},\kappa)$. Notice that $\tilde{\tau}_\kappa\le \tau_0$, $\P$-a.s. for any $\nu\in\cA_{t,x}^-$, by analogous arguments to those following \eqref{eq:nuadmissible}. Moreover, on $\{\tau_0\le \tau_\kappa\}$ we have $\Delta\tilde\nu_{\tau_0}-\Delta\nu_{\tau_0}= -X^\kappa_{\tau_0}$ and $\tilde\tau_\kappa=\tau_0$. Instead, on $\{\tau_\kappa< \tau_0\}$ we have $\Delta\tilde\nu_{\tau_\kappa}\ge \Delta\nu_{\tau_\kappa}$ and $\tilde\tau_\kappa=\tau_\kappa$. Finally, $|X^\nu_{s}-\tilde X^\kappa_{s}|\le |X^\nu_{s}-X^\kappa_{s}|$ for $s\in[\![0,\tilde\tau_\kappa]\!]$, because $0=X^\nu_{\tau_0}-\tilde X^\kappa_{\tau_0}\le |X^\nu_{\tau_0}-X^\kappa_{\tau_0}|$, on $\{\tilde \tau_\kappa=\tau_0\le \tau_\kappa\}$, and $0\le X^\nu_{\tilde \tau_\kappa}-\tilde X^\kappa_{\tilde \tau_\kappa}\le X^\nu_{\tilde\tau_\kappa}-X^\kappa_{\tilde \tau_\kappa}$, on $\{\tau_0> \tau_\kappa=\tilde\tau_\kappa\}$.

Let $\tau\in\cT_t$ be optimal for $u^{N}_{\kappa}(t,x)$ and let $\nu\in\cA_{t,x}^-$ be such that $\inf_{\nu'\in\cA_{t,x}}\cJ_{t,x}^{N}(\nu',\tau\wedge\tilde{\tau}_\kappa)\ge \cJ_{t,x}^{N}(\nu,\tau\wedge\tilde{\tau}_\kappa)-\eta$ for $\eta\in(0,1)$. Then, by analogous estimates to those in \eqref{eq:limsupumuN1a1}--\eqref{eq:limsupumuN1} we arrive at 
\begin{align}\label{eq:limsupT}
\begin{aligned}
u^{N}_{\kappa}(t,x)\!-\!\underline{u}^N(t,x)&\le\cJ_{t,x}^{N,\kappa}(\tilde{\nu},\tau)\!-\!\cJ_{t,x}^{N}(\nu,\tau\!\wedge\!\tilde{\tau}_\kappa)+\eta\\
&\le (2\bar{\alpha}^N\!+\!TL_N)\sup_{s\in[0, T-t]}\sup_{x\in U}\sup_{\nu\in\cB^\kappa_{t,x}}\E_x\Big[|X_{s\wedge\tau\wedge\tau_\kappa\wedge\tau_0}^{\nu,\kappa}\!-\!X^\nu_{s\wedge\tau\wedge\tau_\kappa\wedge\tau_0}|\Big]+\eta,
\end{aligned}
\end{align}
where we also use that $\tilde\tau_\kappa\wedge\tau_0=\tau_\kappa\wedge\tau_0$. Then, Lemma \ref{lem:convkappa} yields \eqref{eq:limsupumuN0} by arbitrariness of $\eta$. Combining \eqref{eq:limsupumuN2} and \eqref{eq:limsupumuN0} we conclude. We deduce $u^N\in C(\overline \cO)$ as (locally) uniform limit of continuous functions.
\end{proof}

It remains to remove the hypothesis of boundedness of functions $g^N$ and $h^N$. 
Noticing that Lemma \ref{lem:nu^-} holds independently of $\kappa$, we have that the same result holds for $u^{N}$.

\begin{corollary}\label{cor:nu^-}
We have $u^N(t,x)=\displaystyle\inf_{\nu\in\cA_{t,x}^-}\sup_{\tau\in\cT_t} \cJ_{t,x}^N(\nu,\tau)=\sup_{\tau\in\cT_t}\inf_{\nu\in\cA_{t,x}^-}\cJ_{t,x}^N(\nu,\tau)$, for all $(t,x)\in\overline\cO$.
\end{corollary}

Since the bounds obtained in Lemma \ref{lem:W12pbound} are independent of $N$, then $(u^{N})_{N\in\N}$ is bounded in $W^{1,2;p}_{\ell oc}$. We define $u\coloneqq\lim_{n\to\infty}u^{N_n}$ along a sequence $(N_n)_{n\in\N}$ increasing to $+\infty$, weakly in $W^{1,2;p}(\cK)$ for any compact $\cK\subset\cO$ and strongly in $L^\infty(\cK)$ for any compact $\cK\subset\overline\cO$ (as in the proof of Theorem \ref{thm:highreguued}). Then, 
\begin{align}\label{eq:penalties}
u(t,x)\ge g(t,x)\quad\text{and}\quad |\partial_x u(t,x)|\le \bar{\alpha},\quad\text{for all $(t,x)\in\cO$},
\end{align}
because $u^N_\kappa$ solves Problem \ref{prb:varineq} (cf.\ Remark \ref{rem:thm}).

\begin{theorem}\label{thm:game}
We have $u\in C(\overline\cO)$. Moreover, $u=\underline v=\overline v$ on $\overline\cO$ for $\underline v$ and $\overline v$ as in \eqref{eq:lowuppvfnc}.
\end{theorem}

\begin{proof}
Let $U\subset[0,\infty)$ be compact.
First we prove $\limsup_{N\to\infty}\sup_{(t,x)\in[0,T]\times U} (u^N- \underline{v})(t,x)\le 0$. For that we recall that we can restrict our attention to $\nu\in\cA_{t,x}^{opt}$ by Lemma \ref{lem:Aopt}. Let $\tau\in\cT_t$ be optimal for $u^N(t,x)$ and let $\nu\in\cA_{t,x}^{opt}$ be such that $\inf_{\nu'\in\cA_{t,x}^{opt}}\cJ_{t,x}(\nu',\tau)\ge \cJ_{t,x}(\nu,\tau)-\eta$. Notice that $\tau$ and $\nu$ depend on $N$. Assuming with no loss of generality that $0\le \bar{\alpha}^N-\bar{\alpha}\le \frac{1}{N}$ and recalling that $g^N\le g$ and $h^N\le h$, we obtain
\begin{align*}
u^N(t,x)-\underline{v}(t,x)&\le \cJ^N_{t,x}(\nu,\tau)-\cJ_{t,x}(\nu,\tau)+\eta\\
&=\E_x\Big[\e^{-r(\tau\wedge\tau_0)}\big(g^N(t+\tau\wedge\tau_0,X_{\tau\wedge\tau_0}^\nu)-g(t+\tau\wedge\tau_0,X_{\tau\wedge\tau_0}^\nu)\big)\\
&\qquad\quad+\int_0^{\tau\wedge\tau_0}\e^{-rs}\big(h^N(t+s,X_s^\nu)-h(t+s,X_s^\nu)\big)\,\ud s\\
&\qquad\quad+\int_{[0,\tau\wedge\tau_0]}\e^{-rs}(\bar{\alpha}^N-\bar{\alpha})\ud |\nu|_s\Big]+\eta\le \frac{1}{N}\E\big[|\nu|_{T-t}\big]+\eta.
\end{align*}
Using that $\nu\in\cA_{t,x}^{opt}$, we have $u^N(t,x)-\underline{v}(t,x)\le \frac{K_5}{N}(1+x^2)+\eta$. Taking supremum over $[0,T]\times U$ and limits we obtain our first claim.

Now, we show $\liminf_{N\to\infty}\inf_{(t,x)\in[0,T]\times U}(u^N- \overline{v})(t,x)\ge 0$, which will conclude the proof.
Let $\nu\in\cA_{t,x}^-$ be $\eta$-optimal for $u^N(t,x)$ and let $\tau\in\cT_t$ be $\eta$-optimal for $\sup_{\tau\in\cT_t}\cJ_{t,x}(\nu,\tau)$ (notice that $\tau$ and $\nu$ depend on $N$). Set 
$\rho_N\coloneqq\inf\{s\ge 0\,|\,(t+s,X_s^\nu)\notin A^g_{N-1}\cap A^h_{N-1}\}$ and recall that $g^N=g$ on $A^g_{N-1}$ and $h^N=h$ on $A^h_{N-1}$. Then, by linear growth of $g$ and quadratic growth of $h$ we obtain
\begin{align*}
u^N(t,x)-\overline{v}(t,x)&\ge\cJ_{t,x}^{N}(\nu,\tau)-\cJ_{t,x}(\nu,\tau)-2\eta\\
&=\E_x\Big[\e^{-r(\tau\wedge\tau_0)}\big(g^N(t+\tau\wedge\tau_0,X_{\tau\wedge\tau_0}^\nu)-g(t+\tau\wedge\tau_0,X_{\tau\wedge\tau_0}^\nu)\big)\\
&\qquad\quad+\int_0^{\tau\wedge\tau_0}\e^{-rs}\big(h^N(t+s,X_s^\nu)-h(t+s,X_s^\nu)\big)\,\ud s\Big]-2\eta\\
&\ge-K_1(1+T) \E_x\Big[\Big(1+\sup_{0\leq s\le T-t}|X^\nu_s|^2\Big)\mathds{1}_{\{\rho_N<\tau_0\}}\Big]-2\eta,
\end{align*}
with $K_1$ from Assumption \ref{ass:gen2}(iii).
Since $\nu\in\cA_{t,x}^-$, then $0\le X^{\nu}_s\le X_s^0$ for all $s\in[\![0,\tau_0]\!]$, $\P$-a.s., where $X^0$ is the uncontrolled process. Since $x\mapsto g(t,x)$ and $x\mapsto h(t,x)$ are non-decreasing we also have $\rho_N\ge \sigma_N$ where $\sigma_N\coloneqq\inf\{s\ge 0\,|\,(t+s,X_s^0)\notin A^g_{N-1}\cap A^h_{N-1}\}$. By monotonicity of $g$ and $h$ we can also assume $\{\sigma_N\le \tau_0\}\subseteq\{\bar X^0_{\tau_0}\ge M_N\}$ for suitable $M_N\ge 0$ ($\lim_{N\to\infty}M_N=\infty$) and $\bar X^0_s\coloneqq\sup_{0\le r\le s}X^0_r$.
Therefore, we have
\begin{align*}
u^N(t,x)-\overline{v}(t,x)\ge-K_1 \sup_{x\in\cK}\E_x\Big[\Big(1+|\bar X^0_T|^4\Big)\Big]^\frac12 \P_x\big(\bar X^0_{\tau_0}\ge M_N\big)^\frac12-2\eta.
\end{align*}
Classical estimates for SDEs guarantee $\sup_{x\in U}\E_x[|\bar X^0_T|^4]<\infty$ and therefore $\P_x\big(\bar X^0_{\tau_0}\ge M_N\big)\to 0$ as $N\to\infty$, uniformly on $U$. Thus, $\liminf_{n\to\infty}\inf_{(t,x)\in[0,T]\times U} (u^N-\overline{v})(t,x)\ge -\eta$. By arbitrariness of $\eta$, we conclude. Clearly $u\in C(\overline\cO)$ as (locally) uniform limit of continuous functions.
\end{proof}

The proof of Lemma \ref{lem:nu^-} can be repeated verbatim for $u$ replacing $u^N_\kappa$. Hence, the corollary holds: 
\begin{corollary}\label{cor:nu^-2}
We have $v(t,x)=\displaystyle\inf_{\nu\in\cA_{t,x}^-}\sup_{\tau\in\cT_t}\cJ_{t,x}(\nu,\tau)=\sup_{\tau\in\cT_t}\inf_{\nu\in\cA_{t,x}^-}\cJ_{t,x}(\nu,\tau)$, for any $(t,x)\in\overline\cO$.
\end{corollary}

\begin{remark}
Assuming also $|g(t,x)\!-\!g(s,x)|\!+\!|h(t,x)\!-\!h(s,x)|\le c(1\!+\!|x|)|t\!-\!s|$, for some $c>0$, 
it is indeed possible to show that $v$ is locally Lipschitz on $\overline\cO$. Since $|\partial_x v|\le \bar\alpha$ on $\cO$ by the variational inequality and $0\le v(t,x)\le \bar \alpha x$, it only remains to use arguments as those employed several times in the paper for direct comparisons for $v(t,x)-v(s,x)$. This result is not part of Theorem \ref{thm:usolvar}, we avoid further lengthy calculations. 
\end{remark}

\begin{proposition}\label{thm:limfuncum} 
The function $v$ is the maximal solution of Problem \ref{prb:varineq}.
\end{proposition}

\begin{proof}
Thanks to Lemmas \ref{lem:bndobstpen2} and \ref{lem:W12pbound}, and using the same arguments of proof as in \cite[Thm.\ 5]{bovo2022variational}, when passing to the limit along a suitable sequence $\hat u_j\coloneqq u^{N_j,\eps_j,\delta_j}_{\kappa_j}$ we obtain that the limit function $u$
is a solution of Problem \ref{prb:varineq} (cf.\ also \eqref{eq:penalties}). In Theorem \ref{thm:game}, we showed that $u$ is equal to the value of the game \eqref{eq:valfunc}.
Finally, we can repeat verbatim the proof from \cite[Thm.\ 6]{bovo2022variational} (cf.\ p.\ 32 therein) to obtain maximality of the solution.
\end{proof}

Collecting the results from this section we have proven the first part of Theorem \ref{thm:usolvar} under {\bf A.2}.

\subsection{An optimal stopping strategy}

The last section of the paper is devoted to finding an optimal strategy for the stopper under {\bf A.2} of Theorem \ref{thm:usolvar}. We additionally assume that the diffusion coefficient $\sigma$ be Lipschitz with constant $D_1$. The requirement is needed to apply \cite[Thm.\ 2.5.9]{krylov1980controlled} so that for a constant $C>0$ independent of $\kappa\in(0,1)$ we have
\begin{align*}
\E_x\Big[\sup_{s\in[0,T-t]}\!\big|X^{\nu,\kappa}_{s\wedge\tau_0\wedge\tau_\kappa}\!-\!X^{\nu,0}_{s\wedge\tau_0\wedge\tau_\kappa}\big|^2\Big]\!\le\! C\E_x\Big[\int_0^{\tau_0}\!\!\!\Big(\big|\mu(X^{\nu,0}_s)\!-\!\mu_\kappa(X^{\nu,0}_s)\big|^2\!+\!\big|\sigma(X^{\nu,0}_s)\!-\!\sigma_\kappa(X^{\nu,0}_s)\big|^2\Big)\ud s\Big].
\end{align*}
As in the proof of Lemma \ref{lem:convkappa} we can assume $\mu_\kappa(x)=\mu(x)$ and $\sigma_\kappa(x)=\sigma(x)+\kappa$ for $x\in[0,\kappa^{-1}]$. Now we restrict our attention to $\nu\in\cB^\kappa_{t,x}$ (cf.\ Lemma \ref{lem:convkappa}) and $x\in U$ with $U\subset[0,\infty)$ compact. Letting $\bar X^\nu_s\coloneqq \sup_{0\le r\le s}|X^{\nu,0}_r|$ and $\bar X^0_s\coloneqq \sup_{0\le r\le s}|X^{0,0}_r|$ and continuing from the estimate above we have
\begin{align}\label{eq:Xkappato01}
\begin{aligned}
&\lim_{\kappa\to 0}\sup_{(t,x)\in[0,T]\times U}\sup_{\nu\in\cB^\kappa_{t,x}}\E_x\Big[\sup_{s\in[0,T-t]}\big|X^{\nu,\kappa}_{s\wedge\tau_0\wedge\tau_\kappa}-X^{\nu,0}_{s\wedge\tau_0\wedge\tau_\kappa}\big|^2\Big]\\
&\le C\lim_{\kappa\to 0}\sup_{(t,x)\in[0,T]\times U}\sup_{\nu\in\cB^\kappa_{t,x}}\E_x\Big[\mathds{1}_{\{\bar X^\nu_{\tau_0}\le\kappa^{-1}\}}T\kappa^2+\mathds{1}_{\{\bar X^\nu_{\tau_0}>\kappa^{-1}\}}\int_0^{\tau_0}\!\!\!\Big(1+\big|X^{\nu,0}_s\big|^2\Big)\ud s\Big]\\
&\le C\lim_{\kappa\to 0}\Big(\kappa^2 +\sup_{x\in U}\P_x\big(\bar X^0_{T}>\kappa^{-1}\big)^\frac12\E_x\big[\big(1+\big|\bar X^{0}_{T}\big|^4\big)\big]^\frac12\Big)=0,
\end{aligned}
\end{align}
where $C>0$ may vary from line to line and we used $0\le X^{\nu,0}_s\le X^{0,0}_s$ for all $s\in[\![0,\tau_0]\!]$. The limit holds because $\P_x(\bar X^0_{T}>\kappa^{-1})\to 0$ as $\kappa\to 0$, uniformly for $x\in U$.

In order to state the next result we need to introduce some notation. Fix $(t,x)\in\overline\cO$ and $\nu\in\cA_{t,x}^-$. Let $\tau_\kappa=\rho(t,x;\nu,\kappa)$ be as in \eqref{eq:rhoknu}. For each $\kappa\in(0,1)$ we define $\nu^\kappa\in\cA_{t,x}^{\kappa,-}$ as 
\begin{align}\label{eq:nukap}
\nu_s^\kappa= \nu_s\mathds{1}_{[0, \tau_\kappa)}(s)\!+\!(\nu_{\tau_\kappa-}\!-\!X^{\nu,\kappa}_{\tau_\kappa-})\mathds{1}_{[\tau_\kappa,T-t]}(s).
\end{align}
Then $\rho(t,x;\nu^\kappa,\kappa)=\tau_\kappa$, $\P$-a.s., and $\big|X^{\nu^\kappa,\kappa}_{s}-X^{\nu,0}_{s}\big|\le \big|X^{\nu,\kappa}_{s}-X^{\nu,0}_{s}\big|$ for all $s\in[\![0,\tau_0\wedge\tau_\kappa]\!]$. It follows 
\begin{align}\label{eq:Xkappato02}
\begin{aligned}
&\lim_{\kappa\to 0}\sup_{x\in U}\E_x\Big[\sup_{s\in[0,T-t]}\big|X^{\nu^\kappa,\kappa}_{s\wedge\tau_0\wedge\tau_\kappa}-X^{\nu,0}_{s\wedge\tau_0\wedge\tau_\kappa}\big|^2\Big]\\
&\le \lim_{\kappa\to 0}\sup_{x\in U}\E_x\Big[\sup_{s\in[0,T-t]}\big|X^{\nu,\kappa}_{s\wedge\tau_0\wedge\tau_\kappa}-X^{\nu,0}_{s\wedge\tau_0\wedge\tau_\kappa}\big|^2\Big]=0,
\end{aligned}
\end{align}
by \eqref{eq:Xkappato01}. Combining the above with \eqref{eq:limsupumuN2} and \eqref{eq:limsupumuN0} we can select a sequence $(\kappa_j)_{j\in\N}$ such that 
\begin{align}\label{eq:uniflim}
\lim_{j\to\infty}\sup_{s\in[0,\tau_0\wedge\tau_{\kappa_j}]}\big(\big|Z^{\nu^{\kappa_j},\kappa_j}_{s}-Z^{\nu}_{s}\big|+\big|Z^{\nu^{\kappa_j},\kappa_j}_{s-}-Z^{\nu}_{s-}\big|\big)=0,\quad\P_x-a.s.,
\end{align}
where $Z^{\nu,\kappa}_s\coloneqq (u^N_\kappa-g^N)(t+s,X^{\nu,\kappa}_s)$ and $Z^{\nu}_s\coloneqq (u^N-g^N)(t+s,X^{\nu,0}_s)$ (see details in \cite[Lemma 4.6]{bovo2023b}). We omit the proof of the next two lemmas, because it is a repetition of the proof of \cite[Lemma 4.6]{bovo2023b} and it is based on \eqref{eq:uniflim}.

\begin{lemma}\label{lem:stop_conv}
Recall $\sigma_*^{N,\kappa}$ and $\tau_*^{N,\kappa}$ from \eqref{eq:tau*Nkappa}. Fix $(t,x)\in\cO$, $\nu\in\cA^-_{t,x}$ and $\nu^\kappa\in\cA^{\kappa,-}_{t,x}$ introduced above. There is a sequence $(\kappa_j)_{j\in\N}$ converging to zero and such that 
\begin{align*}
\liminf_{j\to\infty}\big(\sigma_*^{N,\kappa_j}\wedge\tau_*^{N,\kappa_j}\big)(t,x;\nu^{\kappa_j})\ge (\sigma^N_*\wedge\tau^N_*)(t,x;\nu),\quad\text{$\P_x$-a.s.},
\end{align*}
where
\begin{align}
\sigma_*^N&=\sigma_*^N(t,x;\nu)=\inf\big\{s\ge 0\,\big|\,u^N(t+s,X_{s-}^{\nu;x})=g^N(t+s,X_{s-}^{\nu;x})\big\},\\
\tau_*^N&=\tau_*^N(t,x;\nu)=\inf\big\{s\ge 0\,\big|\,u^N(t+s,X_s^{\nu;x})=g^N(t+s,X_s^{\nu;x})\big\}.
\end{align}

\end{lemma}
Since $u^N\to v$ uniformly on compacts by Theorem \ref{thm:game} and the dynamics is not affected by the choice of $N$, the same arguments as in \cite[Lemma 4.3]{bovo2023}.
\begin{lemma}\label{lem:stop_conv2}
Recall $\sigma_*$ and $\tau_*$ from \eqref{eq:taustar}.
For $(t,x)\in\cO$ and $\nu\in\cA^-_{t,x}$ we have 
\begin{align*}
\liminf_{N\to\infty}\big(\sigma_*^{N}\wedge\tau_*^{N}\big)(t,x;\nu)\ge (\sigma_*\wedge\tau_*)(t,x;\nu),\quad\text{$\P_x$-a.s.}
\end{align*}
\end{lemma}

Next we show optimality of $\sigma_*\wedge\tau_*$, thus concluding the proof of Theorem \ref{thm:usolvar} under {\bf A.2}. 
\begin{proposition}
Let $(t,x)\in\overline\cO$. We have $v(t,x)\le \cJ_{t,x}(\nu,\sigma_*(\nu)\wedge\tau_*(\nu))$ for any $\nu\in\cA^-_{t,x}$. Then Corollary \ref{cor:nu^-2} implies 
\begin{align*}
v(t,x)=\inf_{\nu\in\cA_{t,x}} \cJ_{t,x}(\nu,\sigma_*(\nu)\wedge\tau_*(\nu)).
\end{align*}
\end{proposition}
\begin{proof}
Fix $(t,x)\in\overline\cO$ and take $\nu\in\cA_{t,x}^-$ and $\nu^\kappa\in\cA_{t,x}^{\kappa,-}$ as in \eqref{eq:nukap}. Letting $\rho_*^{N}\coloneqq \sigma_*^{N}\wedge\tau_*^N\wedge\sigma_*\wedge\tau_*$ and $\rho_*^{N,\kappa}\coloneqq \sigma_*^{N,\kappa}\wedge\tau_*^{N,\kappa}\wedge\rho_*^{N}$, we have that \eqref{eq:uNkopt} still holds with $\rho_*^{N,\kappa}$ in place of $\sigma_*^{N,\kappa}\wedge\tau_*^{N,\kappa}$. Noticing that $|\partial_xu^{N}_\kappa(s,y)| \le \bar{\alpha}^N$ holds for all $(s,y)\in\overline{\cO}$, we obtain
\begin{align*}
\begin{aligned}
u^{N}_\kappa(t,x)\le\E_x\Big[&\e^{-r\rho_*^{N,\kappa}}\!u^N_\kappa\big(t\!+\!\rho_*^{N,\kappa},X^{\nu^\kappa,\kappa}_{\rho_*^{N,\kappa}-}\big)\!+\!\int_0^{\rho_*^{N,\kappa}}\!\!\!\e^{-rs}h^N(t\!+\!s,X_s^{\nu^\kappa,\kappa})\ud s\!+\!\int_{[0,\rho_*^{N,\kappa})}\!\!\e^{-r s}\bar{\alpha}^N \ud |\nu^\kappa|_s\Big].
\end{aligned}
\end{align*}
Recall that $u^N_\kappa$ and $h^N$ are Lipschitz in the space variable with constant $\bar{\alpha}^N$ and $L_N$, respectively. Notice that $|\nu|_s=|\nu^\kappa|_s$ for $s\in[\![0,\tau_\kappa)\!)$, with $\tau_\kappa$ as in \eqref{eq:nukap}, and $\tau_\kappa\ge \tau^{N,\kappa}_*$ because $u^N_\kappa(t,0)=g^N(t,0)$. Likewise, $\tau_0\ge \tau^N_*$. Then, 
\begin{align*}
\begin{aligned}
u^{N}_\kappa(t,x)\le& (\bar{\alpha}^N+L_N T)\E_x\Big[\sup_{s\in[0,T-t]}\big|X^{\nu^\kappa,\kappa}_{s\wedge\tau_0\wedge\tau_\kappa}-X^{\nu,0}_{s\wedge\tau_0\wedge\tau_\kappa}\big|\Big]\\
&+\E_x\Big[\e^{-r\rho_*^{N,\kappa}}\!u^N_\kappa\big(t\!+\!\rho_*^{N,\kappa},X^\nu_{\rho_*^{N,\kappa}-}\big)\!+\!\int_0^{\rho_*^{N,\kappa}}\!\!\!\e^{-rs}h^N(t\!+\!s,X_s^\nu)\ud s\!+\!\int_{[0,\rho_*^{N,\kappa})}\!\!\e^{-r s}\bar{\alpha}^N \ud |\nu|_s\Big].
\end{aligned}
\end{align*}
Letting $\kappa\downarrow0$ along the sequence from Lemma \ref{lem:stop_conv} and noticing that $\lim_{\kappa\to0}\rho_*^{N,\kappa}=\rho_*^{N}$ from below, we have for a.e. $\omega\in\Omega$
\begin{align*}
\lim_{\kappa\downarrow0} X^{\nu}_{\rho_*^{N,\kappa}-}=X^{\nu}_{\rho_*^N-}\quad\text{and}\quad\lim_{\kappa\downarrow0}\int_{[0,\rho_*^{N,\kappa})}\!\!\e^{-r s}\bar{\alpha}^N \ud \nu_s=\int_{[0,\rho_*^{N})}\!\!\e^{-r s}\bar{\alpha}^N \ud \nu_s.
\end{align*}
Since $u^N$, $g^N$ and $h^N$ are bounded, dominated convergence and \eqref{eq:Xkappato02} yields
\begin{align*}
\begin{aligned}
u^{N}(t,x)&\le\E_x\Big[\e^{-r\rho_*^{N}}\!u^N\big(t\!+\!\rho_*^{N},X^\nu_{\rho_*^{N}-}\big)\!+\!\int_0^{\rho_*^{N}}\!\!\!\e^{-rs}h^N(t\!+\!s,X_s^\nu)\ud s\!+\!\int_{[0,\rho_*^{N})}\!\!\e^{-r s}\bar{\alpha}^N \ud \nu_s\Big].
\end{aligned}
\end{align*}
We use similar arguments to pass to the limit as $N\to\infty$. Notice that $\lim_{N\to\infty}\rho_*^{N}=\sigma_*\wedge\tau_*$ (Lemma \ref{lem:stop_conv2}), $\bar{\alpha}^N\downarrow \bar{\alpha}$, $u^N\to v$ uniformly on compacts (Lemma \ref{thm:game}) and $h^N\uparrow h$. Moreover, $u^N$ and $h^N$ have at most quadratic growth uniformly in $N$ and $\nu$ is fixed. Then, dominated convergence yields
\begin{align}\label{eq:final}
\begin{aligned}
\quad v(t,x)\le\E_x\Big[\e^{-r(\sigma_*\wedge\tau_*)}v\big(t\!+\!\sigma_*\!\wedge\!\tau_*,X^\nu_{\sigma_*\wedge\tau_*-}\big)\!+\!\int_0^{\sigma_*\wedge\tau_*}\!\!\!\!\!\e^{-rs}h(t\!+\!s,X_s^\nu)\ud s\!+\!\int_{[0,\sigma_*\wedge\tau_*)}\!\!\!\!\!\e^{-r s}\bar{\alpha} \ud \nu_s\Big].
\end{aligned}
\end{align}
On $\{\sigma_*<\tau_*\}$, we use that $|\partial_x g|\le \bar{\alpha}$ and thus
\begin{align*}
v\big(t\!+\!\sigma_*\!\wedge\!\tau_*,X^\nu_{\sigma_*\wedge\tau_*-}\big)=g\big(t\!+\!\sigma_*,X^\nu_{\sigma_*-}\big)\le g\big(t\!+\!\sigma_*,X^\nu_{\sigma_*-}\big)+\bar{\alpha}\Delta\nu_{\sigma_*}.
\end{align*}
Similarly, on $\{\sigma_*\ge \tau_*\}$ we have 
\begin{align*}
v\big(t\!+\!\sigma_*\!\wedge\!\tau_*,X^\nu_{\sigma_*\wedge\tau_*-}\big)\le v\big(t\!+\!\tau_*,X^\nu_{\tau_*}\big)+\bar{\alpha}\Delta\nu_{\sigma_*}= g\big(t\!+\!\tau_*,X^\nu_{\tau_*}\big)+\bar{\alpha}\Delta\nu_{\sigma_*}.
\end{align*}
Plugging these two inequalities into the right-hand side of \eqref{eq:final} we obtain 
$v(t,x)\le \cJ_{t,x}(\nu,\sigma_*\wedge\tau_*)$.
By the arbitrariness of $\nu\in\cA^-_{t,x}$, the result holds.
\end{proof}

\appendix
\section{Auxiliary results}
\subsection{Construction of the function \texorpdfstring{$\tilde{g}_m^N$}{gNm}}\label{app:gtilde}

For the ease of the exposition, but with a small abuse of notation, we denote $\overline{\Theta}^{N}_{\kappa}$ by $\overline{\Theta}$. With no loss of generality we assume that $\partial_x g^N_m(T,\overline{\Theta})\neq 0$, as otherwise we do not need to construct $\tilde{g}^N_m$ for the proof of Lemma \ref{lem:lbw}. Similarly, we assume that $g^N_m(T,\overline{\Theta})> 0$ as otherwise $\overline{\Theta}$ would be a minimum of $g^N_m(T,\cdot)$ hence contradicting $\partial_x g^N_m(T,\overline{\Theta})\neq 0$ (recall that $g^N_m\ge 0$). 

We construct $\tilde{g}^N_m$ so that the three conditions in \eqref{eq:tildegmN} hold. Recalling the cut-off function $\xi_m$, we choose $\tilde{g}^N_m(t,x)=f(t)\phi(x)\xi_{m-1}(x)$, for $(t,x)\in[0,T]\times [0,m]$ with
\begin{align}\label{eq:phif}
\phi(x)=(x^2-2\overline{\Theta}x+\overline{\Theta}^2+c)\quad\text{and}\quad f(t)=\frac{g^N_m(t,\zeta_m(t))}{\phi(\zeta_m(t))\xi_{m-1}(\zeta_m(t))},
\end{align}
where $c>0$ is a constant chosen at the end. We emphasise that $\phi$, $f$ and $c$ depend on $N,\kappa,m$. It is immediate to check that $\phi(x)>0$ for all $x\in[0,m]$ with minimum $\phi(\overline \Theta)=c$. 

The first condition of \eqref{eq:tildegmN} holds by construction, i.e., $\tilde g^N_m(t,y)=g^N_m(t,y)$ for every $t\in[0,T]$ and $y\in\{\zeta_m(t),m\}$. The derivatives of $\tilde{g}$ read:
\begin{align}\label{eq:dergtilde}
\begin{aligned}
&\partial_t \tilde{g}^N_m(t,x)=\dot{f}(t)\phi(x)\xi_{m-1}(x),\\
&\partial_x\tilde{g}^N_m(t,x)=f(t)\dot{\phi}(x)\xi_{m-1}(x)+f(t)\phi(x)\dot{\xi}_{m-1}(x),\\
&\partial_{xx}\tilde{g}^N_m(t,x)=f(t)\ddot{\phi}(x)\xi_{m-1}(x)+2f(t)\dot{\phi}(x)\dot{\xi}_{m-1}(x)+f(t)\phi(x)\ddot{\xi}_{m-1}(x).
\end{aligned}
\end{align}
We notice that $\xi_{m-1}\equiv 1$ for $x\in[0,m-1]$, and so $\dot{\xi}_{m-1}(x)=\ddot{\xi}_{m-1}(x)=0$ for $x\le \Theta<m-1$. Therefore, the second condition in \eqref{eq:tildegmN}, i.e., $\partial_x \tilde g^N_m(T,\overline\Theta)=\partial_x \tilde g^N_m(T,m)=0$, holds because $\dot{\xi}_{m-1}(\overline{\Theta})=\dot{\phi}(\overline{\Theta})=0$. 

It remains to prove that $(\partial_t \tilde g^N_m+\cL_\kappa\tilde g^N_m-r\tilde g^N_m+h^N_m)(T,y)\ge 0$ for $y\in\{\overline\Theta,m\}$. The condition is trivially satisfied at $y=m$ because the cut-off function makes all the terms vanish. For the condition at $\overline \Theta$ we resort to explicit calculations.

From \eqref{eq:phif} and using $\xi_{m-1}(\zeta_m(T))=\xi_{m-1}(\overline \Theta)=1$ and $\dot{\xi}_{m-1}(\overline{\Theta})=\dot{\phi}(\overline{\Theta})=0$ it is easy to verify
\begin{align}\label{eq:dotf}
\dot{f}(T)&=\frac{\partial_t g^N_m(T,\overline{\Theta})+\partial_x g^N_m(T,\overline{\Theta})\dot{\zeta}_m(T)}{\phi(\overline{\Theta})}.
\end{align}
Therefore, using again $\xi_{m-1}(\overline \Theta)=1$, $\dot{\xi}_{m-1}(\overline{\Theta})=\ddot{\xi}_{m-1}(\overline{\Theta})=0$ and \eqref{eq:dergtilde} we obtain
\begin{align*}
\widetilde\Gamma(\overline\Theta)&\coloneqq(\partial_t+\cL_\kappa-r)\tilde{g}^N_m(T,\overline{\Theta})+h_m^N(T,\overline{\Theta})\\
&=\dot{f}(T)\phi(\overline{\Theta})+\frac{(\sigma_\kappa(\overline{\Theta}))^2}{2}f(T)\ddot{\phi}(\overline{\Theta})+\mu(\overline{\Theta})f(T)\dot{\phi}(\overline{\Theta})-rf(T)\phi(\overline{\Theta})+h_m^N(T,\overline{\Theta}).
\end{align*}
Now we observe that $\dot \phi(\overline\Theta)=0$ and $f(T)=g^N_m(T,\overline\Theta)/\phi(\overline\Theta)=g^N_m(T,\overline\Theta)/c$. That and \eqref{eq:dotf} yield
\begin{align*}
\widetilde\Gamma(\overline\Theta)&=\partial_t g^N_m(T,\overline{\Theta})+\partial_x g^N_m(T,\overline{\Theta})\dot{\zeta}_m(T)+\frac{(\sigma_\kappa(\overline{\Theta}))^2}{2}\frac{g^N_m(T,\overline{\Theta})}{c}\ddot{\phi}(\overline{\Theta})-rg^N_m(T,\overline{\Theta})+h_m^N(T,\overline{\Theta})\\
&=-\big(\cL_\kappa g^N_m\big)(T,\overline\Theta)+\partial_x g^N_m(T,\overline{\Theta})\dot{\zeta}_m(T)+\frac{(\sigma_\kappa(\overline{\Theta}))^2}{2 c}g^N_m(T,\overline{\Theta})\ddot{\phi}(\overline{\Theta}),
\end{align*}
where in the second equality we used $(\partial_tg^N_m-rg^N_m)(T,\overline{\Theta})+h_m^N(T,\overline{\Theta})=-(\cL_\kappa g^N_m)(T,\overline{\Theta})$ by \eqref{eq:overtheta}. 
Therefore, $\widetilde \Gamma(\overline\Theta)\ge 0$ if
\begin{align*}
\frac{\ddot{\phi}(\overline{\Theta})}{c}\ge \frac{2}{(\sigma_\kappa(\overline{\Theta}))^2g^N_m(T,\overline{\Theta})}\Big(\big(\cL_\kappa g^N_m\big)(T,\overline\Theta)-\partial_x g^N_m(T,\overline{\Theta})\dot{\zeta}_m(T)\Big).
\end{align*}
If the right-hand side above is negative, the inequality holds for any $c>0$ because $\ddot \phi(\overline\Theta)=2$. If instead the right-hand side above is positive, we choose 
\[
c=\frac{(\sigma_\kappa(\overline{\Theta}))^2g^N_m(T,\overline{\Theta})}{\big(\cL_\kappa g^N_m\big)(T,\overline\Theta)-\partial_x g^N_m(T,\overline{\Theta})\dot{\zeta}_m(T)}.
\]
This concludes the proof of the third condition in \eqref{eq:tildegmN}.
\hfill$\square$

\medskip
\noindent{\bf Funding}: Both authors received partial financial support from EU -- Next Generation EU -- PRIN2022 (2022BEMMLZ) CUP: D53D23005780006. T.\ De Angelis also received partial financial support from PRIN-PNRR2022 (P20224TM7Z) CUP: D53D23018780001. 

\bibliographystyle{plain}
\bibliography{Bibliography}

\end{document}